\documentclass{article}
\pdfoutput=1

\usepackage{fullpage}
\usepackage{import}
\usepackage{latexsym}
\usepackage{fixltx2e}   % many things, contains \textsubscript
\usepackage{amsfonts}
\usepackage{amsmath}
\usepackage{amssymb}
\usepackage{amsthm}
\usepackage{mathtools}
\usepackage{graphicx}  % for scalebox
\usepackage{bbm}
\usepackage{bm}
\usepackage{upgreek}
\usepackage{enumerate}
\usepackage{array} % e.g. for specifying column separation
\usepackage{calc} % e.g. for \widthof
\usepackage{etoolbox}
\robustify
\robustify\addtocounter % see http://tex.stackexchange.com/a/114770/42225
\robustify
\robustify
\usepackage{xspace}
\usepackage{float} % for [H] option of float placement

\usepackage[hypcap=true,justification=centering]{subcaption} % replacement for subfigure (obsolete) and subfig (out-of-date)
\usepackage{tikz}
\usetikzlibrary{calc,intersections,decorations.pathreplacing}
\usepackage{tikzscale}
\usepackage[toc,title]{appendix}
\usepackage[ruled,vlined]{algorithm2e} %for pseudocode
\usepackage{listings}
\usepackage{color}
\usepackage{xparse}

\usepackage{aliascnt} % necessary to fix interaction of \newtheorem and \autoref
\usepackage{hyperref}
\providecommand{\bbC}{\mathbb{C}}

\providecommand{\bbI}{\mathbb{I}}

\providecommand{\bbN}{\mathbb{N}}

\providecommand{\bbR}{\mathbb{R}}
\providecommand{\bbS}{\mathbb{S}}

\providecommand{\bbZ}{\mathbb{Z}}

\providecommand{\CB}{\mathcal{B}}
\providecommand{\CC}{\mathcal{C}}

\providecommand{\CE}{\mathcal{E}}
\providecommand{\CF}{\mathcal{F}}

\providecommand{\CH}{\mathcal{H}}

\providecommand{\CO}{\mathcal{O}}
\providecommand{\CP}{\mathcal{P}}

\newcommand{\VA}{{\mathbf{A}}}

\newcommand{\VF}{{\mathbf{F}}}

\newcommand{\VP}{{\mathbf{P}}}

\newcommand{\VW}{{\mathbf{W}}}

\newcommand{\Va}{{\mathbf{a}}}

\newcommand{\Vc}{{\mathbf{c}}}

\newcommand{\Vf}{{\mathbf{f}}}

\newcommand{\Vu}{{\mathbf{u}}}
\newcommand{\Vv}{{\mathbf{v}}}

\newcommand{\Vx}{{\mathbf{x}}}

\newcommand{\Vz}{{\mathbf{z}}}

\newcommand{\Vzero}{\boldsymbol{0}}

\newcommand*{\wt}[1]{\widetilde{#1}}
\newcommand*{\wh}[1]{\widehat{#1}}

\newcommand{\matlab}{{\sc Matlab}\xspace}

\newcommand{\eps}{\varepsilon}

\newcommand{\Hs}[1][\vec s]{H^{#1}} % has to be called with square brackets
\newcommand{\Hks}[1][\vec s]{H^{k+#1}} % has to be called with square brackets

\renewcommand{\d}{\mathop{}\!\dd} % \mathop for correct spacing
\newcommand{\dd}{\mathrm{d}}
\newcommand{\ee}{\mathrm{e}}
\newcommand{\ii}{\mathrm{i}}

\newcommand{\ind}{\mathbbm{1}}

\newcommand{\argmin}{\mathop{\mathrm{argmin}}}
\newcommand{\ran}{\mathrm{ran}}

\newcommand{\supp}{\mathrm{supp}\mathop{}}
\newcommand{\dist}{\mathrm{dist}}
\newcommand{\spann}{\operatorname{span}}

\newcommand{\bbSd}{{\mathbb{S}^{d-1}}}
\newcommand{\jl}{{j\!\?\?,\ell}}
\newcommand{\jld}{{j'\!\!,\?\?\ell'}}

\newcommand{\jlk}{{j\!\?\?,\?\ell\!\?\?,\?\vec k}}
\newcommand{\jlkd}{{j'\!\!,\?\?\ell'\!\!,\?\?\vec k'}}
\newcommand{\lld}{{\lambda,\lambda'}}
\newcommand{\jjd}{{j,j'}}
\newcommand{\mmd}{{m,m'}}
\newcommand{\Rs}{R_{\vec s}}

\newcommand{\Rjl}{R_\jl}
\newcommand{\Rjld}{R_\jld}
\newcommand{\tilRjl}{\wt R_\jl}
\newcommand{\tilRjld}{\wt R_\jld}
\newcommand{\sjl}{\vec s_\jl}
\newcommand{\sjld}{\vec s_\jld}

\newcommand{\rs}{\rho_{\vec s}}

\newcounter{universalcounter}[section]
\renewcommand{\theuniversalcounter}%
	{\arabic{section}.\arabic{universalcounter}}
\renewcommand{\theequation}%
	{\arabic{section}.\arabic{equation}}

\newcommand*{\mynewtheorem}[2]{
  \newaliascnt{#1}{universalcounter}
  \newtheorem{#1}[#1]{#2}
  \aliascntresetthe{#1}
  \expandafter\def\csname#1autorefname\endcsname{#2}
}

\theoremstyle{definition}
\mynewtheorem{definition} {Definition}
\mynewtheorem{example}    {Example}
\mynewtheorem{remark}     {Remark}
\theoremstyle{plain}
\mynewtheorem{theorem}    {Theorem}
\mynewtheorem{lemma}      {Lemma}
\mynewtheorem{corollary}  {Corollary}
\mynewtheorem{proposition}{Proposition}
\mynewtheorem{fact}       {Fact}
\mynewtheorem{assumption} {Assumption}

\NewDocumentCommand{\Dn} {m O{} m}  {\frac{\dd^{#1}#2}{\dd{#3}^{#1}}}
\NewDocumentCommand{\Dpn}{m O{} m}  {\frac{\partial^{#1}#2}{\partial{#3}^{#1}}}
\NewDocumentCommand{\Dpi}{m O{} m}  {\frac{\partial^{\abs{#1}}#2}{\partial{#3}^{#1}}}
\NewDocumentCommand{\Dpp}{m O{} m m}{\frac{\partial^{#1}#2}{\partial{#3}\,\partial{#4}}}
\NewDocumentCommand{\D}  {O{} m}    {\Dn{}[#1]{#2}}
\NewDocumentCommand{\DD} {O{} m}    {\Dn{2}[#1]{#2}}
\NewDocumentCommand{\Dp} {O{} m}    {\Dpn{}[#1]{#2}}
\NewDocumentCommand{\DDp}{O{} m}    {\Dpn{2}[#1]{#2}}

\DeclarePairedDelimiter {\nrmInternal}   {\lVert} {\rVert}
\DeclarePairedDelimiter {\absInternal}   {\lvert} {\rvert}
\DeclarePairedDelimiter {\parInternal}   {\lparen}{\rparen}
\DeclarePairedDelimiter {\braInternal}   {\lbrack}{\rbrack}
\DeclarePairedDelimiter {\ceiInternal}   {\lceil} {\rceil}
\DeclarePairedDelimiter {\flrInternal}   {\lfloor}{\rfloor}
\DeclarePairedDelimiter {\inpInternal}   {\langle}{\rangle}
\DeclarePairedDelimiter {\crlInternal}   {\{}     {\}}
\DeclarePairedDelimiterX{\setInternal}[2]{\{}     {\}}     {#1\colon#2} %{#1\,\delimsize\vert\,#2}

\NewDocumentCommand{\norm}   {s o m}  {\resizerOneInput {\nrmInternal}{#1}{#2}{#3}    }
\NewDocumentCommand{\abs}    {s o m}  {\resizerOneInput {\absInternal}{#1}{#2}{#3}    }
\NewDocumentCommand{\snorm}  {s o m}  {\resizerOneInput {\absInternal}{#1}{#2}{#3}    }
\NewDocumentCommand{\card}   {s o m}  {\resizerOneInput {\absInternal}{#1}{#2}{#3}    }
\NewDocumentCommand{\parens} {s o m}  {\resizerOneInput {\parInternal}{#1}{#2}{#3}    }
\NewDocumentCommand{\bracket}{s o m}  {\resizerOneInput {\braInternal}{#1}{#2}{#3}    }
\NewDocumentCommand{\ceil}   {s o m}  {\resizerOneInput {\ceiInternal}{#1}{#2}{#3}    }
\NewDocumentCommand{\floor}  {s o m}  {\resizerOneInput {\flrInternal}{#1}{#2}{#3}    }
\NewDocumentCommand{\inpr}   {s o m}  {\resizerOneInput {\inpInternal}{#1}{#2}{#3}    }
\NewDocumentCommand{\reg}    {s o m}  {\resizerOneInput {\inpInternal}{#1}{#2}{#3}    }
\NewDocumentCommand{\curly}  {s o m}  {\resizerOneInput {\crlInternal}{#1}{#2}{#3}    }
\NewDocumentCommand{\set}    {s o m m}{\resizerTwoInputs{\setInternal}{#1}{#2}{#3}{#4}}

\NewDocumentCommand{\resizerOneInput}{m m m m}{
	\IfBooleanTF{#2}   % star yes/no
		{#1[\big]{#4}} % star triggers \big
		{\IfNoValueTF{#3} % no star -> check if size specified
			{#1*{#4}}     % uses \left...\right
			{#1[#3]{#4}}} % uses size specified
}
\NewDocumentCommand{\resizerTwoInputs}{m m m m m}{
	\IfBooleanTF{#2}
		{#1[\big]{#4}{#5}}
		{\IfNoValueTF{#3}
			{#1*{#4}{#5}}
			{#1[#3]{#4}{#5}}}
}

\newcommand{\forceheight}[2][b]{\smash{#2}\vphantom{#1}} %for aligning... the letter b has good height for normal text

\newcommand{\phantomrel}{\mathrel{\phantom{=}}}

\newlength{\hspacetemp} % to be able to use \widthof in hspace, on has to set a length with setlength first, which has to be defined

\newcommand{\rcopywidth}[2]{%
	\mathrlap{#1}\phantom{\smash{#2}}}

\newcommand{\csetwidth}[2]{%
	\hspace{0.5\dimexpr#2}\mathclap{#1}\hspace{0.5\dimexpr#2}}

\let\oldleft\left
\let\oldright\right
\renewcommand{\left}{\mathopen{}\mathclose\bgroup\oldleft}
\renewcommand{\right}{\aftergroup\egroup\oldright}

\let\oldvec\vec
\def\vec#1{\oldvec{#1{}}}

\newcommand{\?}{\:\!} % this shifts by +1/18th em; \replaces the shortcut \. for \dot

\makeatletter
\g@addto@macro\@floatboxreset\centering
\makeatother

\mathtoolsset{showonlyrefs=true}

\DeclareCaptionLabelFormat{continued}{#1~#2 (cont.)}
\hypersetup{ % Print version
    colorlinks=true,
    linkcolor=black,
    citecolor=black,
    urlcolor=black
}
\numberwithin{equation}{section}
\numberwithin{figure}{section}
\numberwithin{table}{section}
\numberwithin{algocf}{section} % counter of algorithm2e package
\makeatletter
\let\c@algocf\c@universalcounter % point to register of universalcounter
\makeatother
\pdfpageattr{/Group <</S /Transparency /I true /CS /DeviceRGB>>} 
\usepgfmodule{nonlineartransformations}
\makeatletter
\def\logtransformation{%
\pgfmathveclen{\pgf@x}{\pgf@y}%
\pgfmathparse{\pgfmathresult/28.4}% factor accounts for pt2cm conversion
\pgfmathparse{\pgfmathresult > 2 ? \pgfmathresult : 1/4*\pgfmathresult^2+1}
\pgfmathparse{log2(\pgfmathresult)*5/6}%
\pgf@xa=\pgfmathresult pt%
\pgfmathatantwo{\pgf@y}{\pgf@x}%
\pgfmathsincos@{\pgfmathresult}%
\pgfmathmultiply{\pgf@xa}{\pgfmathresultx}\pgf@x=\pgfmathresult cm% USED CM AS UNITS TO SCALE UP !!!
\pgfmathmultiply{\pgf@xa}{\pgfmathresulty}\pgf@y=\pgfmathresult cm% USED CM AS UNITS TO SCALE UP !!!
}

\title{Optimal Adaptive Ridgelet Schemes for Linear Transport Equations}
\author{Philipp Grohs, Axel Obermeier \\ Seminar for Applied Mathematics, ETH Z\"urich \\ $\{\texttt{philipp.grohs,axel.obermeier}\}@\texttt{sam.math.ethz.ch}$}

\begin{document}

\maketitle

\begin{abstract}
	In this paper we present a novel method for the numerical solution of linear transport equations, which is based on ridgelets.
	Such equations arise for instance in radiative transfer or in phase contrast imaging.
	Due to the fact that ridgelet systems are well adapted to the structure of linear transport operators, it can be shown that our scheme operates in optimal complexity, even if line singularities are present in the solution.
	
	The key to this is showing that the system matrix (with diagonal preconditioning) is uniformly well-conditioned and compressible -- the proof for the latter represents the main part of the paper. We conclude with some numerical experiments about $N$-term approximations and how they are recovered by the solver, as well as localisation of singularities in the ridgelet frame.
\end{abstract}

% different tocdepth set below for introduction and the rest of the paper
\setcounter{tocdepth}{1}
\noappendicestocpagenum
\tableofcontents

\newpage
%\addtocontents{toc}{\setcounter{tocdepth}{1}} % don't show subsections of introduction
\section{Introduction}\label{sec:introduction}

In the past two decades, a wide range of multiscale systems have been introduced with lasting impact in many different fields, starting
%In the past two decades, Applied Harmonic Analysis has had a big impact on image processing, computer science and applied mathematics, primarily through the introduction of a wide range of multiscale systems, breaking ground
with wavelets \cite{wavelets} and continuing with ridgelets \cite{Can98}, curvelets \cite{Candes2005a,Candes2005b,CDDY06}, shearlets \cite{KuLaLiWe,KL12i}, contourlets \cite{DV05} etc. -- the latter three of which fall into the framework of so-called ``parabolic molecules'' \cite{par-mol}, while all of the mentioned systems are encompassed by the even broader framework of $\alpha$-molecules \cite{alpha-mol}.

These systems share the property that they are very well-adapted to representing certain classes of functions optimally (in the sense of the decay rate of the best $N$-term approximation) -- functions with point singularities for wavelets, line singularities for ridgelets and curved singularities for parabolic molecules. Since these classes make up the fundamental phenomenological features of most images in an extremely diverse set of applications, it is perhaps not surprising, that many of the above-mentioned systems were originally investigated in view of their properties regarding image processing.

With a certain time-lag, it is becoming apparent that these systems are also very suitable for solving partial differential equations -- again, wavelets were the first in this regard, for example leading to provably optimal solvers for elliptic equations \cite{cdd}. For differential equations with strong directional features -- such as transport equations -- it is intuitively clear that optimal solvers will need to take these features into account, however, the development of solvers based on directional systems is still in its infancy.

Following recent results \cite{grohs1}, that ridgelets permit the construction of simple diagonal preconditioners for linear transport equations which arise in collocation-type discretization methods for kinetic transport equations (such as radiative transport), we intend this paper (and its companion \cite{FFT-paper}) to be a first step towards establishing directional representation systems as a useful tool for solving PDEs.

%\subsection{Towards Establishing Directional Systems as a Tool in Numerical Analysis}

Perhaps the main reason for the success of wavelets in PDE solvers (which, as a long term goal, we would like to emulate) is that they do not only represent typical solutions efficiently, but -- crucially -- that they simultaneously sparsify (in a suitable sense) the resulting system matrices corresponding to the differential operator and achieve uniformly well-conditioned matrices with simple preconditioning.

The main focus of the present paper is to demonstrate that the same properties hold for ridgelets applied to the numerical discretisation of linear transport equations, and using the machinery of \cite{cdd} to show that this leads to solvers with optimal complexity.

\subsection{Radiative Transport Equation}

The motivation for this work is the numerical solution of the following model equation, described by the radiative transport equation (RTE),
\begin{equation}\label{eq:RTE}
	A u:=\vec s \cdot \nabla u + \kappa \, u = f + \int_{\bbS^{d-1}} \sigma u \, \dd \vec s'.
\end{equation}
It is a steady state continuity equation describing the conservation of radiative intensity in an absorbing, emitting and scattering medium, see e.g. \cite{modest2013radiative}. We will, however, not treat the scattering operator in this paper, which can be incorporated through a variety of methods, not the least of which -- the source iteration -- we implemented in \cite{FFT-paper}. Let us assume that the following quantities are known at all locations $\vec x \in \Omega \subset \bbR^d$ and for all directions $\vec s \in \bbS^{d-1} := \set{\vec s \in \bbR^d}{\norm{\vec s}_2 = 1}$:
\begin{itemize}
	\item absorption coefficient $\kappa\left(\vec x,\vec s\right) \ge \kappa_0 > 0$
	\item source term $f\left(\vec x,\vec s\right) \in \bbR$
\end{itemize} 
Then, the above equation allows us to find the unknown radiative intensity $u$ as a function $\Omega \times \bbS^{d-1} \to \bbR$.

Although the RTE looks simple, standard numerical techniques for solving it do not perform well for a number of reasons, mainly:
\begin{itemize}
	\item The transport term $s \cdot \nabla u$ leads to ill-conditioned systems of equations.
	\item Singularities in the input data may remain in the solution.
	\item With the dimension of the domain of $u$ being 3 in 2-dimensional physical space and 5 in 3-dimensional space, the problem is fairly high-dimensional. 
\end{itemize}
These issues make the accurate numerical solution of the RTE very costly or even impossible due to memory and compute power limitations of today's hardware. % In the course of the paper, we will show how to address both points by discretizing the RTE with the above-mentioned ridgelets.

\subsection{Ridgelets}

Our proposed approach to solving \eqref{eq:RTE}, while addressing the above-mentioned problems, is to discretise the equation in physical space using ridgelets. At a glance, a ridgelet is a function which is located along a line, orthogonal to which it oscillates heavily and along which it varies only little (see \autoref{fig:ridgelet:real} for an example). The idea is to build a basis (or rather, a frame) out of such ridgelets with varying locations, directions and widths, with which we can represent a function whose features are located along curves by a linear combination of relatively few of them. Solutions of the RTE typically fall into this category of functions that can be efficiently represented by such a system, as the variations along the transport direction are smoothed out while the ones orthogonal to it are not -- in particular, singularities in the input data may remain.

The present work provides a first step towards a ridgelet-based construction of an optimally convergent
numerical solver for \eqref{eq:RTE}. More precisely we consider the RTE for fixed directions $\vec s$
and show that our proposed scheme delivers optimal convergence rates for linear transport equations
\begin{equation}\label{eq:LinTrans}
	\vec s \cdot \nabla u(\vec x)  + \kappa(\vec x) u(\vec x)= f(\vec x).
\end{equation}
Since a number of numerical methods for the solution of \eqref{eq:RTE} heavily relies on 
efficient solvers of the above linear transport equation, the spatial discretization scheme
developed and analyzed in the present paper can be directly utilized for the numerical approximation
of solutions to the RTE -- as is done in \cite{FFT-paper}.

Before we describe our approach in more detail we would like to pause and comment on its novel properties and limitations.

The most important property, and the main result of this paper is the fact that our proposed algorithm
is able to approximate solutions $u$ of \eqref{eq:LinTrans} in optimal complexity. In this regard our results are very strong: complexity here is measured in terms of arithmetic operations to be carried
out by a processor and the solution is even allowed to possess singularities along lines. Moreover our
result hold uniformly in $\vec s$, meaning that they are independent of the transport direction. 
This property is of essential importance for solving the full RTE.

Even though the PDE \eqref{eq:LinTrans} is of admittedly simple form with several efficient methods
to solve it (cf. \cite{ern}) we are not aware of any method with such strong convergence results as is the case for our proposed scheme. For instance our method converges exponentially for solutions $u$ which are piecewise smooth with a line singularity (see \autoref{th:approx}) and this result holds uniformly for all directions $\vec s$. Such a result is far from true for conventional (eg. Finite-Element-based) discretization schemes where the expected convergence rate would be of order $N^{-\frac 12}$ instead, with $N$ being the number of arithmetic operations.

We consider the present paper as a first step in a larger programme of developing ridgelet-based solvers for the RTE. Therefore, in the following paragraphs we outline some limitations of the results as well as some promising directions for future work, opened up by our results.

The convergence results are confined to linear transport equations \eqref{eq:LinTrans}
and our analysis assumes that $\vec x$ belongs to the full space $\bbR^d$. The latter fact poses no problem if for instance the source term $f$ is compactly supported but in many applications one needs to restrict $\vec x$ to a finite domain $D\subset \bbR^d$ and impose inflow boundary conditions. The efficient incorporation of boundary conditions will require the construction of ridgelet frames on finite domains which is the subject of future work (to be more precise, incorporation of inflow boundary conditions is possible with the code developed in \cite{FFT-paper} but a rigorous analysis is still lacking). With such a construction at hand the theoretical analysis carried out in this paper would essentially go through also for finite domains.

With regard to the fact that the model equation \eqref{eq:LinTrans} addressed in this paper is far simpler than the full radiative transport equation we would like to mention that the paper \cite{FFT-paper} combines a ridgelet solver in space with a sparse collocation method to solve the full RTE efficiently. There, a key feature of the use of ridgelets is that collocation in angle leads to uniformly well-conditioned linear systems to be solved, independent of the spatial resolution -- a key property for efficient parallelisation. It is possible to go further by combining the spatial ridgelet discretisation as developed in the present paper with a wavelet discretisation on the sphere by a tensor product construction to develop an adaptive numerical algorithm for the full RTE. Again, this is the subject of future work.

\begin{figure}
	\begin{subfigure}{0.49\textwidth}
		\includegraphics[width=\textwidth]{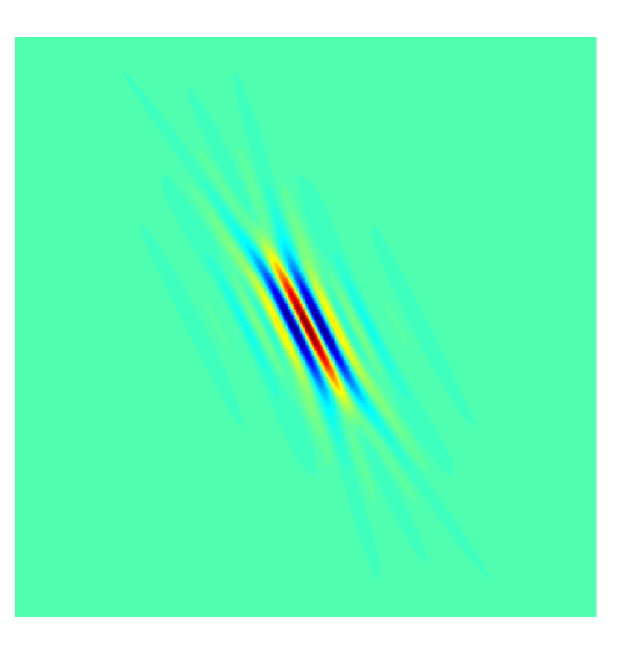}
		\caption{Physical space (green denotes 0)\label{fig:ridgelet:real}}
	\end{subfigure}
	\begin{subfigure}{0.49\textwidth}
		\includegraphics[width=\textwidth]{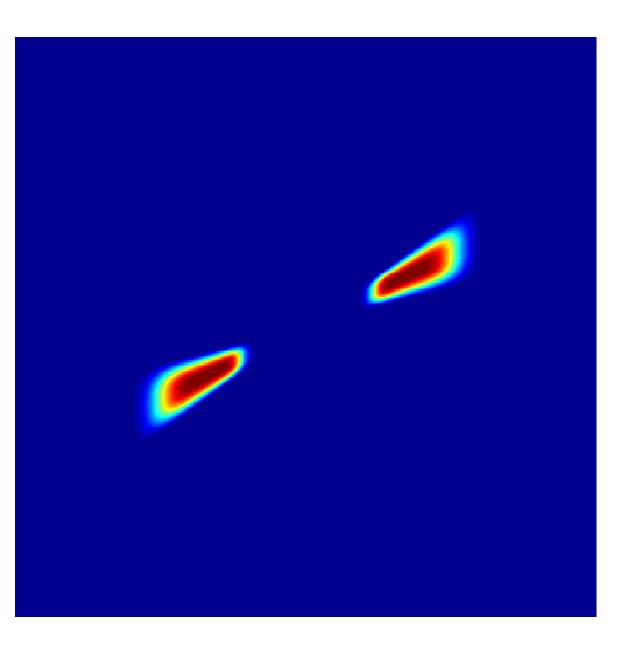}
		\caption{Fourier space (blue denotes 0)\label{fig:ridgelet:ft}}
	\end{subfigure}
	\caption{An illustration of a ridgelet in the two relevant spaces}\label{fig:ridgelet}
\end{figure}

\subsection{Outline}

We begin the paper with a brief investigation of the well-posedness of the main equation in \autoref{sec:wellposed}. In \autoref{sec:discretization}, we introduce the framework of the discretisation, review how the discretised system can be solved algorithmically, and discuss which properties have to be satisfied to achieve optimal complexity -- see \autoref{th:conv_modsolve}.

The subsequent \autoref{sec:ridgeframes} recalls the ridgelet construction and how it forms a frame for the appropriate spaces, as well as the corresponding preconditioner, leading to the stability result \autoref{thm:stabdisc}.

The core of the paper is in \autoref{sec:sparsity}, where we prove compressibility of the system matrix corresponding to the model problem \eqref{eq:LinTrans} -- see \autoref{th:sparse_stiff}. Of the necessary properties for optimal complexity mentioned above, this is the key tool to allow approximate linear-time matrix-vector multiplication. Some necessary but less interesting technical details of the proof are outsourced into the appendix.

In the penultimate \autoref{sec:approx}, we bring together the separate threads to arrive at the result that -- in fact -- ridgelets do achieve the desired optimal complexity (\autoref{cor:ridge_opt_complexity}), and additionally, \emph{also} sparsify typical solutions of such transport equations in the sense of best $N$-term approximations (\autoref{th:approx}).

The final \autoref{sec:numexp} reports on a proof-of-concept implementation and corresponding numerical experiments.

\subsection{Notation}

We let $B_X(x,r):=\{x'\in X:\, \dist_X(x,x')<r\}$ be the open ball in the metric space $X$. Occasionally we omit the space if it is clear from the context. To distinguish the Euclidian norm from the other norms, we denote it by $|\vec x|$. The inner product on $\bbR^d$ is simply denoted by $\vec x \cdot \vec x'$, all other inner products are denoted by $\inpr[\CH]{\cdot,\cdot}$, where the \emph{first} argument is antilinear and the \emph{second} is linear (which is closer to the interpretation as a functional (see e.g. Bra-ket notation) and has several advantages, in our opinion).

The Fourier transform we use is
\begin{align*}
	\hat f(\vec\xi) := \bracket*{\CF (f)} (\vec\xi) := \int_{\bbR^d} f(\vec x)\ee^{-2\pi\ii \vec x \cdot \vec\xi} \d \vec x,
\end{align*}
where we will mostly omit the square brackets for improved legibility if the second term has to be used. In order to limit the amount of constants we have to carry, we define the following relation,
\begin{align*}
	A(y) \lesssim B(y) \, :\Longleftrightarrow \, \exists\, c>0:\, A(y)\le c B(y),
\end{align*}
where the constant has to be independent of $y$. We try to explicitly state each constant at least once, before swallowing it into the $\lesssim$-sign. Additionally, $A\sim B$ denotes the case that both $A\lesssim B$ and $B\lesssim A$ hold.

We abbreviate the minimum and maximum of two quantites (if clear from context which two) by $y_<:=\min(y,y')$ and $y_>:=\max(y,y')$, respectively. %Lastly, we will sometimes use the positive part $(y)_+:=\max(y,0)$.
%\addtocontents{toc}{\setcounter{tocdepth}{2}} % show subsections in the rest

\section{Well-Posedness}\label{sec:wellposed}

Starting point is the differential operator
\begin{align*}
	A: \; \Hs(\bbR^d) \ni u \mapsto \vec s \cdot \nabla u(\vec x)  + \kappa(\vec x) u(\vec x) \in L^2(\bbR^d)
\end{align*}
with fixed $\vec s\in\bbSd$ and a function $\kappa\in L^\infty(\bbR^d)$ that satisfies $\kappa(\vec x)\ge\gamma >0, \; \forall \vec x \in \bbR^d$. The space $\Hs$ is defined as follows.

\begin{definition}\label{def:Hs}
Let $\vec s\in\bbSd$, then we define the \emph{anisotropic Sobolev space}
\begin{align*}
	\Hks(\bbR^d)&:=\set*{f\in L^2(\bbR^d)}{(\vec s\cdot\nabla)f\in H^k(\bbR^d)},
	\intertext{
where $H^k(\bbR^d)$ is the usual Sobolev space. It is equipped with the norm
	}
	\norm{f}_{\Hks(\bbR^d)}^2&:=\norm{f}_{H^k(\bbR^d)}^2+\norm{(\vec s\cdot\nabla)f}_{H^k(\bbR^d)}^2.
	\intertext{
We set $\Hs:=H^{0+\vec s}$. These spaces are more easily characterised on the Fourier side,
	}
	\Hks(\bbR^d)&:=\set{\hat f\in L^2(\bbR^d)}{\reg*{\vec s\cdot \vec\xi} \reg*{\vec\xi}^k \hat f(\hat x, \hat y) \in L^2(\bbR^d)}
	\intertext{
with norm
	}
	\norm*{\hat f}_{\Hks(\bbR^d)}&:=\norm*{\reg*{\vec s\cdot \vec\xi} \reg*{\vec\xi}^k \hat f}_{L^2(\bbR^d)}.
\end{align*}
\end{definition}

To make the operators involved positive definite, we have to restrict ourselves to solving the normal equation $A^*Au=A^*f\in L^2(\bbR^d)$, which we do by minimising the $L^2$-residual,
\begin{align}\label{eq:L2min}
	u_0=\argmin_{v\in \Hs} \norm{A^*Av-A^*f}_{L^2}.
\end{align}
\begin{theorem}\label{th:wellposed}
	The problem of finding $u\in\Hs$ such that $Au=f\in L^2(\bbR^d)$ is well-posed. In addition, for $u\in\Hs$, the following norm-equivalence holds
	\begin{align}\label{eq:Hs_elliptic}
		\norm{Au}_{L^2}\sim \norm{u}_{\Hs}.
	\end{align}
\end{theorem}

Before we come to the proof of \autoref{th:wellposed} we introduce some notation.
Let $\Rs$ be an orthogonal matrix which maps $\vec s$ to $\vec e_1=(1,0,\ldots)^\top$, and let $\Rs^{-1}=\Rs^\top$ be its inverse. This rotation is not unique for $d>3$ (see also \autoref{rem:rotation_ambiguous}), however, an arbitrary but fixed choice suffices for this section. We define the respective pullbacks for $f\in L^2(\bbR^d)$ by
\begin{align*}
	\rs f(\vec x):= f(\Rs^{-1} \vec x), \qquad \rs^{-1} f(\vec x) := f(\Rs \vec x),
\end{align*}
thus $\rs f(\vec e_1) =f (\vec s), \; \rs^{-1} f(\vec s) = f(\vec e_1)$ -- if $f$ is continuous. These pullbacks is also well-defined for $L^2(\bbR^d)$-functions, as long as we don't evaluate at a single value -- since we always integrate in the following, this presents no problem.

We will use these transformations to restrict ourselves to dealing with just the derivative in the first component $x_1$, as the following lemma shows.
\begin{lemma}\label{lem:diff_s}
	For $u\in\Hs(\bbR^ d)$,
	\begin{align}\label{eq:diff_s}
		\vec s \cdot \nabla u = \rs^{-1} \frac{\dd}{\dd x_1}(\rs u)(\vec x).
	\end{align}
\end{lemma}
\begin{proof}
	Our notation for the Jacobian is
	\begin{align*}
		\dd g(\vec x) = \parens{ \Dp[g_i]{x_j}(\vec x) }_{\substack{\phantom{j=1}\mathllap{i=1}\mathrlap{,\ldots,m}\phantom{,\ldots,n}\\ j=1,\ldots,n}} \qquad \text{for } g:\bbR^n\to\bbR^m,
	\end{align*}
	whereby $\D{\vec s}g(\vec x)= \parens*{\dd g(\vec x)} \vec s$, and the chain rule is written as $\dd(g\circ h)(\vec x)=\dd g(h(\vec x)) \dd h(\vec x)$ for $h:\bbR^\ell\to\bbR^n$. If $m=1$, the vector is usually written upright, of course, i.e. $\nabla g = (\dd g(\vec x))^\top$. Thus,
	\begin{align*}
		\D{x_1}(\rs u)(\vec x) = \parens*{\dd(u\circ \Rs^{-1})(\vec x)} \vec e_1 = \dd u (\Rs^{-1}\,\vec x)\smash{\underbrace{\Rs^{-1}\, \vec e_1}_{=\vec s}} = \vec s\cdot\nabla u(\Rs^{-1}\,\vec x)=\rs(s\cdot\nabla u(\vec x)).
	\end{align*}
	Applying $\rs^{-1}$ yields the result.
\end{proof}
\begin{remark}
	An immediate consequence of \autoref{lem:diff_s} is
	\begin{align*}
		u\in\Hs(\bbR^d) \Longleftrightarrow \rs u \in \Hs[\vec e_1](\bbR^d).
	\end{align*}
\end{remark}
\begin{proof}[Proof of \autoref{th:wellposed}]
	This proof is a simple adaptation of the proof in \cite{grohs_schwab}. By the previous lemma we immediately see that
	\begin{align}\label{eq:equiv_pde_trafo_rs}
		Au=f \Longleftrightarrow \rs Au = \rs f \Longleftrightarrow \frac{\dd}{\dd x_1}(\rs u) + \rs \kappa \, \rs u = \rs f.
	\end{align}
	Using variation of constants, this can be solved explicitly for arbitrary $f\in L^2(\bbR^d)$ in the following way. We let $\vec x':= (x_2,\ldots,x_d)^\top \in\bbR^{d-1}$ be the vector of the $d-1$ lower components of $\vec x= \binom{x_1}{\vec x'}$ and compute
	\begin{align*}
		y(x_1,\vec x') := \ee^{-K(x_1,\vec x')} \parens{\int_0^{x_1} \rs f (t,\vec x') \ee^{K(t,\vec x')} \d t + C }, \quad \text{where } \quad K(t,\vec x')= \int_0^t \rs \kappa (r,\vec x') \d r.
	\end{align*}
	Note that since $\kappa\ge\gamma >0$, $K$ is a strictly increasing function of $x_1$ (with slope at least $\gamma$), in particular $K(t,\vec x')-K(x_1,\vec x') \le \gamma (t-x_1)$ for $t\le x_1$.
	
	In general, $y$ will not be in $L^2(\bbR^d)$ -- something we clearly need. As a necessary requirement, it must tend to zero for large negative $x_1$, which -- considering the exponential growth of the first factor -- means that the second factor must tend to zero, thus determining the constant $C$;
	\begin{align*}
		y(x_1,\vec x') \xrightarrow{x_1\to-\infty} 0, \quad \Longrightarrow \quad C=\int_{-\infty}^0 \rs f(t,\vec x') \ee^{K(t,\vec x')} \d t.
	\end{align*}
	For arbitrary $g\in L^2(\bbR^d)$ we compute
	\begin{align*}
		\abs*{\inpr{y,g}_{L^2}}
		&= \abs{ \int_{\bbR^d}  g(\vec x) \int_{-\infty}^{x_1} \rs f(t,\vec x') \ee^{K(t,\vec x')-K(x_1,\vec x')} \d t \d \vec x }
		\le \int_{\bbR^d} \int_{-\infty}^{x_1} |g(\vec x)| |\rs f(t,\vec x')| \ee^{K(t,\vec x')-K(x_1,\vec x')} \d t \d \vec x \\
		&\le \int_{\bbR^d} \int_{-\infty}^{x_1} |g(\vec x)| |\rs f(t,\vec x')| \ee^{\gamma(t-x_1)} \d t \d \vec x
		= \int_{\bbR^d} \int_{-\infty}^0 |g(\vec x)| |\rs f(x_1+r,\vec x')| \ee^{\gamma r} \d r \d \vec x \\
		&\stackrel{\mathclap{(*)}}{=} \int_{-\infty}^0 \int_{\bbR^d} |g(\vec x)| |\rs f(x_1+r,\vec x')| \ee^{\gamma r} \d \vec x \d r
		\stackrel{(**)}{\le} \norm{g}_{L^2} \norm{\rs f}_{L^2}  \int_{-\infty}^0 \ee^{\gamma r} \d r = \frac{1}{\gamma} \norm{g}_{L^2} \norm{f}_{L^2},
	\end{align*}
	where for $(*)$ we used Fubini's theorem, while $(**)$ makes use of the Cauchy--Schwarz inequality, as well as the fact that translating the argument of a function $f\in L^2(\bbR^d)$ by a fixed vector (in this case $(r,0,\ldots)^\top$) preserves the norm. By the Riesz representation theorem, we see  $\norm{y}_{L^2}\le\frac{1}{\gamma}\norm{f}_{L^2}$, and in particular that $y\in \Hs$, since differentiability is obvious from the construction. Setting $u:=\rs^{-1}y$, \eqref{eq:equiv_pde_trafo_rs} shows that we have found a solution of $Au=f$ for arbitrary $f\in L^2(\bbR^d)$, which is what we wanted to prove.
	
	To see \eqref{eq:Hs_elliptic}, we consider the norm of the derivative
	\begin{align*}
		\norm{\D{x_1} \rs u}_{L^2} = \norm{\rs f-\rs \kappa \, \rs u}_{L^2} \le \norm{f}_{L^2} + \norm{\kappa}_{L^\infty}\norm{u}_{L^2} \lesssim \norm{f}_{L^2},
	\end{align*}
	where the last inequality is due to $\norm{u}_{L^2}=\norm{y}_{L^2}\le\frac{1}{\gamma}\norm{f}_{L^2}$. Consequently,
	\begin{align}\label{eq:norm_A_inverse}
		\norm{\rs u}_{\Hs[\vec e_1]}^2 = \norm{\rs u}_{L^2}^2 + \norm{\D{x_1} \rs u}_{L^2}^2 \lesssim \norm{f}_{L^2}^2 = \norm{\rs f}_{L^2}^2 \stackrel{\eqref{eq:equiv_pde_trafo_rs}}{=} \int_{\bbR^d} \abs[\Big]{ \frac{\dd}{\dd x_1} \rs u(\vec x) + \rs\kappa(\vec x)\,\rs u(\vec x)}^2 \d \vec x,
	\end{align}
	and putting everything together (as well as substituting twice), we arrive at
	\begin{align*}
		\norm{u}_{\Hs}^2
		&= \int_{\bbR^d} \abs*{\vec s \cdot \nabla u (\vec x) }^2 + \abs*{u(\vec x)}^2 \d \vec x \,
		\stackrel{\eqref{eq:diff_s}}{=} \int_{\bbR^d} \abs[\Big]{\rs^{-1} \D{x_1}(\rs u)(\vec x) }^2 + \abs*{u(\vec x)}^2 \d \vec x \\
		&= \int_{\bbR^d} \parens{ \abs[\Big]{\D{x_1}(\rs u)(\vec x) }^2 + \abs*{\rs u(\vec x)}^2 } \smash{\underbrace{\abs*{\det \Rs^{-1} }}_{=1}} \vphantom{\Bigg()} \d \vec x \,
		\stackrel{\eqref{eq:norm_A_inverse}}{\lesssim} \int_{\bbR^d} \abs[\Big]{\D{x_1}(\rs u)(\vec x) + \rs\kappa(\vec x) \, \rs u(\vec x)}^2 \d \vec x \\
		&= \int_{\bbR^d} \abs[\Big]{\rs^{-1}\D{x_1}(\rs \,u)(\vec x) + \kappa(\vec x) \, u(\vec x)}^2 \abs*{\det \Rs } \d \vec x \,
		= \int_{\bbR^d} \abs*{\vec s \cdot \nabla u (\vec x) + \kappa(\vec x) \, u(\vec x) }^2 \d \vec x = \norm{Au}_{L^2}^2.
	\end{align*}
	The second inequality necessary for \eqref{eq:Hs_elliptic} is immediate,
	\begin{align*}
		\norm{Au}_{L^2} = \norm{\vec s \cdot \nabla u + \kappa\,u}_{L^2}
		\le \norm{\vec s \cdot \nabla u}_{L^2} + \norm{\kappa}_{L^\infty} \norm{u}_{L^2} \lesssim \norm{u}_{\Hs},
	\end{align*}
	and thus we have shown the equivalence of the norms, $\norm{Au}_{L^2}\sim \norm{u}_{\Hs}$, which finishes the proof.
\newline
\end{proof}

\begin{corollary}\label{cor:welldef}
	For every $\ell\in \parens*{\Hs}'$ -- the dual of $\Hs$ -- there exists a unique $u_0\in \Hs$ which solves \eqref{eq:L2min}. Moreover, the solution is characterized by the variational equation
	\begin{equation}\label{eq:L2minVar}
		a(v,u_0) = \ell(v)\quad \text{ for all }v\in \Hs,
	\end{equation}
	where we put
	\begin{equation}\label{eq:Bilinformdef}
		a(v,u):=\inpr{Av,Au}_{L^2}.
	\end{equation}
	In particular, well-definedness holds for 
	\begin{equation}\label{eq:loadfunctional}
		\ell_f(v):=\inpr{Av,f}_{L^2}\quad \text{with }f\in L^2(\bbR^d).
	\end{equation}
\end{corollary}
\begin{proof}
	The first statement is a direct consequence of \autoref{th:wellposed} (yielding continuity and coercivity of $a$ in terms of $\norm{\cdot}_{\Hs}$) and the Lax-Milgram lemma.
	Equation \eqref{eq:L2minVar} is simply a reformulation as a linear least squares problem. Finally, well-definedness for \eqref{eq:loadfunctional} holds,
	since $\ell$ as defined in \eqref{eq:loadfunctional} is trivially continuous, as can be seen from the
	Cauchy-Schwarz inequality.
\end{proof}
\autoref{cor:welldef} shows that, using $L^2$-regularization, we may interpret
the operator $A^*A$ as a bounded and boundedly invertible operator $A^*A: \Hs \to \parens*{\Hs}'$. 

\section{Discretisation}\label{sec:discretization}
In our paper we aim to solve \eqref{eq:L2min} via solving
a discretization of the linear system \eqref{eq:L2minVar}.
Several ingredients are needed to render this approach efficient:
\begin{enumerate}[(i)]
	\item Uniform well-conditionedness of the resulting infinite discrete linear system
	\item Fast approximate matrix-vector multiplication for the discrete operator matrix
	\item  Efficient approximation of typical solutions
\end{enumerate}
There exists several results which essentially state that, whenever (i), (ii) and (iii) 
are satisfied, then the linear system \eqref{eq:L2minVar} can be solved in optimal 
computational complexity \cite{cdd,Stevenson2004,Dahlke2007}.
We will formalize what is precisely meant by properties (i)--(iii) later on, but first 
we need to introduce some further notation. 
%
%-----------------------------
\subsection{Gelfand Frames}
%----------------------------
%
Following \cite{Dahlke2007}, we will use the concept of a \emph{Gelfand frame} to 
discretize \eqref{eq:L2minVar}. Our starting point is a bounded and boundedly invertible operator
\begin{equation}\label{eq:opboundlin}
	F : \CH \to \CH',
\end{equation}
for some Hilbert space $\CH$, inducing a symmetric and coercive bilinear form,
\begin{equation}
	a(u,v)=\inpr{Fu,v}_{\CH'\times\CH}, \qquad a(v,v)\sim\norm{v}_\CH^2,
\end{equation}
where $\inpr{\cdot,\cdot}_{\CH'\times\CH}$ is the duality pairing of $\CH'$ and $\CH$. The aim is to provide an efficient discretization of this operator.

To do this we first consider discrete systems $\Phi = (\varphi_\lambda)_{\lambda\in \Lambda}$ which 
provide a stable decomposition and reconstruction procedure, so-called \emph{frames}:
\begin{definition}
	Let $\Lambda$ be a discrete set and $\CH$ a Hilbert space. A system $\Phi = (\varphi_\lambda)_{\lambda\in \Lambda}$
	with $\varphi_\lambda\in \CH$ for all $\lambda \in \Lambda$
	is called a \emph{frame} if 
	there exist constants $0<c_\Phi \le C_\Phi <\infty$
	such that
	\begin{equation}\label{eq:framedef}
		c_\Phi\norm{f}_\CH^2
		\le \sum_{\lambda\in \Lambda}\abs*{\inpr{\varphi_\lambda,f}_\CH}^2
		\le C_\Phi \norm{f}_\CH^2.
	\end{equation}
	If $c_\Phi = C_\Phi$ one calls $\Phi$ a \emph{tight frame}; if, additionally $c_\Phi = 1$, 
		one speaks of a \emph{Parseval frame}.
\end{definition}
For a frame $\Phi$ for $\CH$ we also need to define the 
\emph{frame analysis operator}
\begin{align*}
	G &: \left\{\begin{array}{rcl}\CH & \to & \ell^2(\Lambda)\\
	f & \mapsto & \inpr{\Phi,f}_\CH := \parens*{\inpr{\varphi_\lambda,f}_\CH}_{\lambda\in\Lambda},
	\end{array}\right.
	\intertext{
and its dual the \emph{frame reconstruction operator}
	}
	G^* &: \left\{\begin{array}{rcl}\ell^2(\Lambda) & \to & \CH\\
	\Vc & \mapsto & \Phi \Vc := \sum_{\lambda\in\Lambda} c_\lambda \varphi_\lambda.
	\end{array}\right.
\end{align*}
The definition of a frame implies that the operator 
\[
	S_\Phi : \CH\to \CH,\ f \mapsto G^* G f
\]
is symmetric, bounded and boundedly invertible.
The \emph{canonical dual frame} of $\Phi$ is defined as $\wt \Phi:= S_\Phi^{-1}\Phi\subseteq\CH$.

Additionally, we need the notion of a \emph{Gelfand triple}:
\begin{definition}
	Let $\CH$ be a Hilbert space with dual $\CH'$.
	If we have
	\[
		\CH \subseteq L^2(\Omega) \subseteq \CH'
	\]
	with $\CH$ a Hilbert space such that all inclusions above are continuous and dense , then the triplet $(\CH,L^2(\Omega),\CH')$ is called a \emph{Gelfand triple}.
	%[Konsistenz in Definition oder Remark?] if the duality pairing of $\CB$ is compatible with the inner product on $\CH$, i.e. $\inpr{g,f}_{\CB'\times\CB} = \inpr{g,f}_\CH$ for all $g\in\CH$ and $f\in \CB$.
\end{definition}
\begin{remark}
	A canonical example for a Gelfand triple is induced by the Sobolev space $\CH = H^1_0(\Omega)$
	for some domain $\Omega \subseteq \bbR^d$. Of more interest to our purpose is the case
	\[
		\CH = \Hs(\bbR^d),
	\]
	which also induces a Gelfand triple.
	
	The concept of Gelfand triples is actually much more general and allows for $\CB\subseteq\CH\subseteq\CB'$ where $\CB$ is a Banach space and $\CH$ a Hilbert space (with the same requirement on the embeddings). However, since we need a Hilbert space for the results of \cite{Dahlke2007} in relation to \eqref{eq:opboundlin} anyway, we omit this.
\end{remark}
We call a frame $\Phi = (\varphi_{\lambda})_{\lambda\in \Lambda}$
for $\CH$
a \emph{Gelfand frame} if $\Phi \subseteq \CH$, and there exists a Gelfand
triple $(\CH_d , \ell^2(\Lambda),\CH_d')$
of sequence spaces such that the operators
\begin{align*}%\label{eq:}
	G^*_\Phi: &\left\{\begin{array}{rcl}\CH_d & \to & \CH\\
	\Vc & \mapsto & \Phi \Vc
	\end{array}\right.
	&& \text{and} &
	G_{\wt{\Phi}}: &\left\{\begin{array}{rcl}\CH & \to & \CH_d\\
	f & \mapsto & \inpr*{\wt\Phi,f}_{\CH'\times\CH} = \inpr*{\wt\Phi,f}_{L^2}
	\end{array}\right.
	\intertext{
are bounded. By duality, the operators
	}
	G_\Phi: &\left\{\begin{array}{rcl}\CH' & \to & \CH'_d\\
	f & \mapsto & \inpr*{\Phi,f}_{\CH'\times\CH}
	\end{array}\right.
	&& \text{and} &
	G^*_{\wt{\Phi}}: &\left\{\begin{array}{rcl}\CH'_d & \to & \CH'\\
	\Vc & \mapsto & \wt\Phi \Vc
	\end{array}\right.
\end{align*}
are also bounded.
In addition, suppose that there exists an isomorphism
$D_\CH: \CH_d \to \ell^2(\Lambda)$
such that its $\ell^2(\Lambda)$--adjoint $D_{\CH}^*\colon \ell^2(\Lambda)\to\CH_d'$
is also an isomorphism.

Now assume that we want to solve the operator equation 
\begin{equation}\label{eq:opeqabstract}
	Fu = f
\end{equation}
where $f\in \CH'$ and $F$ given above in \eqref{eq:opboundlin}.
Using a Gelfand frame $\Phi$ we can discretize \eqref{eq:opeqabstract} to yield
the discrete system
\begin{equation}\label{eq:opeqabstractdistcrete}
	\VF\Vu = \Vf,
\end{equation}
with 
\[
	\VF = (D_\CH^*)^{-1}G_\Phi F G_\Phi^* D_\CH^{-1}
	\quad \text{and} \quad
	\Vf = (D_\CH^*)^{-1}G_\Phi f.
\]
We have the following result which states that the discrete version \eqref{eq:opeqabstractdistcrete} yields a uniformly well-conditioned infinite linear system.
\begin{lemma}[{\cite[Lemma 4.1]{Dahlke2007}}]\label{lem:stable}
The operator $\VF : \ell^2(\Lambda) \to \ell^2(\Lambda)$ is bounded and boundedly invertible on its range $\ran(\VF) = \ran\parens*{(D_\CH^*)^{-1}G_\Phi}$. Furthermore, $\ker(\VF)=\ker(G_\Phi^* D_\CH^{-1}).$
\end{lemma}
%

%
%%%%%%%%%%%%%%%%%%%%%%%%%%%%%%%%%%%%%%%%%%%%%%%%%%%%%%%%%%%%%%%%%%%5
\subsection{Numerical Solution of the Discrete System}
%%%%%%%%%%%%%%%%%%%%%%%%%%%%%%%%%%%%%%%%%%%%%%%%%%%%%%%%%%%%%%%%%%%
%
In the previous subsection we have reformulated the 
operator equation \eqref{eq:opeqabstract} in terms of a discrete linear system
\begin{equation}\label{eq:linsys}
	\VF \Vu = \Vf,
\end{equation}
with $\VF$ and $\Vf$ given as above (in particular, this means that $\Vf\in\ran(\VF)$ and that $\VF$ is positive definite).
If we were able to compute with infinite vectors, at this point
we could simply
use a standard iterative solver such as a damped Richardson iteration
\begin{equation}\label{eq:Richardson}
	\Vu^{(j+1)} = \Vu^{(j)} - \alpha\parens*{\VF \Vu^{(j)} - \Vf},\quad
	\Vu^{(0)} = \Vzero.
\end{equation}
Due to the well-conditionedness of the matrix $\VF$ ensured by \autoref{lem:stable}
and the fact that the iterates stay in $\ran (\VF)$ in each step, it is easy to show that
for appropriate damping $\alpha$ the sequence $\Vu^{(j)}$ converges
geometrically to the sought solution $\Vu$ in 
the $\ell^2(\Lambda)$-norm, i.e. 
\[
	\norm*{\Vu-\Vu^{(j)}}_{\ell^2(\Lambda)} \lesssim \rho^j
\]
for some $\rho<1$, depending on the spectral properties of the operator
$\VF$.

In view of a practical realization of the above scheme, two fundamental issues arise:
\begin{enumerate}[(A)]
	\item We only have finite computing capabilities at our disposal, and therefore all
	operations in \eqref{eq:Richardson} can only be carried out approximatively
	\item Due to the approximate computation of the iteration \eqref{eq:Richardson},
	we might fall out of $\ran (\VF)$ during iteration. A consequence is that
	an error in $\ker (\VF)$ might not be reduced in subsequent iterations.
\end{enumerate}
In the remainder of this section we discuss how these two issues can be dealt with,
without compromising numerical accuracy. We start with (A) which is by now classical
for wavelet discretizations of elliptic PDEs. The approximative evaluation of 
the Richardson iteration utilizes the following three procedures:
\begin{itemize}
	\item $\mathbf{RHS}[\eps,\Vf]\to \Vf_\eps$: determines for 
	$\Vf\in \ell^2(\Lambda)$ a finitely supported $\Vf_\eps\in \ell^2(\Lambda)$
	such that
	\[
		\norm{\Vf - \Vf_\eps}_{\ell^2(\Lambda)} \le \eps;
	\] 
	\item $\mathbf{APPLY}[\eps,\VA,\Vv]\to \Vv_\eps$: 
	determines for $\VA:\ell^2(\Lambda)\to \ell^2(\Lambda)$ and for a finitely
	supported $\Vv\in \ell^2(\Lambda)$ a finitely supported $\Vv_\eps$ such that
	\[
		\norm{\VA\Vv - \Vv_\eps}_{\ell^2(\Lambda)} \le \eps;
	\]
	\item $\mathbf{COARSE}[\eps,\Vc]\to \Vc_\eps$:
	determines for a finitely supported $\Vu\in \ell^2(\Lambda)$ a finitely
	supported $\Vu_\eps\in \ell^2(\Lambda)$ with 
	at most $N$ nonzero coefficients (by setting the other entries to zero), such that 
	\begin{equation}\label{eq:coarseerror}
		\norm{\Vc - \Vc_\eps}_{\ell^2(\Lambda)} \le \eps.
	\end{equation}
	Moreover, if $N_{\mathrm{min}}$ is the minimal number of coefficients necessary to achieve \eqref{eq:coarseerror}, the output achieves $N\lesssim N_{\mathrm{min}}$ in linear time (whereas satisfying $N_{\mathrm{min}}$ would incur an additional log-factor in the complexity).
\end{itemize}
We refer to \cite{cdd,Stevenson2004,Dahlke2007} for information on the numerical realization
of these routines. Assuming the existence of numerical procedures as above, we can
formulate the first numerical algorithm to solve the discrete linear system
\eqref{eq:linsys} up to accuracy $\eps >0$, given as \autoref{alg:dampedRich} below.

\begin{algorithm}[H]
	%Option "H" for "here" -- otherwise the algorithm is treated as a float and shifted around.
		\caption{Inexact Damped Richardson Iteration}\label{alg:dampedRich}
		\KwData{$\eps>0$, $\VF$, $\Vf$}
		\KwResult{$\Vu_\eps = \mathbf{SOLVE}[\eps,\VF,\Vf]$}
		Let $\theta <\frac 13$ and $K\in \bbN$ such that
		$3\rho^K <\theta$.
		$i:=0$, $\Vu^{(0)} := \Vzero$,
		$\eps_0:=\norm*{\VF\bigr|_{\ran (\VF)}^{-1}}\norm*{\Vf}_{\ell^2(\Lambda)}$
		
		\While{$\eps_i >\eps$}{
			$i:= i + 1$;\\
			$\eps_i := 3\rho^K\eps_{i-1}/\theta$;\\
			$\Vf^{(i)}:=\mathbf{RHS}[\theta\eps_i/(6\alpha K),\Vf]$;\\
			$\Vu^{(i,0)}:= \Vu^{(i-1)}$;\\
			\For{$j=1,\dots , K$}{
				$\Vu^{(i,j)}:=\Vu^{(i,j-1)}-\alpha\parens*{
				\mathbf{APPLY}
				\bracket*{\theta\eps_i/(6\alpha K),\VF,\Vu^{(i,j-1)}}
				- \Vf^{(i)}}$;
			}
			$\Vu^{(i)}:=\mathbf{COARSE}[(1-\theta)\eps_i,\Vu^{(i,K)}]$;
		}
		$\Vu_\eps := \Vu^{(i)}$;
\end{algorithm}\par\vspace{\baselineskip}
Conditional on the three routines above, we have thus formulated a feasible 
algorithm for the approximate solution of \eqref{eq:linsys}.
We will talk about the computational complexity and accuracy of this algorithm
in a moment, but first let us discuss the issue (B), namely that errors in 
$\ker (\VF)$ may not be decreased during the iterations in \autoref{alg:dampedRich}.
In \cite{Stevenson2004} this problem is addressed and in particular it is shown
that possibly the computational complexity of \autoref{alg:dampedRich} may
deteriorate unless some additional conditions are satisfied. 
While it is believed that those conditions -- most notably the compressibility of the orthogonal projection -- are valid, it is impossible
to prove them at this time.

%
%%%%%%%%%%%%%%%%%%%%%%%%%%%%%%%%%%%%%%%%55
\subsubsection{The $\mathbf{modSOLVE}$-Algorithm}
\label{sec:proj}
%%%%%%%%%%%%%%%%%%%%%%%%%%%%%%%%%%%%%%%%%%
%
A remedy is to apply a bounded projection $\VP$ such that
\begin{equation}\label{eq:projectorker}
	\ker (\VP) = \ker(\VF)\stackrel{\text{Lemma }\ref{lem:stable}}{=}\ker (G_\Phi^* D_\CH^{-1})
\end{equation}
every few
Richardson iterations in order to remove unwanted error components
in $\ker (\VF)$.

The following discussion also applies to general Gelfand triples $(\CB,\CH,\CB')$, however, we continue in the notation so far (requiring Hilbert instead of Banach spaces), again using the Gelfand frame $\Phi$ with canonical dual $\wt \Phi$.

In order to arrive at a projector 
satisfying \eqref{eq:projectorker} we consider
the (injective) mapping
\[
	Z: \left\{\begin{array}{ccc}\CB & \to & \ell^2(\Lambda) \\
	f & \mapsto & D_\CH G_{\wt{\Phi}}f
	\end{array}\right.
\]
By the definition of a Gelfand frame, this mapping is bounded.
We also have that
\begin{equation}\label{eq:projident}
	G_\Phi^* D_\CH^{-1}Zf = 
	G_\Phi^* D_\CH^{-1}D_\CH G_{\wt{\Phi}}f
	= G_\Phi^* G_{\wt{\Phi}}f = f
	\quad \text{ for all }f\in \CH.
\end{equation}
Therefore, we can put 
\[
	\VP:=\left\{ \begin{array}{ccc} \ell^2(\Lambda) & \to & \ell^2(\Lambda)\\
	\Vc & \mapsto & Z G_\Phi^* D_\CH^{-1}\Vc
	\end{array}\right.
\]
and see, using \eqref{eq:projident}, that this mapping is indeed a projector with %
\[
	\ker (\VP) = \ker (G_\Phi^* D_\CH^{-1}),
\]
which is exactly what we wanted.

To find the matrix representation of $\VP$ we note that
\begin{equation}\label{eq:projmatrix}
	\VP =D_{\CH} \inpr*{\wt{\Phi},\Phi }_{L^2} D_{\CH}^{-1}.
\end{equation}
Using the projection operator $\VP$ as just defined, we now follow
\cite{Stevenson2004} and formulate a slightly modified algorithm to approximatively
solve \eqref{eq:linsys} in \autoref{alg:modsolve}.
\begin{algorithm}
		\caption{Modified Inexact Damped Richardson Iteration}\label{alg:modsolve}
		\KwData{$\eps>0$, $\VF$, $\Vf$}
		\KwResult{$\Vu_\eps = \mathbf{modSOLVE}[\eps,\VF,\VP,\Vf]$}
		Let $\theta <\frac13$ and $K\in \bbN$ such that
		$3\rho^K \norm{\VP}<\theta$.
		$i:=0$, $\Vu^{(0)} := \Vzero$,
		$\eps_0:=\norm*{\VP}
		\norm*{\VF\bigr|_{\ran (\VF)}^{-1}}\norm{\Vf}_{\ell^2(\Lambda)}$
		
		\While{$\eps_i >\eps$}{
			$i:= i + 1$;\\
			$\eps_i := 3\rho^K \norm{\VP} \eps_{i-1}/\theta$;\\
			$\Vf^{(i)}:=\mathbf{RHS}[\theta\eps_i/(6\alpha K\norm{\VP}),\Vf]$;\\
			$\Vu^{(i,0)}:= \Vu^{(i-1)}$;\\
			\For{$j=1,\dots , K$}{
				$\Vu^{(i,j)}:=\Vu^{(i,j-1)}-\alpha\parens*{
				\mathbf{APPLY}
				\bracket*{\theta\eps_i/(6\alpha K\norm{\VP}),\VF,\Vu^{(i,j-1)}}
				- \Vf^{(i)}}$;
			}
			$\Vz^{(i)}:=\mathbf{APPLY}\bracket*{\theta\eps_i/3,\VP,
			\Vu^{(i,K)}}$;\\
			$\Vu^{(i)}:=\mathbf{COARSE}[(1-\theta)\eps_i,\Vz^{(i)}]$;
		}
		$\Vu_\eps := \Vu^{(i)}$;
\end{algorithm}
%
%
%%%%%%%%%%%%%%%%%%%%%%%%%%%%%%%%%%%%%%%%%%%%%%%%%%%%%
\subsubsection{Complexity Analysis}
%%%%%%%%%%%%%%%%%%%%%%%%%%%%%%%%%%%%%%%%%%%%%%%%%%%%%
%
We now turn to a complexity analysis of the algorithms 
$\mathbf{SOLVE}$ and $\mathbf{modSOLVE}$ introduced above.
To this end it it convenient to work with so-called
\emph{weak $\ell^p$-spaces}.
\begin{definition}
	For $0<p<2$ we define the weak $\ell^p$--space
	-- denoted by $\ell^p_w(\Lambda)$ -- as
	\[
		\ell^p_w(\Lambda):=
		\set*{\Vc\in \ell^2(\Lambda)}{|\Vc|_{\ell^p_w(\Lambda)}
			:=\sup_{n\in \bbN}n^{\frac 1p}|\gamma_n(\Vc)| <\infty},
	\]
	where $\gamma_n(\Vc)$ denotes the $n$-th largest coefficient
	in modulus of $\Vc$.
\end{definition}
\begin{remark}
	The quasi-Banach spaces $\ell^p_w$ are instrumental in 
	the study of nonlinear best $N$-term approximation.
	More precisely, membership of the coefficient sequence in $\ell^p_w$ is equivalent
	to a best $N$-term approximation rate of order
	$N^{-\sigma}$, where $\sigma = \frac{1}{p}-\frac{1}{2}$,
	see \cite{DeVore1998}.
	Moreover, it is easy to see that we have the inclusions
	\[
		\ell^p \subseteq \ell^p_w \subseteq 
		\ell^{p+\eps}
	\]
	for any $\eps >0$.
\end{remark}
To achieve optimal convergence rates for our problem through the techniques introduced in
\cite{cdd}, a key ingredient is \emph{compressibility} of the discretized operator equation. 
Such a property guarantees
the existence of linear-time approximate matrix-vector multiplication algorithms
$\mathbf{APPLY}$
which are used in the iterative solution of the operator equation,
see \cite{cdd,Stevenson2004} for more information.
\begin{definition}
	A matrix $\VA$ is called \emph{$\sigma^*$-compressible}
	if for every $\sigma<\sigma^*$ and $k\in \bbN$ there exists
	a matrix $\VA^{[k]}$ such that
	\begin{enumerate}[(i)]
		\item the matrix $\VA^{[k]}$
		has at most $\alpha_k2^k$ non-zero entries in 
		each column,
		\item we have 
		\begin{align*}
			\norm*{\VA - \VA^{[k]}}_2
			\le C_k
		\end{align*}
	\end{enumerate}
	so that
	the sequences $(\alpha_k)_{k\in \bbN},\ (C_k2^{\sigma k})_{k\in\bbN}$
	are both summable.
\end{definition}
\begin{definition}[{\cite[Def. 3.9]{Stevenson2004}}]\label{def:sigma-optimal}
A vector $\Vc \in \ell^2$ is called $\sigma^*$-\emph{optimal}, when for a suitable routine \textbf{RHS}, for each $\sigma\in(0,\sigma^*)$ with $p:=(\frac 12 + \sigma)^{-1}$, the following is valid for $\Vc_\eps=\textbf{RHS}[\eps,\Vc]$:
\begin{enumerate}
	\item $\# \supp \Vc_\eps \lesssim \eps^{-1/\sigma} \abs{\Vc}_{\ell^p_w}^{1/\sigma}$
	\item The number of arithmetic operations used to compute $\Vc_\eps$ is at most a multiple of $\eps^{-1/\sigma} \abs{\Vc}_{\ell^p_w}^{1/\sigma}$.
\end{enumerate}
\end{definition}
We can now formulate the main result of \cite[Theorem 3.11]{Stevenson2004}.
\begin{theorem}[Convergence of $\mathbf{modSOLVE}$]\label{th:conv_modsolve}
	Assume that for some $\sigma^*>0$, the matrices $\VF$ and 
	$\VP$ are $\sigma^*$-compressible and that
	for some $\sigma\in (0,\sigma^*)$ and $p:=\frac{1}{\frac12 + \sigma}$,
	the system $\VF \Vu = \Vf$ has a solution
	$\Vu \in \ell^p_w(\Lambda)$.
	Moreover, assume that $\Vf$ is $\sigma^*$-optimal. Then for all $\eps >0$, $\Vu_\eps:=\mathbf{modSOLVE}[\eps,\VF,\VP,\Vf]$
	satisfies
	\begin{enumerate}[(I)]
		\item $\# \supp \Vu_\eps 
		\lesssim \eps^{-1/\sigma}|\Vu|_{\ell^p_w(\Lambda)}^{1/\sigma}$,
		\item the number of arithmetic operations to compute
		$\Vu_\eps$ is at most a multiple of 
		$\eps^{-1/\sigma}|\Vu|_{\ell^p_w(\Lambda)}^{1/\sigma}$.
	\end{enumerate}
	Furthermore,
	$\norm{\VP\Vu - \Vu_\eps}_{\ell^2(\Lambda)}\le \eps$ and so
	$\norm{u - G^*_\Phi D_\CB^{-1}\Vu_\eps}_\CH \lesssim \eps$.
\end{theorem}
An analogous result holds also for the algorithm $\mathbf{SOLVE}$, provided
that the orthogonal projector $\VP$ onto the range of $\VF$ is $\sigma^*$-compressible
as above. However, except for trivial cases it is not possible to 
verify this assumption with current mathematical technology \cite{Stevenson2004,Dahlke2007}.

The line of attack to solve the operator equation 
\eqref{eq:L2min} is now clear:
We have to construct a Gelfand frame 
$\Phi$ for the Gelfand triple
$\parens*{\Hs , L^2, (\Hs)'}$ and show that the resulting matrices 
$\VF$ and $\VP$ are compressible. This is done in the following sections.

\section{Ridgelet Frames}\label{sec:ridgeframes}
%
%------------------------------------------------------
	\subsection{Ridgelet Gelfand Frames for $\Hs$}
%------------------------------------------------------
%
In order to make use of the general results of the previous subsection
for our problem \eqref{eq:L2minVar}, leading to a stable
discretization, the task is to construct a Gelfand frame
for the Gelfand triple induced by $\CH = \Hs(\bbR^d)$.

%%%%%%%%%%%%%%%%%%%%%%%
To this end, in \cite{grohs1}, a Parseval frame 
$\Phi = (\varphi_\lambda)_{\lambda\in \Lambda}$
of ridgelets was constructed -- we need to reproduce it in some detail, in order to be able to derive a number of properties which will be indispensable to prove sparsity of the operator in this discretisation. The key to the construction is a certain set of functions $\psi_\jl \in L^2(\bbR^d)$, which form a partition of unity in the frequency domain, i.e.
\begin{align}\label{eq:partititon_unity}
	\parens*{\psi_\jl}_{j\in\bbN_0,\, \ell\in\{0,\ldots,L_j\}} \quad \text{such that} \quad \sum_{j=0}^\infty \sum_{\ell=0}^{L_j} \hat \psi_\jl^2 = 1.
\end{align}
\begin{definition}\label{def:psi_jl}
	To partition the angular component, we need a covering (approximately uniform) of the sphere $\bbSd$, which we choose according to the following construction for $\alpha = 2^{-j}$:
	\begin{align}
		\text{Choose } \{\vec s_\ell\in\bbSd\}_{\ell \in \{0,\ldots,L\}} \quad \text{such that} \quad
		\begin{cases}
			{\bigcup_{\ell=0}^{L}} B_\bbSd(\vec s_\ell,\alpha) = \bbSd, \\
			B_\bbSd \parens*{\vec s_\ell,\frac{\alpha}{3} } \text{ pairwise disjoint}.
		\end{cases} 
	\end{align}
	Here $B_\bbSd(\vec s,\alpha)$ is the open ball on the sphere of radius $\alpha$ in the geodesic metric (see \autoref{app:hypsphere}). This can be shown to imply $L\sim \parens{\frac{1}{\alpha}}^{d-1}$ (for details see \cite{borup} or \autoref{app:construction_sjl}), and thus $L_j\sim 2^{j(d-1)}$.
	
	Furthermore, set $\alpha_j:=2^{-j+1}$ and choose (smooth and bounded) window functions $W$, $W^{(0)}$, $V^{(\jl)}$, such that
	\begin{enumerate}
		\item $\supp W \subseteq \parens*{\frac 12, 2 }$,
		\item $\supp W^{(0)} \subseteq [0, 2)$,
		\item $\supp V^{(\jl)} \subseteq B_\bbSd \parens*{\vec s_\jl, \alpha_j}$,
		\item Lower bounds for all functions in a suitable subset, see \cite{grohs1}.
	\end{enumerate}
	From these properties, it can be shown that
	\begin{align*}%\label{eq:}
		\Phi(\vec \xi) := W^{(0)}(|\vec \xi|)^2 + \sum_{j\in\bbN_0}\sum_{\ell=0}^{L_j} W\parens*{2^{-j} |\vec \xi|}^2 \, V^{(\jl)} \!\parens[\bigg]{\! \frac{\vec \xi}{|\vec \xi|} \!}^2
	\end{align*}
	is bounded from above and below. Now, define
	\begin{align}\label{eq:def_psi_jl}
		\hat\psi_{0,0} (\vec \xi) := \frac{W^{(0)}(|\vec \xi|)}{\sqrt{\Phi(\smash{\vec \xi})\vphantom{\big()}}} \quad \text{and} \quad
		\hat\psi_\jl (\vec \xi) := \frac{W\parens*{2^{-j} |\vec \xi|} \, V^{(\jl)} \parens[\Big]{\!\frac{\vec \xi}{|\vec \xi|} \!}}{\sqrt{\Phi(\smash{\vec \xi})\vphantom{\big()}}}, \qquad j\ge 1, \, \ell = 0,\ldots,L_j.
	\end{align}
	Note that all $\psi_\jl$ are defined via their Fourier transforms, and that for notational convenience, we have set $L_0:=0$ to extend the indexing consistently to the function for $j=0$ as well. From these definitions, it is easy to check that \eqref{eq:partititon_unity} holds.
\end{definition}
\begin{definition}\label{def:phi_jlk}
	Using \autoref{def:psi_jl}, a Parseval frame for $L^2(\bbR^d)$ is defined by 
	\begin{align*}
		\varphi_\jlk = \, 2^{-\frac{j}{2}} T_{U_\jl \vec k}\,\psi_\jl, \quad
		j\in\bbN_0, \, \ell \in \{0,\ldots,L_j\}, \, \vec k \in \bbZ^d, 
	\end{align*}
	with $T$ the translation operator, $T_{\vec y}f(\cdot) := f(\cdot-\vec y)$, and $U_\jl:= \Rjl^{-1} D_{2^{-j}}$, where $\Rjl$ is the transformation introduced in \autoref{sec:wellposed}, and $D_a$ dilates the first component, $D_a \vec k := (a \, k_1, k_2,\ldots,k_d)^\top$. The rotation $\Rjl$ is arbitrary (to the extent that it is ambiguous, see \autoref{rem:rotation_ambiguous}) but fixed. Whenever possible, we will subsume the indices of $\varphi$ by $\lambda=(\jlk)$.
\end{definition}
We note that for a Parseval frame, the frame operator $S_\Phi=\bbI$, since
\begin{equation}
	\inpr{S_\Phi f, f} = \inpr{G^*Gf, f} = \inpr{Gf, Gf} = \norm{\inpr{\Phi,f}}_{\ell^2}^2 = \norm{f}^2 = \inpr{f,f},
\end{equation}
which implies $\wt\Phi=\Phi$.

With the ridgelet frame $\Phi$ in hand we go on to show that $\Phi$ is indeed a Gelfand frame for the
Gelfand triple $\parens*{\Hs, L^2(\bbR^d),(\Hs)'}$.
First, we need to find suitable sequence spaces $\CH_d$. 
To this end we introduce the diagonal preconditioning matrix
\begin{align}\label{eq:precond_matr}
	\VW_\lld = \left\{ \!\!\! \begin{array}{rl}
		0, & \lambda \neq \lambda', \\
		w(\lambda):= 1 + 2^j|\vec s \cdot \vec s_\jl|, & \lambda=\lambda',
	\end{array} \right.
\end{align}
%
%which depends on all indices $\jlk$, even though its entries do not depend on $\vec k$.
and define the weighted $\ell^2$-spaces
\[
	\CH_d := \ell^2_\VW(\Lambda): =\{\Vc\in \ell^2(\Lambda):\ \norm{\VW\Vc}_{\ell^2(\Lambda)}<\infty\}
\]
and the corresponding isomorphisms
\[
	D_{\ell^{2,\VW}}: \left\{ \begin{array}{rcl}
		\CH_d & \to & \ell^2(\Lambda), \\
		\Vc & \mapsto & \VW\Vc,
	\end{array}\right. \quad \text{and} \quad
	D_{\ell^2_\VW}^*: \left\{ \begin{array}{rcl}
		\ell^2(\Lambda) & \to & \CH_d' = \ell^2_{\VW^{-1}}(\Lambda), \\
		\Vc & \mapsto & \VW\Vc.
	\end{array}\right.
\]
\begin{theorem}\label{thm:GelfandFrame}
	The ridgelet frame $\Phi$ as constructed above constitutes a Gelfand frame for the 
	Gelfand triple $\parens*{\Hs,L^2(\bbR^d),(\Hs)'}$.
\end{theorem}
\begin{proof}
	Essentially, this has been shown in \cite{grohs1}, where it is observed that
	\begin{equation}\label{eq:RidgeletFrame}
		\norm{f}_{\Hs} \sim \norm{\inpr{\Phi,f}_\CH}_{\ell^2_\VW}.
	\end{equation}
	Using the fact that $\wt\Phi = \Phi$, we immediately infer
	boundedness of the operator $G_{\wt \Phi}:\ \Hs\to \ell^2_\VW(\Lambda)$.
	To show the boundedness of the operator $G_{\Phi}^* :\ \ell^2_\VW(\Lambda)\to \Hs$ we need
	to estimate the $\Hs$-norm of $\Phi\Vc$ in terms of the $\ell^2_\VW$-norm
	of $\Vc$. To see this, we first observe that, due to \eqref{eq:RidgeletFrame}, 
	we have  
	\[
		\norm{\Phi\Vc}_{\Hs} \lesssim \norm{\inpr{\Phi,\Phi}_{L^2} \Vc}_{\ell^2_\VW}.
	\]
	In order to arrive at the desired bound it remains to note that the matrix operator 
	$\inpr{\Phi,\Phi}_{L^2} \colon \ell^2_\VW(\Lambda)\to \ell^2_\VW(\Lambda)$ is bounded, which is a simple consequence
	of the frequency support properties of the frame elements.
\end{proof}
\begin{theorem}\label{thm:stabdisc}
	With $\Phi$ the ridgelet system and $A$ the differential operator defined above, consider the (infinite) matrix 
	\begin{align}\label{eq:sys_matr_cond}
		\VF:= \VW^{-1}\inpr{A\,\Phi, A\,\Phi}_{L^2} \VW^{-1}.
	\end{align}
	Then the operator $\VF : \ell^2(\Lambda) \to
	\ell^2(\Lambda)$ is bounded as well as boundedly invertible on its range 
	$\ran (\VF) = \ran\parens*{(D_{\ell^2_\VW})^{-1}G_\Phi}$.
\end{theorem}
\begin{proof}
	This is a direct consequence of \autoref{lem:stable}, \autoref{cor:welldef}
	and \autoref{thm:GelfandFrame}.
\end{proof}
In summary, we have achieved (i) above, namely a stable discretization of the operator equation \eqref{eq:L2minVar}. In order to make use of the convergence results presented 
in \autoref{sec:discretization}
we also need to derive a matrix representation of the projector
$\VP$ defined in \autoref{sec:proj}.
Using the fact that for our ridgelet frame construction $\Phi$,
the dual frame coincides with the primal frame, e.g., $\wt\Phi
= \Phi$, it is easy to see that
\begin{equation}\label{eq:PMatRep}
	\VP = \VW\inpr{ \Phi, \Phi}_{L^2}\VW^{-1}.
\end{equation}

\subsection{Remarks on the construction}

\begin{remark}
	By the support properties of $V^{(\jl)}$ and $W$, we see that
	\begin{align*}
		\supp \hat \psi_\jl 
		&\subseteq P_\jl:= \set[\Big]{
		\vec \xi \in \bbR^d}{ 2^{j-1}<|\vec \xi| < 2^{j+1}, \, \smash{\frac{\vec \xi}{|\vec \xi|}} \in B_\bbSd(\vec s_\jl, \alpha_j)}, \quad j\ge 1, \\
		\supp \hat \psi_{0,0}
		&\subseteq P_{0,0}:=\set*{\vec \xi \in \bbR^d}{ |\vec \xi| < 2}.
	\end{align*}
	For several reasons, we will need to know for which $j$ and $\ell$ the intersections $P_\jl \cap P_\jld$ are non-empty if $\jld$ is fixed. For example, to prove the sparsity of the ridgelet discretisation of the transport operator $A$ introduced in \autoref{sec:wellposed}, we will have to consider a sum of terms involving the $\hat \psi_\jl$ over all parameters as in \eqref{eq:def_sparse} -- only through a criterion of the above-mentioned form will we be able to bound the sum. Luckily, it is straightforward to check that a non-empty intersection necessarily implies
	\begin{align*}%\label{eq:ind_impl}
		|j-j'|\le 1 \quad \text{and} \quad \dist_\bbSd(\vec s_\jl, \vec s_\jld) \le \alpha_j + \alpha_{j'} \stackrel{(*)}\le 3\alpha_{j'},
	\end{align*}
	where $(*)$ makes use of the first condition. Often, it will turn out to be convenient to cast these conditions into an inclusion, in other words,
	\begin{align*}%\label{eq:ind_incl}
		\set*{(\jl)}{ P_\jl\cap P_\jld \neq \emptyset } \subseteq \set*{(\jl)}{ |j-j'|\le 1,  \dist_\bbSd(\vec s_\jl, \vec s_\jld) \le 3\alpha_{j'} }.
	\end{align*}
	This knowledge lets us revisit the function $\Phi$ in \autoref{def:psi_jl} -- in particular, for $\vec \xi \in P_\jl$, the sum consists of only the terms ``neighbouring" $j$ and $\ell$,
	\begin{align}\label{eq:Phi_neighbour}
		\Phi(\vec \xi) = \sum_{\substack{ \forceheight{j'\in\bbN_0:} \\
			\forceheight[\vec b]{|j-j'|\le 1}	}	} \,\,
		\sum_{\substack{\forceheight{\ell' \in \{0,\ldots, L_{j'}\}:} \\
			\forceheight[\vec b]{\dist_\bbSd(\vec s_\jl, \vec s_\jld)\le 3\alpha_j}	}	} \hspace{-0.5cm}
		W\parens*{2^{-j'} |\vec \xi|}^2 \, V^{(\jld)} \!\parens[\bigg]{ \frac{\vec \xi}{|\vec \xi|} }^2.
	\end{align}
	%
%	The indices of the sum are primed twice to avoid confusion with the always fixed indices $\jld$ (as mentioned above, and used throughout the main proof).
	Of course, the above describes the case $j>2$ -- otherwise, the term $W^{(0)}(|\vec \xi|)$ would also appear in the sum.
\end{remark}
\begin{remark}\label{rem:rotation_ambiguous}
	In dimensons $d>3$, the rotation $\Rs$ turning $\vec s$ into $\vec e_1$ is no longer unique, although all other possible choices must satisfy
	\begin{align*}%\label{eq:}
		\wt \Rs = \begin{pmatrix}
			1 & 0 \\ 0 & R
		\end{pmatrix} \Rs,
	\end{align*}
	where $R\in\mathrm{SO}(d-1)$. Due to this ambiguity, the Lipschitz condition
	\begin{align}\label{eq:rot_lipschitz}
		\norm{R_{\vec{s}} - R_{\vec{s}'}} \lesssim \dist_\bbSd(\vec{s},\vec{s}')
	\end{align}
	will not hold in general (if, for example, $\vec s = \vec s'$ and the matrix $R$ above contains two reflections). However, it is possible to choose such an $\Rs$ for fixed $R_{\vec{s}'}$ (as proved in \autoref{lem:rot_lipschitz}) -- this suffices for our purposes, since in essence, we do not need this Lipschitz condition globally, but only in a neighbourhood of $\vec s$, where the ambiguity is irrelevant. 
\end{remark}
In the course of the proof (of sparsity of the Ridgelet discretisation), we need to control the derivatives of $\hat \psi_\jl$ under a pullback related to the above-mentioned $U_\jl$. We formulate this as an assumption that has to satisfied when choosing the window functions.
\begin{assumption}\label{assump:psi_smooth}
	The window functions in \autoref{def:psi_jl} are chosen in such a way, that \emph{for any} rotation $\Rjl$ (taking $\vec s_\jl$ to $\vec e_1$), the pullbacks under the transformation $U_\jl^{-\top}=\Rjl^{-1}D_{2^j}$,
	\begin{align}\label{eq:psi_trafo}
		\hat \psi_{(\jl)} (\vec \eta):= \hat\psi_\jl(U_\jl^{-\top}\vec \eta) = \frac{W\parens*{2^{-j} |D_{2^j}\vec \eta|} \, V^{(\jl)} \parens[\Big]{\frac{D_{2^j}\vec \eta}{|D_{2^j}\vec \eta|} }}{\sqrt{\Phi(\smash{U_\jl^{-\top}\vec \eta})\vphantom{\big()}}},
	\end{align}
	have bounded derivatives \emph{independently} of $j$ and $\ell$. Thus, for all $n$ up to an upper bound $N$ dependent on the differentiability of the window functions (or possibly for all $n\in\bbN$ if the window functions are $\CC^\infty$), we have the estimate
	\begin{align*}%\label{eq:}
		\norm*{\hat \psi_{(\jl)}}_{\CC^n} \le\beta_{n}.
	\end{align*}
	In \autoref{lem:suitable_window}, we show that this assumption can be satisfied with a reasonable (and still quite flexible) choice of window functions.
\end{assumption}

\section{Compressibility}\label{sec:sparsity}

In this section, we show the main result, that the relevant bi-infinite matrices ($\VF$ and $\VP$) appearing in \autoref{alg:modsolve} are in fact compressible.

\subsection{Preliminary Considerations}
In general, compressibility is difficult to verify directly. Instead we
use the following notion of sparsity for a (possible bi-infinite) matrix
$\VA$:
\begin{definition}
	Let $p>0$.
	A matrix $\VA= \parens*{a_{\lambda,\lambda'}}_{\lambda\in \Lambda,\lambda'\in \Lambda'}$
	is called \emph{$p$-sparse} if
	\begin{align}\label{eq:def_sparse}
		\norm{\VA}_{\ell^p(\Lambda)\to \ell^p(\Lambda)} :=
		\max\parens[\bigg]{
		\sup_{\lambda'\in \Lambda'}
		\sum_{\lambda\in \Lambda}|a_{\lambda,\lambda'}|^p,
		\sup_{\lambda\in \Lambda}
		\sum_{\lambda'\in \Lambda'}|a_{\lambda,\lambda'}|^p
		}^{\frac 1p} <\infty.
	\end{align}
\end{definition}
\begin{proposition}\label{prop:lpcompress}
	Assume that $\VA$ is $p$-sparse for $0<p<1$. Then 
	$\VA$ is $\frac{1}{2}\parens*{\frac 1p - 1}$-compressible.
\end{proposition}
To prove this result, we require the following version of Schur's test (which is \cite[Thm. 5.2]{halmos} for a discrete measure):
\begin{theorem}\label{lem:Schur}
	Let $\VA:= (a_{\lambda,\lambda'})_{\lambda,\lambda'\in \Lambda}$
	be an operator.
	Then the following holds:
	\begin{align*}
		\norm{\VA}_{\ell^2(\Lambda)\to \ell^2(\Lambda)}
		\le \parens[\bigg]{\sup_{\lambda\in \Lambda}\sum_{\lambda'\in \Lambda'}
		|a_{\lambda,\lambda'}|}^{\frac 12} \parens[\bigg]{\sup_{\lambda'\in \Lambda'}
		\sum_{\lambda\in \Lambda}|a_{\lambda,\lambda'}|}^{\frac 12}
	\end{align*}
\end{theorem}
\begin{proof}[Proof of \autoref{prop:lpcompress}]
	Note that by assumption each column $\Va_\lambda$ of $\VA$
	has $\ell^p$ norm bounded by
	\begin{align*}
		\norm{\Va_\lambda}_p\le \norm{\VA}_{\ell^p(\Lambda) \to \ell^p(\Lambda)}.
	\end{align*}
	%
	%with the implicit constant independent of $\lambda\in\Lambda$.
	%Without loss of generality we may assume that $\Lambda = \bbN$.
	This means that for each $k\in\bbN$ and for
	some summable sequence $\alpha_k$, we may keep only the $\alpha_k 2^k$ largest coefficients of
	the column vector $\Va_\lambda$, which gives the approximation
	$\VA^{[k]}$ consisting of columns $\Va_\lambda^{[k]}$. An immediate observation is, that it would be non-sensical to let $\alpha_k$ decay quicker than $2^{-k}$, since then the number of approximating coefficients would decrease in each step. While it is possible to let $\alpha_k$ decay like $2^{-k(1-\epsilon)}$, this would impact the achieved compressibility (see below), and so we choose a sequence whose inverses grow at most polynomially (say, $\alpha_k = k^{-2}$).
	
	To compute the error of this approximation,
	\begin{align*}
		\norm*{\VA - \VA^{[k]}}_{\ell^2(\Lambda)\to \ell^2(\Lambda)},
	\end{align*}
	denote by $\Va_\lambda^* := \parens*{(a_\lambda^*)_i}_{i\in \bbN}$
	%%
	%\begin{align*}
	%	\Va_\lambda^* := \parens*{(a_\lambda^*)_i}_{i\in \bbN}
	%\end{align*}
	%%
	the non-increasing rearrangement of $\Va_\lambda$. The defining condition for weak-$\ell^p$ spaces is, that $\Va_\lambda \in \ell^p$ implies $(a_\lambda^*)_i\lesssim i^{-\frac 1p}$. In order to be able to apply Schur's Lemma, we use this fact to estimate the (square of the) first factor,
	\begin{align*}
		\sup_{\lambda\in \Lambda}\sum_{\lambda'\in \Lambda}
		|a_{\lambda,\lambda'}- a_{\lambda,\lambda'}^{[k]}|
		= \sup_{\lambda\in \Lambda}\sum_{i \geq \alpha_k2^k}
		|(a_{\lambda}^*)_i|
		\lesssim \sum_{i \geq \alpha_k 2^k}
		i^{-\frac 1p}.
	\end{align*}
	To continue, choose $\ell\in\bbN$ such that $2^{\ell-1}\le \alpha_k 2^k \le 2^\ell$ and let
	\begin{align*}%\label{eq:}
		F_m:=\sum_{i=2^{m-1}}^{2^m} i^{-\frac 1p} \le (2^m-2^{m-1})(2^{m-1})^{-\frac 1p}.
	\end{align*}
	Consequently, using $p<1$, we see that
	\begin{align*}%\label{eq:}
		\sum_{i \geq \alpha_k 2^k} i^{-\frac{1}{p}}
		&\le \sum_{i \geq 2^{\ell-1}} i^{-\frac 1p} = \sum_{m=\ell}^\infty F_m \le \sum_{m=\ell}^\infty 2^{-(m-1)(\frac 1p-1)} =  2^{-(\ell-1)(\frac 1p-1)} \sum_{m=0}^\infty 2^{-(m-1)(\frac 1p-1)} \\
		&= \parens{\frac 1{2^\ell}}^{\frac 1p -1} \frac{2^{\frac 1p-1}}{1-2^{-(\frac 1p-1)}} \lesssim \parens*{\alpha_k 2^k}^{-(\frac 1p -1)}.
	\end{align*}
	On the other hand, since all sequences $\Vx$ satisfy $\norm{\Vx}_q \le \norm{\Vx}_p$ for all $0<p\le q\le\infty$, we have (with $q=1$)
	\begin{align*}
		\sup_{\lambda'\in \Lambda}\sum_{\lambda\in \Lambda} |a_{\lambda,\lambda'}- a_{\lambda,\lambda'}^{[k]}|
		\le 2 \sup_{\lambda'\in \Lambda}\sum_{\lambda\in \Lambda} |a_{\lambda,\lambda'}|
		= 2 \sup_{\lambda'\in \Lambda} \norm{\Va_{\lambda'}\!}_1 \le 2 \sup_{\lambda'\in \Lambda} \norm{\Va_{\lambda'}\!}_p
		\le 2\norm{\VA}_{\ell^p(\Lambda)\to \ell^p(\Lambda)}<\infty.
	\end{align*}
	Applying Schur's Lemma we get that
	\begin{align*}
		\norm{\VA - \VA^{[k]}}_{\ell^2(\Lambda)\to \ell^2(\Lambda)}
		\lesssim \parens*{\alpha_k^{-1} 2^{-k}}^{\frac 12 (\frac 1p-1)} =: C_k.
	\end{align*}
	We see that for any $\sigma<\frac{1}{2}\parens*{\frac 1p - 1}$, the sequence
	$(2^{\sigma k}C_k)_{k\in \bbN}$ is summable, since the polynomial growth in $\alpha_k^{-1}$ does not affect the exponential decay of $2^{-\parens{\frac{1}{2}(\frac 1p - 1)-\sigma}k}$ (up to a constant). This proves the compressibility.
\end{proof}

\subsection{Sparsity of $\VF$}

Having introduced the concepts of compressibility and sparsity in the last section, % axel: keep label current if not "last section" anymore
we now want to show sparsity of the ridgelet discretisation of the transport operator $A$ introduced in \autoref{sec:wellposed}. By \autoref{prop:lpcompress}, this will prove compressibility of the ridgelet discretisation of this operator.

\begin{theorem}\label{th:sparse_stiff}
	We consider the frame $\Phi = (\varphi_\lambda)_{\lambda\in\Lambda}$ for $L^2(\bbR^d)$ (see \autoref{def:phi_jlk}), satisfying \autoref{assump:psi_smooth} for $2n$ with $\frac d2<n\in\bbN$, and choose $p\in\bbR$ such that $1 > p > \frac{d}{2n}$. Furthermore, we introduce the differential operator $A: u \mapsto \vec s \cdot \nabla u  + \kappa u$ with fixed $\vec s \in \bbSd$, where the absorption coefficient $\kappa$ has a decomposition $\kappa=\gamma + \kappa_0$ with constant $\gamma>0$, and $\kappa_0\ge 0$ satisfying $\kappa_0, \hat\kappa_0 \in L_\infty(\bbR^d)$. Finally, we demand the existence of $r_0, c_0 >0$, such that the decay condition
	\begin{align}\label{eq:kappa_decay}
		\abs*{\hat \kappa_0 (\vec \xi)} \le \frac{c_0}{|\vec \xi|^q} \qquad
		\forall \, \vec \xi\in\bbR^d \colon |\vec \xi|\ge r_0,
	\end{align}
	is fulfilled for a fixed $q>2d+2n+\frac 32+\frac{d-1}{p}$. Then the preconditioned stiffness matrix $\VF$, see \eqref{eq:sys_matr_cond}, is $p$-sparse in this frame -- in other words,
	\begin{align}\label{eq:sparse_stiff}
		\norm*{ \VF }_{\ell^p(\Lambda)\to\ell^p(\Lambda)} =
		\norm[\Big]{ \VW^{-1}\inpr{A\,\Phi, A\,\Phi}_{L^2} \VW^{-1} }_{\ell^p(\Lambda)\to\ell^p(\Lambda)}  < \infty.
	\end{align}
\end{theorem}
\begin{remark}
	As we have seen in \autoref{prop:lpcompress}, the smaller $p$, the better the compressibility. The theorem is formulated in a way that $p$ is chosen according to the restrictions imposed by $d$ and $n$ -- however, since it is possible to construct window functions of arbitrary smoothness (and thus arbitrarily smooth $\hat\psi_\jl$), the limiting factor for $p$ then becomes the decay rate of $\hat \kappa_0$. In the case that $\hat \kappa_0$ decays faster than any polynomial (say, exponentially), arbitrarily small $p$ can be achieved (for infinitely smooth $\hat \psi_\jl$) -- of course at the cost of exploding constants.
\end{remark}
Before we begin with the proof, as a service to the reader, we collect a few results on technical details, which we have moved to  Appendix \ref{sec:geometric}.
\begin{proposition}[\autoref{lem:U_jl}]
	Let $w(\lambda) = 1+2^j |\vec s \cdot \sjl|$, $U_\jl= \Rjl^{-1} D_{2^{-j}}$ and $U_\jld= \Rjld^{-1} D_{2^{-j}}$ with arbitrary $\Rjld$ such that \eqref{eq:rot_lipschitz} holds for $\vec s = \sjl$ and $\vec s' = \sjld$. Then we have the estimates
	\begin{align}
		\abs*{U_\jl^{-1}\vec s} &\le w(\lambda), \quad \text{and} \quad \abs*{U_\jl^{-1}\vec s } \lesssim \max(2^{j-j'},1)\parens*{w(\lambda') + 2^{j'} \dist_\bbSd (\sjl, \sjld) },
		\label{eq:inv_Ujls}
		\intertext{%
	as well as
		}
		\norm*{U_\jl^{-1} U_\jld} &\lesssim \max(2^{j-j'},1) + 2^j \dist_\bbSd (\sjl, \sjld). \label{eq:U_lld}
	\end{align}
\end{proposition}
\begin{proposition}[\autoref{prop:bounding_cylinder}]
	For $j\ge 1$, the transformation $U_\jl^\top$ takes the ``frequency tiles" $P_\jl$ back into a bounded set around the origin (illustrated in \autoref{fig:UjlPjl}),
	\begin{align}\label{eq:trafo_P_jl}
		U_\jl^\top P_\jl \subseteq B_{\bbR^d}(0,5) \qquad \text{and} \qquad U_\jl^\top(P_\jl^m)\subseteq B_{\bbR^d}(0,5+2^m),
	\end{align}
	where $P_\jl^m$ is again the Minkowski sum $P_\jl + B_{\bbR^d}(0,2^m)$.
	
	Additionally, we can calculate the opening angle of the cone containing $P_\jl^m$ as follows,
	\begin{align}\label{eq:est_alpha_j_m}
		\alpha_j^m  = \alpha_j+\arcsin \parens[\bigg]{\frac{2^m}{2^{j-1}}} \le c_\omega 2^{m-j},
	\end{align}
	as long as $j\ge m+1$, where $c_\omega\le \pi+2$ (illustrated in \autoref{fig:alpha_m}).
\end{proposition}
\begin{proposition}[\autoref{lem:intersection}]
	Let $\jld$ as well as $\mmd$ be fixed and denote $m_>:= \max(m,m')$, then we have the following inclusion for the set of parameters that can yield a non-empty intersection with $P_\jld^{m'}$,
	\begin{multline}\label{eq:ind_jl_incl}
		\set*{(\jl)}{P_\jl^m\cap P_\jld^{m'} \neq \emptyset} \subseteq \\ \bigcup_{j=0}^{m_>+2} \curly*{j} \times  \curly*{0\le \ell \le L_j } \cup \bigcup_{j\ge m_>+3} \curly*{j} \times \set*{\ell}{ \dist_\bbSd (\vec s_\jl, \vec s_\jld) \le 5c_\omega 2^{m_>-j'}},
	\end{multline}
\end{proposition}
\begin{proof}[Proof of \autoref{th:sparse_stiff}]
	Due to symmetry, we are able to express \eqref{eq:sparse_stiff} without taking the maximum (cf. \eqref{eq:def_sparse}),
	\begin{align}\label{eq:sparse_stiff_expanded} 
	\begin{split}
		\norm*{\VF}_{\ell^p(\Lambda)\to\ell^p(\Lambda)}^p 
		&= \sup_{\lambda'\in\Lambda} \sum_{\lambda\in\Lambda} \abs[\Big]{w(\lambda)^{-1}w(\lambda')^{-1}\inpr{A\,\varphi_{\lambda}, A\,\varphi_{\lambda'}}_{L^2}}^p \\ 
		&= \sup_{\lambda'\in\Lambda}  \sum_{j\in\bbN_0 \vphantom{\scriptscriptstyle \bbZ^d}} \sum_{\ell=0 \vphantom{\scriptscriptstyle \bbZ^d}}^{L_j} \sum_{\smash{\vec k} \in \bbZ^d}
		\abs[\Big]{w(\lambda)^{-1}w(\lambda')^{-1}\inpr*{A\,\varphi_\jlk, A\,\varphi_\jlkd}_{L^2} }^p < \infty.
	\end{split}
	\end{align}
	\noindent {\bfseries Step 1 -- Transforming the integral:}
	Recalling the definition of the $\varphi_\lambda$, we compute
	\begin{align}
		\VF_\lld
		&=w(\lambda)^{-1}w(\lambda')^{-1}\inpr{A\,\varphi_{\lambda}, A\,\varphi_{\lambda'} }_{L^2} = w(\lambda)^{-1}w(\lambda')^{-1}\inpr{\wh{A\,\varphi_{\lambda}}, \wh{A\,\varphi_{\lambda'}} }_{L^2} \notag \\
		&= \frac{2^{-\frac{j+j'}{2}}}{w(\lambda)w(\lambda')} \int \overline{\CF\parens*{\vec s \cdot \nabla \rcopywidth{\psi_\jl}{\psi_\jld}(\vec x - \rcopywidth{U_\jl}{U_\jld} \rcopywidth{\vec k}{\vec k'}) + \kappa(\vec x) \rcopywidth{\psi_\jl}{\psi_\jld}(\vec x - \rcopywidth{U_\jl}{U_\jld} \rcopywidth{\vec k}{\vec k'}) } (\vec \xi)} \cdot \ldots \notag \\
		&\phantom{{}= \smash{\frac{2^{-\frac{j+j'}{2}}}{w(\lambda)w(\lambda')} \int{}}} %alignment of integrands
		\mathllap{\ldots \cdot{}} \CF\parens*{\vec s \cdot \nabla \psi_\jld(\vec x - U_\jld \vec k') + \kappa(\vec x) \psi_\jld(\vec x - U_\jld \vec k')} (\vec \xi) \, \d\vec \xi \notag\\
		&= \frac{c_\CF 2^{-\frac{j+j'}{2}}}{w(\lambda)w(\lambda')} \int
		\overline{\parens[\Big]{
			(\vec s \cdot \vec \xi)\, \rcopywidth{\hat \psi_\jl}{\hat \psi_\jld}(\vec \xi) + \bracket*{\rcopywidth{\hat \psi_\jl}{\hat \psi_\jld} * \rcopywidth{M_{\lambda}}{M_{\lambda'}}\hat\kappa } (\vec \xi)
			}} \cdot \ldots \\
		&\phantom{\smash{{}=\frac{2^{-\frac{j+j'}{2}}}{w(\lambda)w(\lambda')} \int}} \mathllap{\ldots\cdot{}}\, % alignment of integrands, "\," for effect of \overline{}
		\parens[\Big]{
			(\vec s \cdot \vec \xi) \, \hat \psi_\jld (\vec \xi) + \bracket*{\hat \psi_\jld * M_{\lambda'}\hat\kappa } (\vec \xi)
			}
		\exp\parens*{2\pi\ii \vec \xi \cdot (U_\jl \vec k - U_\jld \vec k')} \d\vec \xi, \label{eq:stiff_matr_elem_pre_trafo}
	\end{align}
	where $c_\CF=(2\pi\ii)^2$ and $\bracket*{M_{\lambda}\hat\kappa}(\vec \xi) = \hat\kappa(\vec \xi) \exp(2\pi\ii \vec \xi \cdot U_\jl \vec k) = \CF\parens*{\kappa(\vec x + U_\jl \vec k)}(\vec \xi)$ is the modulation operator corresponding to a shift of $U_\jl\vec k$ according to our definition of the Fourier transform.
	
	The transformation $U_\jl$ modifying $\vec k$ in the exponential function makes summing $\vec k$ difficult, and therefore, we will transform all the integral by $\vec \xi = U_\jl^{-\top} \vec\eta$ -- introducing a factor $2^j$ from the determinant of the Jacobian and
	yielding the exponent
	\begin{align*}%\label{eq:}
		2\pi\ii \vec \eta \cdot (\vec k - U_\jl^{-1}U_\jld \vec k') = 2\pi\ii \vec \eta \cdot (\vec k - U^\jl_\jld \vec k'),\quad \text{where}\quad U^\jl_\jld:=U_\jl^{-1}U_\jld.
	\end{align*}
	
	Additionally, we observe that the Fourier transform of $\kappa$ splits as follows,
	$\hat\kappa(\vec \xi) = (\wh{\gamma +\kappa_0 \vphantom{k}})(\vec \xi) =\gamma \delta(\vec \xi) + \hat \kappa_0(\vec \xi)$. We expand all products and rearrange, observing also that $\vec s \cdot \vec \xi = \vec s \cdot U_\jl^{-\top}\eta = U_\jl^{-1}\vec s \cdot \eta$ is a real number. Thus,
	\begin{align}
		\VF_\lld
		&= \frac{c_\CF 2^{\frac{j-j'}{2}}}{w(\lambda)w(\lambda')} \int
		\Bigl(
			\parens*{U_\jl^{-1}\vec s \cdot \eta + \gamma}^2 \underbrace{\hat \psi_{(\jl)} (\vec \eta) \hat \psi_{(\jld)} (\wt U_\jld^\top U_\jl^{-\top} \vec \eta)}_{=:h^{00}_\lld(\vec \eta)} + \ldots \\
		&\phantom{\smash{{}=\frac{2^{-\frac{j+j'}{2}}}{w(\lambda)w(\lambda')} \int \Bigl(}} \mathllap{\ldots\cdot{}} % alignment of integrands
			\parens*{U_\jl^{-1}\vec s \cdot \eta + \gamma} \, \underbrace{\psi_{(\jl)} (\vec \eta) \bracket*{\hat \psi_\jld * M_{\lambda'} \hat\kappa_0}  (U_\jl^{-\top} \vec \eta)}_{=:h^{0*}_\lld(\vec \eta)} + \ldots \\
		&\phantom{\smash{{}=\frac{2^{-\frac{j+j'}{2}}}{w(\lambda)w(\lambda')} \int \Bigl(}} \mathllap{\ldots\cdot{}} % alignment of integrands
			\parens*{U_\jl^{-1}\vec s \cdot \eta + \gamma} \underbrace{\bracket*{\hat \psi_\jl * M_{\lambda} \hat\kappa_0}  (U_\jl^{-\top} \vec \eta) \, \psi_{(\jld)} (\wt U_\jld^\top U_\jl^{-\top} \vec \eta)}_{=:h^{*0}_\lld(\vec \eta)} + \ldots \\
		&\phantom{\smash{{}=\frac{2^{-\frac{j+j'}{2}}}{w(\lambda)w(\lambda')} \int \Bigl(}} \mathllap{\ldots\cdot{}} % alignment of integrands
			\underbrace{\bracket*{\hat \psi_\jl * M_{\lambda} \hat\kappa_0}  (U_\jl^{-\top} \vec \eta) \, \bracket*{\hat \psi_\jld * M_{\lambda'} \hat\kappa_0}  (U_\jl^{-\top} \vec \eta)}_{=:h^{**}_\lld(\vec \eta)}
		\Bigr)
		\exp\parens*{2\pi\ii \vec \eta \cdot (\vec k - U^\jl_\jld \vec k')} \d\vec \eta, \hspace{0.8cm} \label{eq:stiff_matr_elem_trafo}
	\end{align}
	where we used the representation of $\hat\psi_\jl$ from \autoref{assump:psi_smooth} -- which holds for arbitrary rotations $\tilRjld$ taking $\vec s_\jld$ to $\vec e_1$. We choose $\tilRjld$ in $\wt U_\jld:=\tilRjld^{-1} D_{2^{-j}}$ in such a way that \eqref{eq:rot_lipschitz} holds for $\vec s = \vec s_\jl$ and $\vec s' = \vec s_\jld$, which is possible due to \autoref{lem:rot_lipschitz}. Unsurprisingly, we set $\wt U^\jl_\jld:=U_\jl^{-1} \wt U_\jld$.
	
	It should be noted that $h$-terms does not depend on $\vec k, \vec k'$ -- however, we have chosen this notation for reasons of readability (as well as uniformity with the $Y$- and $Z$-terms, which appear in Step 2 and 7, respectively).
	
	For each of the terms in \eqref{eq:stiff_matr_elem_trafo}, we have to show that the sum over all parameters in \eqref{eq:sparse_stiff_expanded} is finite -- which we will do for $\vec k$ first, then for $\ell$ and finally for $j$. \par\vspace{\baselineskip}
	
	\noindent {\bfseries Step 2 -- Integration by parts:}
	Even though the exponent is purely imaginary, we cannot estimate the exponential function by one, as we would then sum constants in $\vec k$.  However, a simple calculation shows $\Delta_{\vec \eta} \exp(2\pi\ii \vec \eta \cdot \vec y) = - (2\pi)^2|\vec y|^2 \exp(2\pi\ii\vec \eta \cdot \vec y)$, which entails
	\begin{align}\label{eq:laplace_exp}
		\Delta_{\vec \eta} \exp\parens*{2\pi\ii \vec \eta \cdot ( \vec k- U^\jl_\jld \vec k')} = - (2\pi)^2 \abs*{\vec k- U^\jl_\jld \vec k'}^2 \exp\parens*{2\pi\ii \vec \eta \cdot ( \vec k- U^\jl_\jld \vec k')}.
	\end{align}
	Applying Green's second identity iteratively, we will use this to generate a denominator of sufficient power to be summed over all $\vec k \in \bbZ^d$ -- on the other hand, this forces us to estimate the derivatives of the remaining factors of the integrands. All differential operators will be with respect to $\vec \eta$, which we will not indicate anymore in the following.
	
	The $h^{00}_\lld$ is unproblematic because of its compact support, however, for the other funcitons, their unbounded support (as a superset of $\supp \hat \kappa_0$) means that the boundary integral does not vanish trivially. Nevertheless, we can always shift the derivatives of the convolution away from $\hat \kappa$ and so the vanishing of the boundary term can be seen by applying Green's second identity to the domain $B_{\bbR^d}(0,R)$ and exploiting the decay of $\hat \kappa$ as $R\to\infty$ (bearing in mind that $\lld$ are fixed for this consideration). Due to the growing surface of $B_{\bbR^d}(0,R)$, this requires a decay $q>d-1$ of $\hat \kappa_0$, which is satisfied by our assumption.
	
	Thus, for $\vec k \neq U^\jl_\jld\vec k'$,
	\begin{multline}
		\VF_\lld
		= \frac{c_\CF 2^{\frac{j-j'}{2}}}{w(\lambda)w(\lambda')} \frac{(-1)^n(2\pi)^{-2n}}{\abs*{\vec k - U^\jl_\jld \vec k'}^{2n}} \int
		\Bigl[
			\Delta^n \parens[\Big]{\parens*{U_\jl^{-1}\vec s \cdot \eta + \gamma}^2 h^{00}_\lld(\vec \eta)} +
			\Delta^n \parens[\Big]{\parens*{U_\jl^{-1}\vec s \cdot \eta + \gamma} \, h^{0*}_\lld(\vec \eta)} + \ldots \\
		\shoveright{\ldots+
			\Delta^n \parens[\Big]{\parens*{U_\jl^{-1}\vec s \cdot \eta + \gamma} h^{*0}_\lld(\vec \eta)} +
			\Delta^n \parens[\Big]{h^{**}_\lld(\vec \eta)}
		\Bigr]
		\exp\parens*{2\pi\ii \vec \eta \cdot (\vec k - U^\jl_\jld \vec k')} \d\vec \eta} \\ 
		\shoveleft{\smash[b]{\phantom{\VF_\lld} =: \abs*{\vec k - U^\jl_\jld \vec k'}^{-2n} \parens[\Big]{Y^{00}_\lld+Y^{0*}_\lld+Y^{*0}_\lld+ Y^{**}_\lld}.}}
		\label{eq:stiff_matr_elem_part_int} \\[-\baselineskip]  % linebreak necessary for \shoveleft above to work
	\end{multline}\vspace{0.1cm}
	
	\noindent {\bfseries Step 3 -- Estimating the Derivatives:}
	Before we can deal with the derivatives of the $h$-terms, we have to disentangle the derivatives of $U_\jl^{-1}\vec s \cdot \eta$ from them. Computing $\nabla \parens*{U_\jl^{-1}\vec s \cdot \vec \eta + \gamma} = U_\jl^{-1} \vec s$, a simple induction shows
	\mathtoolsset{showonlyrefs=false} % create tags, even though some in the middle are only implicitly referenced by "from eq:first--eq:last"
	\begin{align}\label{eq:laplace_one_kappa}
		\Delta^n \parens*{(U_\jl^{-1}\vec s \cdot \vec \eta + \gamma) \, h(\vec \eta) } 
		&= (U_\jl^{-1}\vec s \cdot \vec \eta + \gamma) \Delta^n h(\vec \eta) + 2n \, U_\jl^{-1}\vec s \cdot \nabla\parens*{\Delta^{n-1} h(\vec \eta)},
	\end{align}
	and inserting $\tilde h(\vec \eta) := (U_\jl^{-1}\vec s \cdot \vec \eta +\gamma) \, h(\vec \eta)$ into this formula also yields
	\begin{multline}\label{eq:laplace_no_kappa}
		\Delta^n \parens[\Big]{ \parens*{U_\jl^{-1}\vec s \cdot \vec \eta +\gamma}^2 h(\vec \eta) }
		= \parens*{U_\jl^{-1}\vec s \cdot \vec \eta + \gamma}^2 \Delta^n h(\vec \eta)
		+ 4n \parens*{U_\jl^{-1}\vec s \cdot \vec \eta + \gamma} \parens[\Big]{U_\jl^{-1} \vec s \cdot \nabla\parens*{\Delta^{n-1} h(\vec \eta)}\!} +\ldots \hspace{0.8cm}\\
		\phantom{{}={}} \ldots + 2n \abs*{U_\jl^{-1} \vec s}^2 \Delta^{n-1} h(\vec \eta)
		+ 4n(n-1)\, U_\jl^{-1} \vec s \cdot \parens[\Big]{ \Dpp{2}{\eta_s}{\eta_t}\parens*{\Delta^{n-2} h(\vec \eta)}}_{s,t=1}^d U_\jl^{-1} \vec s.
	\end{multline}
	Alternatively, both of these formulas can be obtained by applying \eqref{eq:prodrule} -- the product rule for the Laplacian. This product rule is also the tool to obtain the following estimate, see \autoref{cor:est_prodrule},
	\begin{align}\label{eq:est_prodrule}
		\abs*{\bracket*{\Delta^n \parens*{fg}}(\vec \eta)} \le (4d)^n |f(\vec \eta)|_{\CC^{2n}} \, |g(\vec \eta)|_{\CC^{2n}} \le (4d)^n \norm{f}_{\CC^{2n}} \norm{g}_{\CC^{2n}},
	\end{align}
	where $ |f(\vec \eta)|_{\CC^{2n}}=\max_{0\le r\le 2n} |f^{(r)}(\vec \eta)|$ is the maximum of all derivatives up to order $2n$ of $f$ at $\vec \eta$. We will use this to estimate the derivatives of the $h$-terms.
	
	The last ingredient of this step is an estimate of the derivatives of the convolution terms, proved in \autoref{lem:deriv_conv},
	\begin{align}\label{eq:deriv_conv}
		\abs[\Big]{ \Dpi{\alpha}{\vec \eta} \parens*{ (f*g)(U\vec \eta) } } \le \norm[\Big]{\Dpi{\alpha}{\vec \eta} \parens*{f(U \cdot)}}_\infty \parens*{\ind_{\supp f}(\cdot) * |g(\cdot)| }(U \eta),
	\end{align}
	where $U$ is any invertible linear transformation and $\Dpi{\alpha}{\vec \eta}$ is the standard differentiation in multi-index notation. The reason for the form of this estimate will become apparent in the next step.
	
	Using \autoref{assump:psi_smooth} with an appropriate choice of $\tilRjld$ (see \eqref{eq:rot_lipschitz}), together with $\norm*{\hat\psi_{(\jl)}}_{\CC^{2n}} \le\beta_{2n}<\infty$, \eqref{eq:deriv_conv} yields
	\begin{align}
	\begin{split}\label{eq:deriv_conv_jld}
		\abs[\Big]{ \Dpi{\alpha}{\vec \eta} \parens[\Big]{ \!\parens*{\hat \psi_\jld * M_{\lambda'}\hat \kappa_0 }(U_\jl^{-\top}\vec \eta) \!}}
		&\le \norm[\Big]{\Dpi{\alpha}{\vec \eta} \parens*{\hat \psi_\jld (U_\jl^{-\top} \cdot)}}_\infty \parens*{\ind_{P_\jld} * \abs*{ M_{\lambda'}\hat \kappa_0 }} (U_\jl^{-\top} \vec \eta) \\
		&= \norm[\Big]{\Dpi{\alpha}{\vec \eta} \parens[\Big]{\hat \psi_{(\jld)} \parens*{(\wt U^\jl_\jld)^{\!\top} \!\cdot}\!}}_\infty \parens*{\ind_{P_\jld} * |\hat\kappa_0|} (U_\jl^{-\top} \vec \eta) \\
		&\le \beta_{|\alpha|} \norm*{\wt U^\jl_\jld}^{|\alpha|} \parens*{\ind_{P_\jld} * |\hat\kappa_0|} (U_\jl^{-\top} \vec \eta),
	\end{split}
	\end{align}
	as well as
	\begin{align}
	\begin{split}\label{eq:deriv_conv_jl}
		\abs[\Big]{ \Dpi{\alpha}{\vec \eta} \parens[\Big]{ \!\parens*{\hat \psi_\jl * M_{\lambda}\hat \kappa_0 }(U_\jl^{-\top}\vec \eta) \!}}
		&\le \norm[\Big]{\Dpi{\alpha}{\vec \eta} \hat \psi_{(\jl)} }_\infty \parens*{\ind_{P_\jl} * \abs*{ M_{\lambda}\hat \kappa_0 }} (U_\jl^{-\top} \vec \eta) \\
		&\le \beta_{|\alpha|} \parens*{\ind_{P_\jl} * |\hat\kappa_0|} (U_\jl^{-\top} \vec \eta).
	\end{split}
	\end{align}
	\mathtoolsset{showonlyrefs=true} % go back to only creating labels that are referenced

	Equations \eqref{eq:laplace_one_kappa}--\eqref{eq:deriv_conv_jl} allow us to estimate the different terms in \eqref{eq:stiff_matr_elem_part_int}, remembering to keep the support information of terms we estimate away:
	\begin{align}
		\abs{Y^{00}_\lld} &\lesssim 2^{\frac{j-j'}{2}} \! \int_{U_\jl^\top(P_\jl \cap P_\jld)} \norm*{\wt U^\jl_\jld}^{2n} \frac{\abs*{U_\jl^{-1} \vec s}^2}{w(\lambda)w(\lambda')} \parens*{ |\vec \eta|^2 + (4n + 2\gamma)|\vec \eta| + 4n(n+\gamma) + \gamma^2} \d\vec\eta
		\label{eq:stiff_matr_Y00}\\
		\abs{Y^{0*}_\lld} &\lesssim 2^{\frac{j-j'}{2}} \! \int_{\rcopywidth{U_\jl^\top P_\jl}{U_\jl^\top(P_\jl \cap P_\jld)}} \norm*{\wt U^\jl_\jld}^{2n} \frac{\abs*{U_\jl^{-1} \vec s}}{w(\lambda)w(\lambda')} \parens*{ |\vec \eta| + 2n + \gamma} \parens*{\ind_{P_\jld} * |\hat\kappa_0|} (U_\jl^{-\top} \vec \eta) \d\vec\eta 
		\label{eq:stiff_matr_Y0s}\\
		\abs{Y^{*0}_\lld} &\lesssim 2^{\frac{j-j'}{2}} \! \int_{\rcopywidth{U_\jl^\top P_\jld}{U_\jl^\top(P_\jl \cap P_\jld)}} \norm*{\wt U^\jl_\jld}^{2n} \frac{\abs*{U_\jl^{-1} \vec s}}{w(\lambda)w(\lambda')} \parens*{ |\vec \eta| + 2n + \gamma} \parens*{\ind_{P_\jl} * |\hat\kappa_0|} (U_\jl^{-\top} \vec \eta) \d\vec\eta
		\label{eq:stiff_matr_Ys0}\\
		\abs{Y^{**}_\lld} &\lesssim 2^{\frac{j-j'}{2}} \! \int_{\rcopywidth{}{U_\jl^\top(P_\jl \cap P_\jld)}} \norm*{\wt U^\jl_\jld}^{2n} \frac{1}{w(\lambda)w(\lambda')} \parens*{\ind_{P_\jl} * |\hat\kappa_0|} (U_\jl^{-\top} \vec \eta) \parens*{\ind_{P_\jld} * |\hat\kappa_0|} (U_\jl^{-\top} \vec \eta) \d\vec\eta \hspace{0.4cm}
		\label{eq:stiff_matr_Yss}
	\end{align}
	
	\noindent {\bfseries Step 4 -- Dealing with the convolution terms:}
	Ultimately, we have to find a restriction on $\jl$ in terms of fixed $\jld$ to be able to sum \eqref{eq:sparse_stiff_expanded} -- see Step 6. The way to obtain this restriction is to inspect the supports of $\hat \psi_\jl$ and $\hat \psi_\jld$ and see where the intersection is non-empty. However, the convolution $\hat \psi_\jl * M_{\lambda} \hat \kappa$ is supported on the Minkowski sum $\supp \hat \psi_\jl + \supp \hat \kappa$, and since the latter summand may be unbounded we have to tackle this problem a little differently.
	
	The idea now is to decompose $\bbR^d$ in such a way, that the increases in non-empty intersections (and other contributions arising therefrom) are offset by the decay of $\hat \kappa$. To this end, choose $m_0$ as the minimal $m\in\bbN$ such that $2^m \ge r_0$ (from the decay condition \eqref{eq:kappa_decay}). Then $\bbR^d$ may be written as the direct sum
	\begin{align*}
		\bbR^d = \parens*{P_\jl + B_{\bbR^d}(0,2^{m_0})} \,\, \dot\cup \dot{\bigcup_{m\ge m_0}} \parens*{P_\jl + B_{\bbR^d}(0,2^{m+1})}\setminus \parens*{P_\jl + B_{\bbR^d}(0,2^m)},
	\end{align*}
	which we choose to abbreviate by setting $P_\jl^m:= P_\jl + B_{\bbR^d}(0,2^m)$, as well as $Q_\jl^m:=P_\jl^{m+1}\setminus P_\jl^m$, thus
	\begin{align*}
		\bbR^d = P_\jl^{m_0} \,\, \dot\cup \dot{\bigcup_{m\ge m_0}} P_\jl^{m+1}\setminus P_\jl^m =  P_\jl^{m_0} \,\, \dot\cup \dot{\bigcup_{m\ge m_0}} Q_\jl^m.
	\end{align*}

	To see the gain of the above decomposition, we consider the convolution terms appearing in the estimates \eqref{eq:deriv_conv_jl}. In particular, for $\vec \xi \in Q_\jl^m$ and $\vec\zeta \in P_\jl^{m_0}$, the construction of the $Q_\jl^m$ yields $|\vec \xi - \vec \zeta| \ge 2^m-2^{m_0}$, which we can exploit to estimate (using $q>d$)
	\begin{align}\label{eq:est_conv_decomp}
		\parens*{\ind_{P_\jl} \! * \! |\hat\kappa_0| }(\vec \xi)
		= \! \int \! \ind_{P_\jl}(\vec \zeta) |\hat\kappa_0 (\vec \xi - \vec \zeta)| \d\vec \zeta \le \! \int_{|\vec \zeta| \ge 2^m-2^{m_0}} \! |\hat\kappa_0 (\vec \zeta)| \d\vec \zeta
		\stackrel{\eqref{eq:kappa_decay}}{\le} \int_{2^m-2^{m_0}}^\infty \frac{c_0}{r^q} \, r^{d-1} \d\Omega \d r \lesssim 2^{-m(q-d)}. \hspace{0.5cm}
	\end{align}

	\noindent {\bfseries Step 5 -- Dealing with the anisotropic terms:}
	The terms $\vec s \cdot \vec \xi$, respectively $\vec s \cdot U_\jl^{-\top}\eta = U_\jl^{-1}\vec s \cdot \eta$ after the transformation, can get very large ($\sim 2^j$), in particular if $\vec s$ is close to $P_\jl$ -- the support of $\hat \psi_\jl$. It is this behaviour that has to be counteracted by the preconditioning, since we would not be able to bound this term otherwise -- together with \eqref{eq:ind_jl_incl}, \eqref{eq:inv_Ujls} implies
	\mathtoolsset{showonlyrefs=false} % create tags, because it is only referenced implicitly
	\begin{align}\label{eq:inv_Ujls_final}
		\abs*{U_\jl^{-1}\vec s} &\le w(\lambda), \quad \text{and} \quad \abs*{U_\jl^{-1} \vec s } \lesssim \max(2^{j-j'},1) \parens*{w(\lambda') + 2^{j'} \dist_\bbSd (\vec s_\jl, \vec s_\jld) } \lesssim w(\lambda'),
	\end{align}
	\mathtoolsset{showonlyrefs=true}
	which conveniently cancels with the weights in the denominator.
	
	Finally, we need the inclusion \eqref{eq:trafo_P_jl} for estimating both $\abs{\vec \eta}$ as well as the volume of the integral. This illustrates another property of the transformation we employed, since, in essence (resp. as a consequence of the construction), it transforms the highly anisotropic sets $P_\jl$ back into a subset of a ball around the origin, see also \autoref{fig:UjlPjl}.\par\vspace{\baselineskip}
	
	\noindent {\bfseries Step 6 -- The condition for $j$ and $\ell$:}
	As noted above, to be able to sum $j$ and $\ell$, we need a restriction for these indices in terms of $\jld$. For constant (or compactly supported) $\hat \kappa$, such a condition follows naturally from the fact that the integral of the inner product is zero when the supports of the functions $\hat \psi_\jl$ and $\hat \psi_\jld$ do not intersect. For general $\hat \kappa$, we have to consider these intersections for the decomposition we introduced above. Here, the concentric structure of the $Q_\jl^m$ is irrelevant, it suffices to consider the sets $P_\jl^{m+1} \supseteq Q_\jl^m$. The main argument in this respect is \eqref{eq:ind_jl_incl}, which basically says that for $j\ge m_>+3$, we are able to restrict $j$ and $\ell$ in relation to $\jld$, whereas for $j\le m_>+2$, we have to assume the worst-case scenario of all indices contributing to the sum (or rather, a more precise estimate doesn't change anything in this case).
	
	An immediate consequence is, that all terms of the form $2^{j-j'}$ can be estimated
	\begin{align}\label{eq:est_diff_j}
		2^{j-j'} \le \left\{ \!\! \begin{array}{rl}
			2^{|j-j'|}, & j\ge m_> +3 \\
			2^{j}, & j\le m_>+2
		\end{array} \!\! \right\} \lesssim 2^{m_>}.
	\end{align}
	Similarly, we can estimate the norm of $\wt U^\jl_\jld$. After appealing to a \eqref{eq:U_lld}, we apply \eqref{eq:ind_jl_incl}, first the condition in $\ell$ for $(*)$ and then in $j$ for $(**)$, and finally \eqref{eq:est_diff_j}, to arrive at
	\begin{align}
%	\begin{split}
		\norm*{\wt U^\jl_\jld} \! \stackrel{\eqref{eq:U_lld}}{\lesssim} \! \max(2^{j-j'},1) + 2^j \dist_\bbSd (\vec s_\jl, \vec s_\jld) &\stackrel{\mathclap{(*)}}{\le} \left\{ \!\! \begin{array}{rl}
			\max(2^{j-j'},1) + 5c_\omega 2^{m_>+ j-j'}, & j\ge m_> +3 \\
			\max(2^{j-j'},1)+ \frac{\pi}{2}\, 2^j, & j\le m_>+2
		\end{array} \right. \notag \\
		&\lesssim \left\{ \!\! \begin{array}{rl}
			2^{m_> + |j-j'|}, & j\ge m_> +3 \\
			2^j, & j\le m_>+2
		\end{array} \!\! \right\} \stackrel{(**)}{\lesssim} 2^{m_>}. \label{eq:norm_U_lld}
%	\end{split}
	\end{align}

	Applying the estimates \eqref{eq:trafo_P_jl} and \eqref{eq:est_conv_decomp}--\eqref{eq:norm_U_lld}, we see that (by definition, $m=m'=0$ for $Y^{00}_\lld$)
	\mathtoolsset{showonlyrefs=false} % create tags in the middle, even though only first and last are referenced
	\begin{align}
		\abs{Y^{00}_\lld} &\lesssim \int_{U_\jl^\top(P_\jl \cap P_\jld)} \d\vec\eta
		\label{eq:stiff_matr_Y00_est}\\
		\abs{Y^{0*}_\lld} &\lesssim \int_{\rcopywidth{U_\jl^\top (P_\jl \cap P_\jld^{m_0})}{U_\jl^\top(P_\jl \cap P_\jld)}} \d\vec\eta + \sum_{m\ge m_0} 2^{-m(q-d-2n-\frac 12)} \int_{\rcopywidth{U_\jl^\top(P_\jl \cap Q_\jld^m)}{U_\jl^\top(P_\jl \cap P_\jld)}} \d\vec\eta
		\label{eq:stiff_matr_Y0s_est}\\
		\abs{Y^{*0}_\lld} &\lesssim \int_{\rcopywidth{U_\jl^\top (P_\jl^{m_0} \cap P_\jld)}{U_\jl^\top(P_\jl \cap P_\jld)}} \d\vec\eta + \sum_{m\ge m_0} 2^{-m(q-d-2n-\frac 32)} \int_{\rcopywidth{U_\jl^\top(Q_\jl^m \cap P_\jld)}{U_\jl^\top(P_\jl \cap P_\jld)}} \d\vec\eta
		\label{eq:stiff_matr_Ys0_est}\\
		\abs{Y^{**}_\lld} &\lesssim \int_{\rcopywidth{U_\jl^\top (P_\jl^{m_0} \cap P_\jld^{m_0})}{U_\jl^\top(P_\jl \cap P_\jld)}} \d\vec\eta + \sum_{m\ge m_0} 2^{-m(q-d-2n-\frac 12)} \int_{\rcopywidth{U_\jl^\top(P_\jl^{m_0} \cap Q_\jld^m)}{U_\jl^\top(P_\jl \cap P_\jld)}} \d\vec\eta + \ldots \label{eq:stiff_matr_Yss_est} \\
		&\phantomrel \ldots + \sum_{m\ge m_0} 2^{-m(q-d-2n-\frac 12)} \int_{\rcopywidth{U_\jl^\top(Q_\jl^m \cap P_\jld^{m_0})}{U_\jl^\top(P_\jl \cap P_\jld)}} \d\vec\eta + \sum_{m, m'\ge m_0} 2^{-(m+m')(q-d-2n-\frac 12)} \int_{\rcopywidth{U_\jl^\top(Q_\jl^m \cap Q_\jld^m)}{U_\jl^\top(P_\jl \cap P_\jld)}} \d\vec\eta
		\notag
	\end{align}
	\mathtoolsset{showonlyrefs=true} % go back to only creating labels that are referenced
	
	\noindent {\bfseries Step 7 -- Summing $\vec k$:}
	Thus far, we have omitted the case $\vec k = U^\jl_\jld \vec k'$ -- in fact, to sum over $\vec k$, we need treat even more elements differently. In order to estimate the term $\abs*{\vec k-U^\jl_\jld\vec k'}$, we choose $K^\jl_\jld\vec k' \in \bbZ^d$ as a (possibly non-unique) closest lattice element to $U^\jl_\jld\vec k'$ (for example by rounding every component to the nearest integer), which may be interpreted as a projection of $U^\jl_\jld \vec k'$ onto the lattice $\bbZ^d$. Then $\abs*{K^\jl_\jld\vec k'  - U^\jl_\jld\vec k' }\le \frac{\sqrt{d}}2$, and if we restrict $\vec k \in\bbZ^d$ such that $\abs*{\vec k-K^\jl_\jld\vec k' }\ge \sqrt{d}$, it holds that
	\vspace{-0.2cm}
	\begin{align}\label{eq:estimate_Kk'}
		\abs*{\vec k - U^\jl_\jld\vec k'} \ge \abs*{\vec k- K^\jl_\jld\vec k'} - \frac{\sqrt{d}}2
		\ge \frac 12 \abs*{\vec k- K^\jl_\jld\vec k'}.
	\end{align}
	For $\vec k \in \bbZ^d$ such that $\abs*{\vec k- K^\jl_\jld \vec k'}<\sqrt{d}$, we retrace the derivation of all above estimates without the partial integration, which, in effect, only eliminates the divisor $\abs*{\vec k- U^\jl_\jld \vec k'}^{2n}$ (and reduces the constants). Putting the estimates for \eqref{eq:stiff_matr_Y00_est}--\eqref{eq:stiff_matr_Yss_est} together, we arrive at
	\begin{align*}
		\VF_\lld &\lesssim \abs*{\vec k - U^\jl_\jld \vec k'}^{-2n} \parens[\Big]{ \abs*{Y^{00}_\lld} + \abs*{Y^{0*}_\lld} + \abs*{Y^{*0}_\lld} +\abs*{Y^{**}_\lld}} =: Z^{00}_\lld + Z^{0*}_\lld + Z^{*0}_\lld + Z^{**}_\lld
		\intertext{
	for $\abs*{\vec k- K^\jl_\jld \vec k'}\ge\sqrt{d}$, and similarly for $\abs*{\vec k- K^\jl_\jld \vec k'}<\sqrt{d}$,
		}
		\VF_\lld &\lesssim \abs*{Y^{00}_\lld} + \abs*{Y^{0*}_\lld} + \abs*{Y^{*0}_\lld} +\abs*{Y^{**}_\lld}  =: Z^{00}_\lld + Z^{0*}_\lld + Z^{*0}_\lld + Z^{**}_\lld.
	\end{align*}
	Note that the different cases for $\vec k$ are incorporated in the definition of the $Z$-terms. %\AO{Definition von $Z$-Termen nach wie vor n\"otig/sinnvoll? Legacy notation...}
	
	The intention now is to  prove \eqref{eq:sparse_stiff_expanded} by showing
	\begin{align}
		\sup_{\lambda'\in\Lambda} \sum_{\lambda\in\Lambda} \abs*{\VF_\lld}^p
		&\lesssim \sup_{\lambda'\in\Lambda} \sum_{\lambda\in\Lambda} \! \parens[\Big]{Z^{00}_\lld +Z^{0*}_\lld +Z^{*0}_\lld +Z^{**}_\lld}^p \notag \\
		&\le \sup_{\lambda'\in\Lambda} \sum_{\lambda\in\Lambda} \parens*{Z^{00}_\lld}^p +\parens*{Z^{0*}_\lld}^p +\parens*{Z^{*0}_\lld}^p +\parens*{Z^{**}_\lld}^p <\infty,
		\label{eq:est_stiff_matr}
	\end{align}
	where the second inequality requires $p\le 1$. All four terms have the same structure in $\vec k$ -- namely, the integrals do not depend on this parameter. Thus we may calculate the sum over $\vec k$ separately, which crucially requires the condition $p>\frac{d}{2n}$,
	\begin{align*}
		\sum_{\substack{\vec k \in\bbZ^d \\ \abs{\vec k- K^\jl_\jld\vec k' }\ge \sqrt{d}}}
		\frac{1}{\abs*{\vec k- U^\jl_\jld \vec k'}^{2np}} \stackrel{\eqref{eq:estimate_Kk'}}{\le} \sum_{\substack{\vec k \in\bbZ^d \\ \abs{\vec k- K^\jl_\jld\vec k' }\ge \sqrt{d}}} \frac{2^{2np}}{\abs*{\vec k- K^\jl_\jld \vec k'}^{2np}} = \sum_{\substack{\smash{\vec k} \in \bbZ^d \\ \abs{\vec k }\ge \sqrt{d}}} \frac{2^{2np}}{\abs*{\vec k}^{2np}} =: G_{d,2np} <\infty.
	%\end{split}
	\end{align*}

	The remaining sum over $\vec k\in\bbZ^d:\abs*{\vec k- K^\jl_\jld\vec k' }< \sqrt{d}$ has at most $\CO (\sqrt{d}^d)$ terms. Taken together, this implies (recall that the $Y$-terms do not depend on $\smash{\vec k}$ and $\smash{\vec k'}$)
	\begin{align}
		\sup_{\lambda'\in\Lambda} \sum_{\lambda\in\Lambda} \abs*{\VF_\lld}^p \lesssim \sum_{j\in\bbN_0} \sum_{\ell=0}^{L_j} \abs*{Y^{00}_\lld}^p + \abs*{Y^{0*}_\lld}^p + \abs*{Y^{*0}_\lld}^p +\abs*{Y^{**}_\lld}^p.
	\end{align}
	
	\noindent {\bfseries Step 8 -- Estimating the Number of Intersections:}
	
	As the last important tool to show the finiteness of \eqref{eq:est_stiff_matr}, we show one more estimate -- the number of non-empty intersections $P_\jl^m \cap P_\jld^{m'}$ in terms of $\ell$. Recalling $\alpha_j=2^{-j+1}$, we derive that (for $j,j'\ge m_>+1$)
	\begin{align}
	\begin{split}\label{eq:intersection_ell}
		N_{\jjd,\ell'}^\mmd
		&:=\#\set*{\ell\in\{0,\ldots,L_j\} }{ P_\jl^m \cap P_\jld^{m'} \neq\emptyset } \le \#\set*{\ell\in\{0,\ldots,L_j\} }{ \CP_\bbSd(P_\jl^m) \cap \CP_\bbSd(P_\jld^{m'}) \neq\emptyset } \hspace{-1cm} \\
		&\phantom{:}= \#\set*{ \ell\in\{0,\ldots,L_j\} }{ B_\bbSd(\vec s_\jl,\alpha_j^m) \cap B_\bbSd(\vec s_\jld,\alpha_{j'}^{m'})\neq\emptyset } \\ &\phantom{:}\stackrel{\mathclap{\eqref{eq:est_alpha_j_m}}}{\le} \#\set*{ \ell\in\{0,\ldots,L_j\} }{ B_\bbSd(\vec s_\jl, c_\omega 2^{m-j}) \cap B_\bbSd(\vec s_\jld, c_\omega 2^{m'-j'})\neq\emptyset } \\
		&\phantom{:}\stackrel{\mathclap{\eqref{eq:est_intersect_sl}}}{\le} \,\,
		\frac{\mu\parens*{ B_\bbSd(\vec s_\jld,3c_\omega 2^{m_>-j_<}) }}{\mu\parens*{ B_\bbSd(\vec s_\jl,\frac{1}{3} 2^{-j}) }} \stackrel{\eqref{eq:hyp_cap_est}}{\le} 
		\frac{C_d}{c_d} \parens*{ 9 c_\omega \? 2^{j-j_< +m_>}}^{d-1} \lesssim 2^{\parens{|j-j'|+m_>}(d-1)}.
	\end{split}
	\end{align}
	In particular, the estimate is independent of the choice of $\ell'$.\par\vspace{\baselineskip}
	
	\noindent {\bfseries Step 9 -- Summing $j$ and $\ell$:}
	The following procedure is very similar for all four terms \eqref{eq:stiff_matr_Y00_est}--\eqref{eq:stiff_matr_Yss_est}, we demonstrate the procedure with the most difficult term. With $L_j\lesssim 2^{j(d-1)}$ and $\max(m+1,m'+1)=m_>+1$ for \eqref{eq:intersection_ell}, we have
	\begin{align*}%\label{eq:stiff_matr_two_kappa}
		\sum_{\lambda\in\Lambda} \parens*{Z^{**}_\lld}^p
		&\le \sum_{j\in\bbN_0 } \sum_{\ell=0 }^{L_j} \parens[\bigg]{ \sum_{m,m'\ge m_0} 2^{-(m+m')(q-d-2n-\frac 12)} \int_{U_\jl^{\top} (Q_\jl^m \cap Q_\jld^{m'})} \d\vec \eta }^p \\
		&\stackrel{\mathclap{p\le 1}}{\lesssim} \sum_{m,m'\ge m_0} 2^{-p(m+m')(q-d-2n-\frac 12)} \sum_{j\in\bbN_0 } \sum_{\ell=0 }^{L_j} \parens[\bigg]{\int_{U_\jl^{\top} (Q_\jl^m \cap Q_\jld^{m'})} \d\vec \eta }^p \\
		&\stackrel{\mathclap{\eqref{eq:ind_jl_incl}}}{\le} \sum_{m,m'\ge m_0} 2^{-p(m+m')(q-d-2n-\frac 12)} \Bigg[ \,\,\, \sum_{j=0}^{\mathclap{m_>+3}} \,\, (L_j+1) \sup_{\ell=0}^{L_j} \parens[\bigg]{  \int_{U_\jl^{\top} (Q_\jl^m \cap Q_\jld^{m'})} \d\vec \eta }^p + \ldots \\
		&\phantom{{}\le \sum_{m,m'\ge m_0} 2^{-p(m+m')(q-d-2n-\frac 12)} \Bigg[\,\,\,}\mathllap{\ldots +} %alignment of sums
		\sum_{\mathclap{\substack{j\ge m_>+4 \\ |j-j'|\le 2}}} \,\, N_{\jjd,\ell'}^{m+1,m'+1} \sup_{\ell=0}^{L_j} \parens[\bigg]{ \int_{U_\jl^{\top} (Q_\jl^m \cap Q_\jld^{m'})} \d\vec \eta }^{p} \Bigg] \\
		&\stackrel{\mathclap{\eqref{eq:trafo_P_jl}}}{\lesssim} \sum_{m,m'\ge m_0} 2^{-p(m+m')(q-d-2n-\frac 12)} \Bigg[ \,\,\, \sum_{j=0}^{\mathclap{m_>+3}} \,\, (L_j+1) \, 2^{mdp} + \sum_{\substack{j\ge m_>+4 \\ |j-j'|\le 2}} N_{\jjd,\ell'}^{m+1,m'+1} 2^{mdp} \Bigg] \\
		&\stackrel{\mathclap{\eqref{eq:intersection_ell}}}{\lesssim} \sum_{m,m'\ge m_0} 2^{-p(m+m')(q-d-2n-\frac 12)} \Bigg[ \,\,\, \sum_{j=0}^{\mathclap{m_>+3}} \,\, 2^{j(d-1) + m_> dp} + \sum_{\substack{j\ge m_>+4 \\ |j-j'|\le 2}} 2^{\parens{|j-j'|+m_> +1}(d-1)+ m_> dp} \Bigg] \\
		&\lesssim \sum_{m,m'\ge m_0} 2^{-p(m+m')(q-d-2n-\frac 12)} 2^{m_>(dp+d-1)} \le \sum_{m,m'\ge m_0} 2^{-p(m+m')(q-2d-2n-\frac 12-\frac{d-1}{p})} \\
		&= \sum_{m\ge m_0}  2^{-mp(q-2d-2n-\frac 12-\frac{d-1}{p})} \sum_{m'\ge m_0} 2^{-m'p(q-2d-2n-\frac 12-\frac{d-1}{p})} = \parens{\frac{c^{m_0}}{1-c}}^2 < \infty.
	\end{align*}
	Here, $c:= 2^{-p(q-2d-2n-\frac 12-\frac{d-1}{p})}$ can be summed geometrically, since $q> 2d+2n+\frac 32+\frac{d-1}{p}$ implies $c<1$. As we have now estimated every term in \eqref{eq:est_stiff_matr} independently of $\lambda'$, taking the supremum does not change anything and the proof is finished.
\end{proof}

\subsection{Sparsity of $\VP$}

Recall that in order for \autoref{alg:modsolve} to converge in optimal complexity (compare \autoref{th:conv_modsolve}), not only $\VF$ but also $\VP$ need to be compressible (resp. sparse). This is formulated in the following theorem.

\begin{theorem}\label{th:sparse_proj}
	Again, let $\Phi = (\varphi_\lambda)_{\lambda\in\Lambda}$ satisfy \autoref{assump:psi_smooth} for $2n$ with $\frac d2<n\in\bbN$, and choose $p\in\bbR$ such that $1 > p > \frac{d}{2n}$. Then the projection $\VP$, see \eqref{eq:PMatRep}, is $p$-sparse in this frame -- in other words,
	\begin{align}%\label{eq:sparse_stiff}
		\norm*{ \VP }_{\ell^p(\Lambda)\to\ell^p(\Lambda)} =
		\norm[\Big]{ \VW\inpr{\Phi, \Phi}_{L^2} \VW^{-1} }_{\ell^p(\Lambda)\to\ell^p(\Lambda)}  < \infty.
	\end{align}
\end{theorem}

As the proof proceeds exactly along the lines of \autoref{th:sparse_stiff} (but with substantial simplifications due to the lack of the operator $A$), we leave it to the reader.

\section{Main results}\label{sec:approx}

The results so far allow us to formulate the following corollary to \autoref{thm:GelfandFrame}, which, in essence, states that the complexity of \textbf{modSOLVE} is \emph{linear} with respect to the number of relevant coefficients of the discretisation.

\begin{corollary}\label{cor:ridge_opt_complexity}
	Assume that $\Vf$ is $\sigma^*$-optimal (compare \autoref{def:sigma-optimal}) and that the system $\VF \Vu = \Vf$ has a solution
	$\Vu \in \ell^p_w(\Lambda)$ for $\sigma\in (0,\sigma^*)$ and $p:=\frac{1}{\frac12 + \sigma}$. Then the solution $\Vu_\eps:=\mathbf{modSOLVE}[\eps,\VF,\VP,\Vf]$ of the ridgelet-based solver recovers this approximation rate -- i.e.
	\[
		\# \supp \Vu_\eps \lesssim \eps^{-1/\sigma}|\Vu|_{\ell^p_w(\Lambda)}^{1/\sigma},
	\]
	and the number of arithmetic operations is at most a multiple of $\eps^{-1/\sigma}|\Vu|_{\ell^p_w(\Lambda)}^{1/\sigma}$.
\end{corollary}

%%%% From Paris Slides:
%%The frame property of the ridglets (see \autoref{thm:GelfandFrame}) and the compressibility of $\VF$ and $\VP$ (see \autoref{sec:sparsity}) -- compare ingredients i and ii formulated at the beginning of \autoref{sec:discretization} -- together with the results of \autoref{sec:wellposed} allow us to directly use results from \cite{Stevenson2004,Dahlke2007} to construct
%%an algorithm $\textbf{SOLVE}[\eps,\VF,\Vf]\to \Vu_{\eps}$ which computes
%%an approximate solution $\Vu_{\eps}$
%%of the linear system $\VF\Vu=\Vf$ 
%%up to an error $\eps$ in optimal complexity.
%%
%%
%%This means that if $u\in \Hs$ is a solution which has an $N$-term approximation rate
%%of order $\sigma<\sigma^{\ast}$,
%%%i.e. for all $N$ there exist linear combinations $v_N$ of $N$ frame elements, such that
%%%\[
%%%	\norm{u - v_N}_{\Hs}\lesssim N^{-\sigma},
%%%\]
%%then the algorithm $\mathbf{SOLVE}$ produces an approximation $u_N$ with 
%%the same asymptotic rate in order $N$ flops,
%%\[
%%	\norm{u - u_N}_{\Hs}\lesssim N^{-\sigma}.
%%\]

Finally, the last assumption -- that the discretisation of typical solutions are in $\ell^p_w(\Lambda)$ -- is also satisfied by the ridgelet discretisation. The proof of this theorem is based on arguments of \cite{mutilated} and is the subject of an upcoming paper \cite{go_approx}.

\begin{theorem}\label{th:approx}
For a function $u\in L^2(\bbR^d)$ such that $u,\, \vec s\cdot \nabla u \in H^t(\bbR^d)$ apart from discontinuities across hyperplanes containing $\vec s$. Then $\VW \inpr{\Phi,u}_{L^2} \in \ell^p_w$, the weak $\ell^p$-space with $\frac 1p = \frac td + \frac 12$. This is the best possible approximation rate for functions in $H^t(\bbR^d)$ (even without singularities!).
\end{theorem}

To conclude the theoretical discussion, the bottom line is that the presented construction ``sparsifies'' \emph{both} the system matrix as well as typical solutions of transport problems (in the sense of compressibility and $N$-term approximations, respectively), which makes it the ideal candidate for the development of fast algorithms, as underscored also by the results of \autoref{cor:ridge_opt_complexity}.

\section{Numerical Experiments}\label{sec:numexp}

To underpin the theoretical claims of the paper, we implemented \autoref{alg:dampedRich} in \matlab. We need to stress, however, that as things stand, this is not a competitive solver, but rather a proof-of-concept. There are two main caveats: \textbf{APPLY} is not fully adaptive -- but rather uses a heuristic based on the distance of the translations -- and the necessary quadrature effort is substantial. To a degree, this is the price for sticking close to the theory -- another paper desribing an implementation based on \texttt{FFT} is forthcoming, see \cite{FFT-paper}.

\subsection{Implementation}\label{ssec:implementation}

As a proof-of-concept, the implementation is only $2$-dimensional. The first important difference between theory and implementation is that instead of using the rotations $\Rjl$ to generate the different ridgelets, we use a shearing matrix
\begin{align}\label{eq:shear_matrix}
	S_\jl := \begin{pmatrix}
		1 & \frac{\ell}{2^j} \\ 0 & 1
	\end{pmatrix}.
\end{align}
The advantage of this is that the transformed grid of translations ($U_\jl \bbZ^d$, compare \autoref{def:phi_jlk}) becomes invariant to the relevant shears and one can avoid tedious interpolation that would be necessary for a faithful implementation based on rotations. The difference between the different partitions of unity is illustrated in \autoref{fig:part}. This change does not affect the theoretical properties, as can be shown using \cite{alpha-mol}.
\begin{figure}
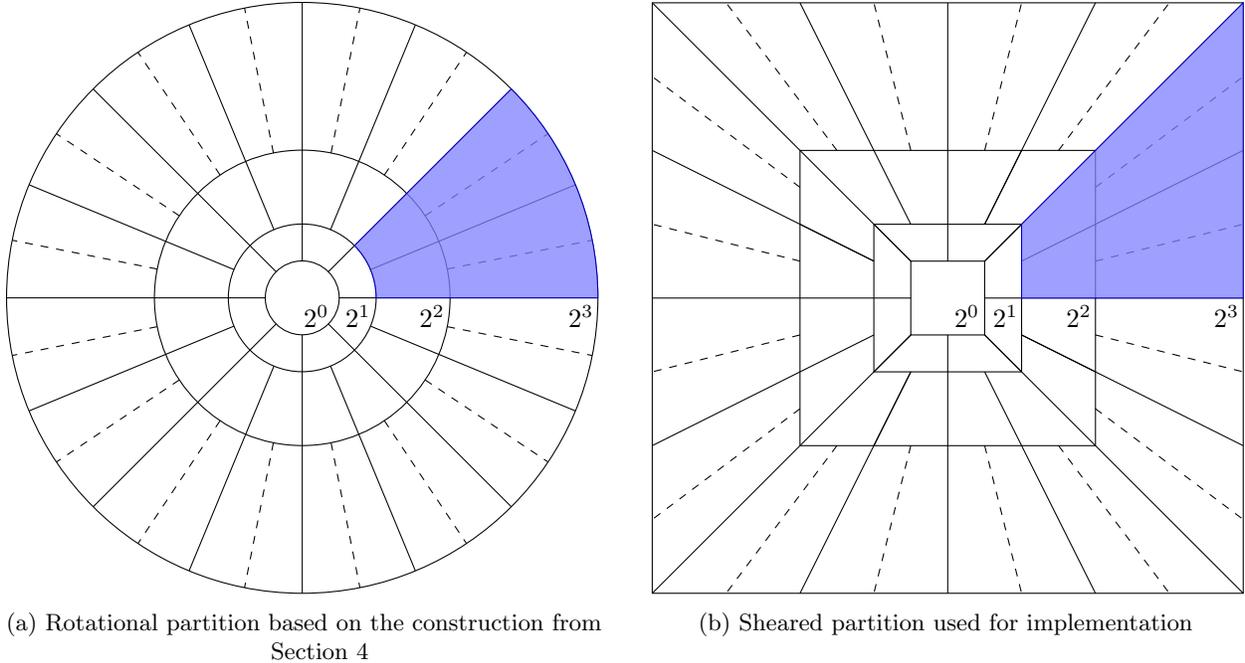

\subcaptionbox{Rotational partition based on the construction from \autoref{sec:ridgeframes}\label{fig:part:rot}}%
	{\includegraphics[width=0.48\textwidth]{figures/partition_circle.tikz}}%
\hfill%
\subcaptionbox{Sheared partition used for implementation\label{fig:part:sh}}%
	{\includegraphics[width=0.48\textwidth]{figures/partition_square.tikz}}%
\caption{Two different partitions in Fourier space, with the support of one ridgelet on scale $j=2$ shaded in blue\protect\footnotemark. The dashed lines are the continuation of the pattern but only become relevant for the following scale $j=3$.}\label{fig:part}
\end{figure}
% see \footnotetext below -- check that it is actually on the same page as the plot!
\footnotetext{For the implementation, the support is actually taken together with its point-symmetric mirror image (around the origin) to reduce the number of parameters.}

As in \autoref{sec:ridgeframes}, the ridgelets are constructed in the Fourier domain, for a sufficiently smooth transition function between zero and one. In our case, we used
\begin{equation}
	t(\xi)=\frac{\exp\left(-\frac{1}{\xi^\alpha}\right)}{\exp\left(-\frac{1}{\xi^\alpha}\right)+\exp\left(-\frac{1}{(1-\xi)^\alpha}\right)}
\end{equation}
with $\alpha=1.1$, as opposed to the ``classical'' choice of $\alpha=2$ (for an example how $V^{(\jl)}$ and $W$ can be constructed from $t$, see \autoref{sec:window}). The reason for this is that the transition for $\alpha=2$ -- while being $\CC^\infty$ -- nevertheless has a very sharp step around $\xi=0.5$, which results in an unfavourable localisation in physical space. The choice of $\alpha=1.1$ alleviates this problem to a degree.

We note that -- while we have neglected quadrature effort in the theoretical discussion (as is usually the case, see \cite{Stevenson2004}) -- the actual computational cost is not negligible. Furthermore, for a good localisation of the solution in physical space, we have chosen a relatively high value for the absorption coefficient (which is constant in the model problem for the implementation). However, all of these caveats can be alleviated by the above-mentioned \texttt{FFT}-based solver in \cite{FFT-paper}, although at the ``cost'' of a less direct relationship to the theoretical results.

\subsection{Convergence of the solver}

For both smooth and singular functions, we have used \autoref{alg:modsolve} to solve the transport equation. In particular, we have observed that \emph{without} the projection $\VP$, the results deteriorate and eventually diverge, while with the projection, convergence can be observed. To the best of our knowledge, this is the first time where the positive effect of such a projection is observed in practice. It is worth noting however, that the kernels of $\VF$ and $\VP$ in their respective restrictions to a finite set of coefficients don't match anymore. This is also observed in the sense that for frames that are ``too small'', the projection does not improve the convergence.

\subsubsection{Smooth functions}

For smooth functions, the algorithm works as expected, as illustrated in Figures \ref{fig:reg_sol} and \ref{fig:reg_nterm}.

\begin{figure}
%\subcaptionbox{Right-hand side\label{fig:reg_sol:rhs}}%
%	{\includegraphics[width=0.5\textwidth]{figures/exp_rhs}}%
%\subcaptionbox{Solution\label{fig:reg_sol:sol}}%
	{\includegraphics[width=0.7\textwidth]{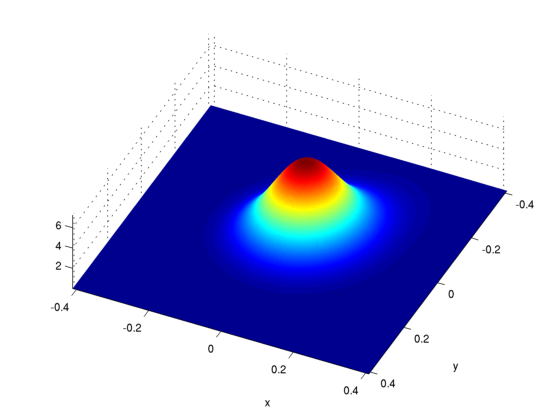}}%
\caption{Solution to the transport equation for a smooth right-hand side (Gaussian)}\label{fig:reg_sol}
\end{figure}

\begin{figure}
\subcaptionbox{Relative error between iterations\label{fig:reg_nterm:err_rel}}%
	{\includegraphics[width=0.5\textwidth]{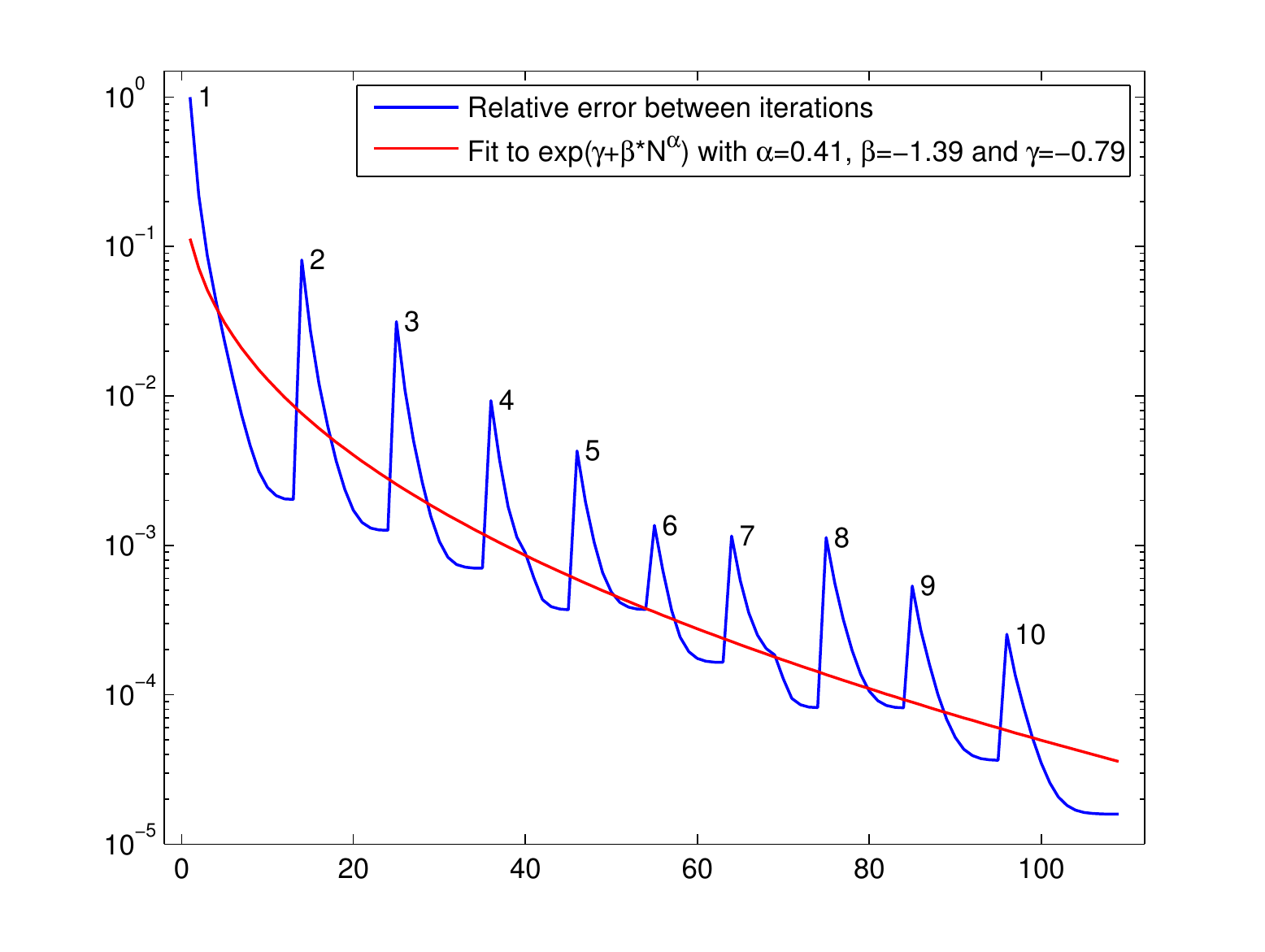}}%
\subcaptionbox{$N$-term approximations of outer iterations\label{fig:reg_nterm:nterm}}%
	{\includegraphics[width=0.5\textwidth]{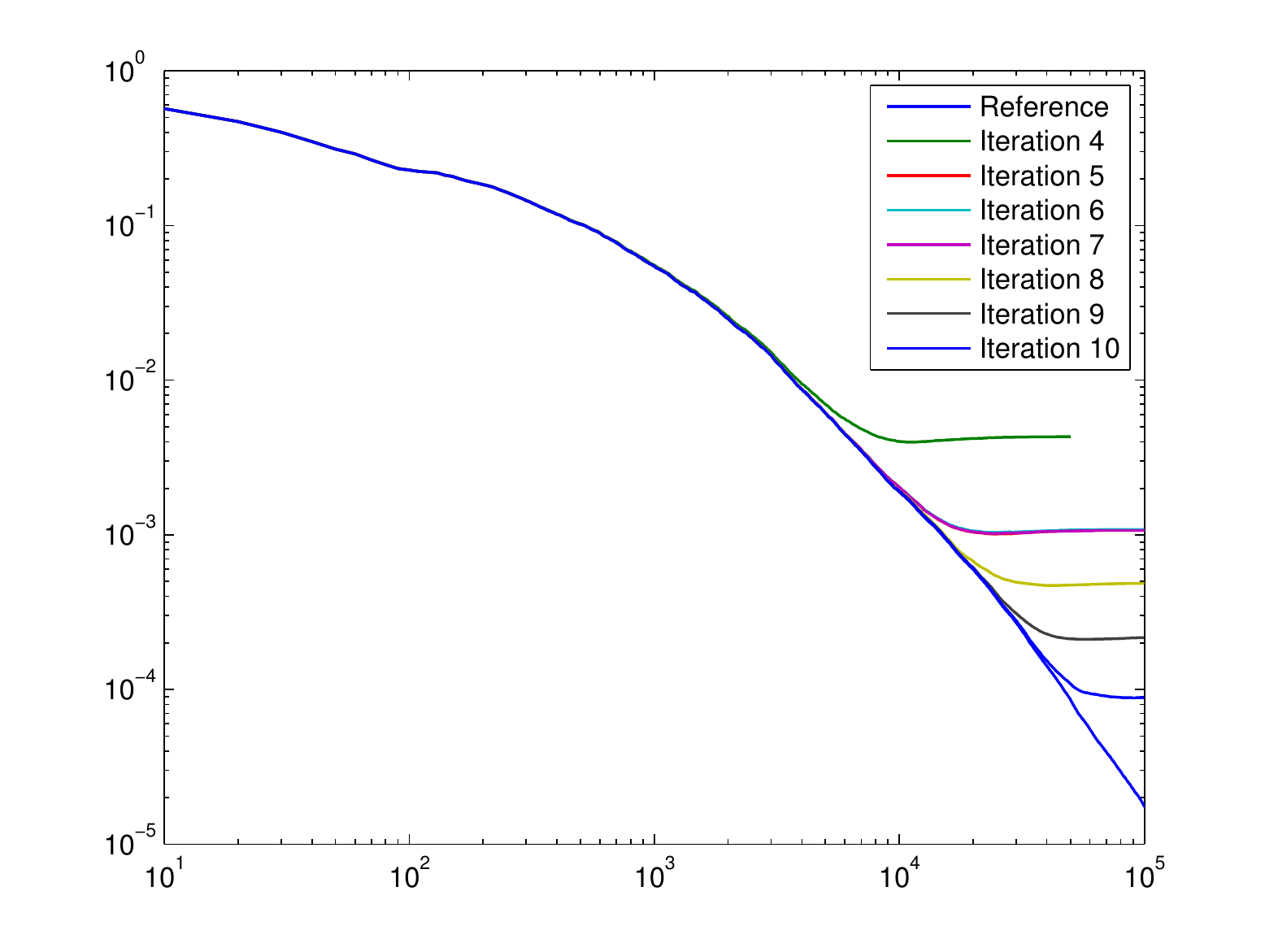}}%
\caption{Convergence of solver to solution in \autoref{fig:reg_sol}. Subplot \subref{fig:reg_nterm:err_rel} shows the relative error between iterations, while \subref{fig:reg_nterm:nterm} shows the $N$-term approximations at the end of an outer iteration (beginnings marked in \subref{fig:reg_nterm:err_rel}) against the reference solution.}\label{fig:reg_nterm}
\end{figure}

\subsubsection{Functions with Singularities}

The main thrust of the construction, however, is that singularities are resolved with the same $N$-term approximation (and complexity) as if they weren't there. We illustrate this with the following example in \autoref{fig:sing_sol} -- the right-hand side is a product of a Gaussian times a box function, rotated in a direction that isn't aligned with any $\vec s_\jl$ in the ridgelet frame (rather arbitrarily constructed from irrational numbers such that it lies in the second octant) -- $\vec s=\parens*{\frac{\sqrt{2}}{2},\frac{\pi}{3}}^\top$, normed to $1$. We use the same $\vec s$ as the transport direction in the differential equation, which means that there is no smoothing orthogonal to $\vec s$, and in particular, the singularities of the right-hand side remain in the solution.

In \autoref{fig:sing_nterm}, the $N$-term approximation of the reference solution (which can be calculated explicitly thanks to constant $\kappa$), compared against the $N$-term approximation of the numerical solution after several outer iteration of \autoref{alg:dampedRich}. Even though the functions have singularities, the $N$-term approximation can be seen to converge exponentially!

\begin{figure}
%\subcaptionbox{Right-hand side\label{fig:sing_sol:rhs}}%
%	{\includegraphics[width=0.5\textwidth]{figures/box_rhs}}%
%\subcaptionbox{Solution\label{fig:sing_sol:sol}}%
	{\includegraphics[width=0.7\textwidth]{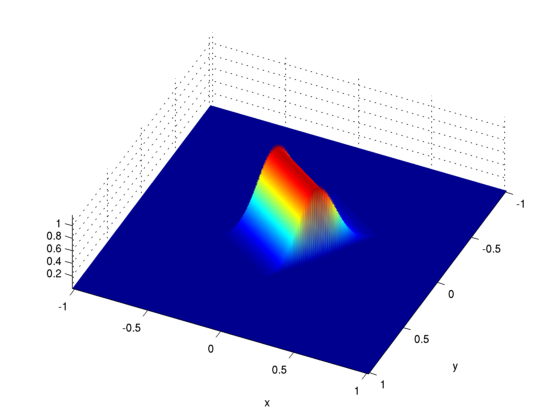}}%
\caption{Solution to the transport equation for a singular right-hand side (box times Gaussian) with transport direction parallel to the orientation of the singularity.
%The comparatively small difference between the plots (aside from the $z$-scaling) is due to relatively strong absorption (see \autoref{ssec:implementation}).
}\label{fig:sing_sol}
\end{figure}

\begin{figure}
\includegraphics[width=0.7\textwidth]{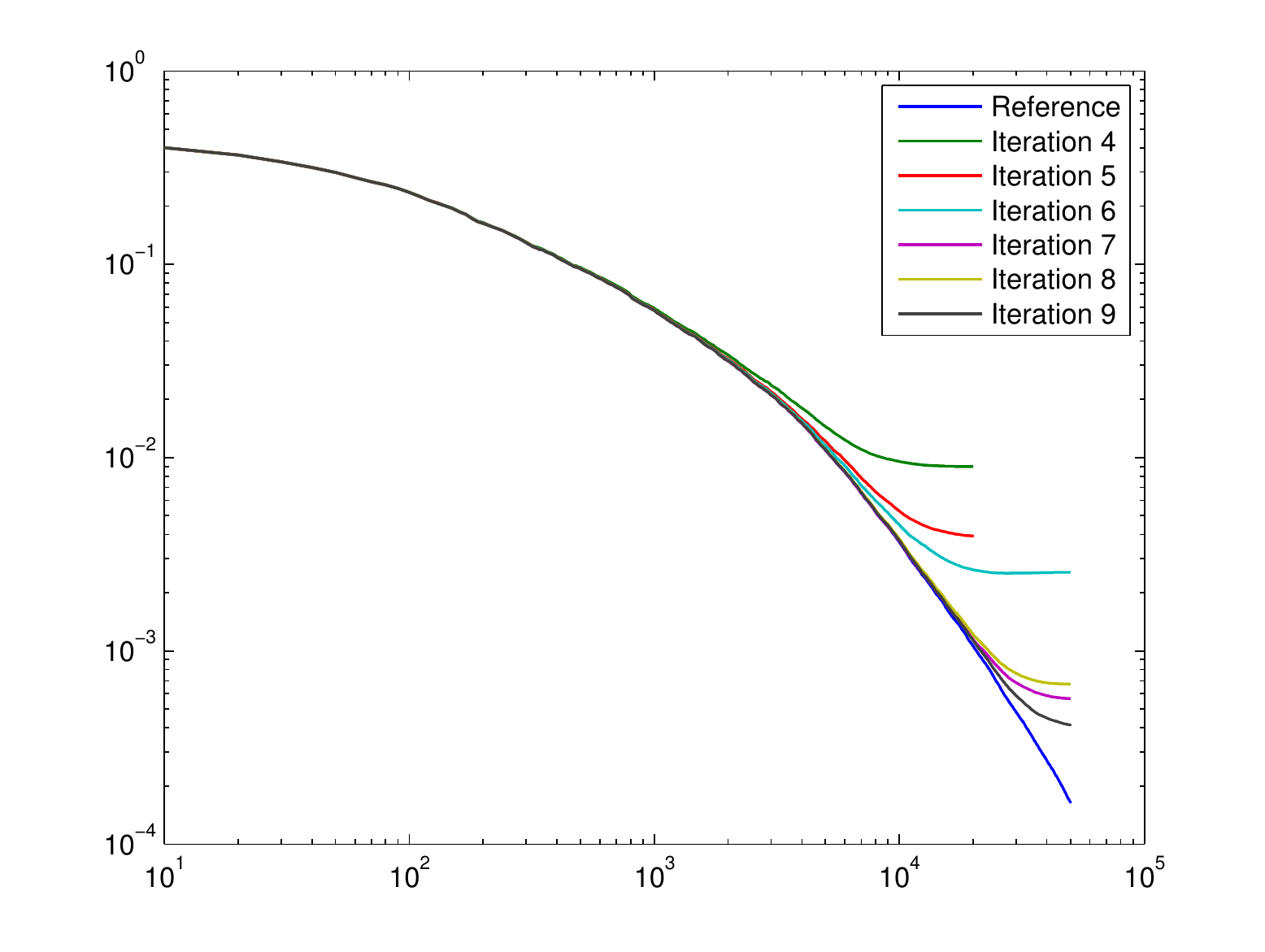}
\caption{Comparison of the $N$-term approximations of the reference solution, against the $N$-term approximations of the output of \autoref{alg:dampedRich}. Even though the functions are singular, the $N$-term approximations converge exponentially!}\label{fig:sing_nterm}
\end{figure}

In \autoref{fig:sing_loc}, we illustrate the localisation properties of the ridgelet frame -- namely, for the singular solution from \autoref{fig:sing_sol}, we consider the translations corresponding to the 10000 largest coefficients (of the discretisation up to scale $j=10$). As can be seen in \autoref{fig:sing_loc}, \subref{fig:sing_loc:0}--\subref{fig:sing_loc:10}, the higher the scale, the less coefficients are active (also with diminishing maximal Euclidian norms), and the more they are aligned with the location of the singularities. Note that the grid is only refined in one direction, which explains the constant distance in the $y$-direction. Conversely, in the $x$-direction, what may appear like a single point are often many points very close to each other.

\begin{figure}
\newlength{\reduceabove}
\setlength{\reduceabove}{-0.15cm}
\newlength{\reducebelow}
\setlength{\reducebelow}{-0.05cm}
\newcommand{\reduceabovecaptionskip}{\vspace{\reduceabove}}
\subcaptionbox{Solution\label{fig:sing_loc:ref}}%
	{\includegraphics[width=0.33\textwidth]{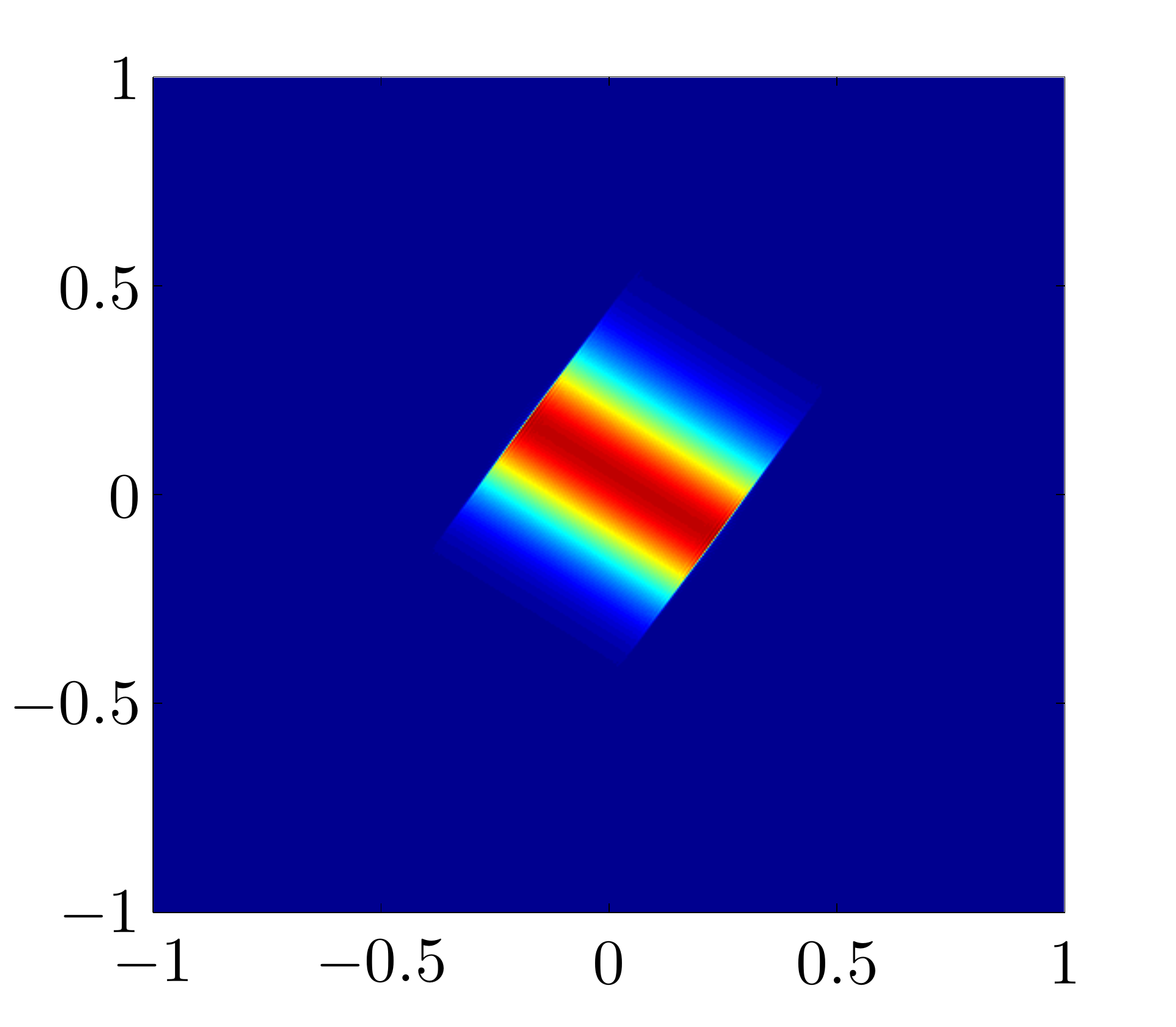}\reduceabovecaptionskip}%
\hfill%
\subcaptionbox{$j=0$: 181/10000\label{fig:sing_loc:0}}%
	{\includegraphics[width=0.33\textwidth]{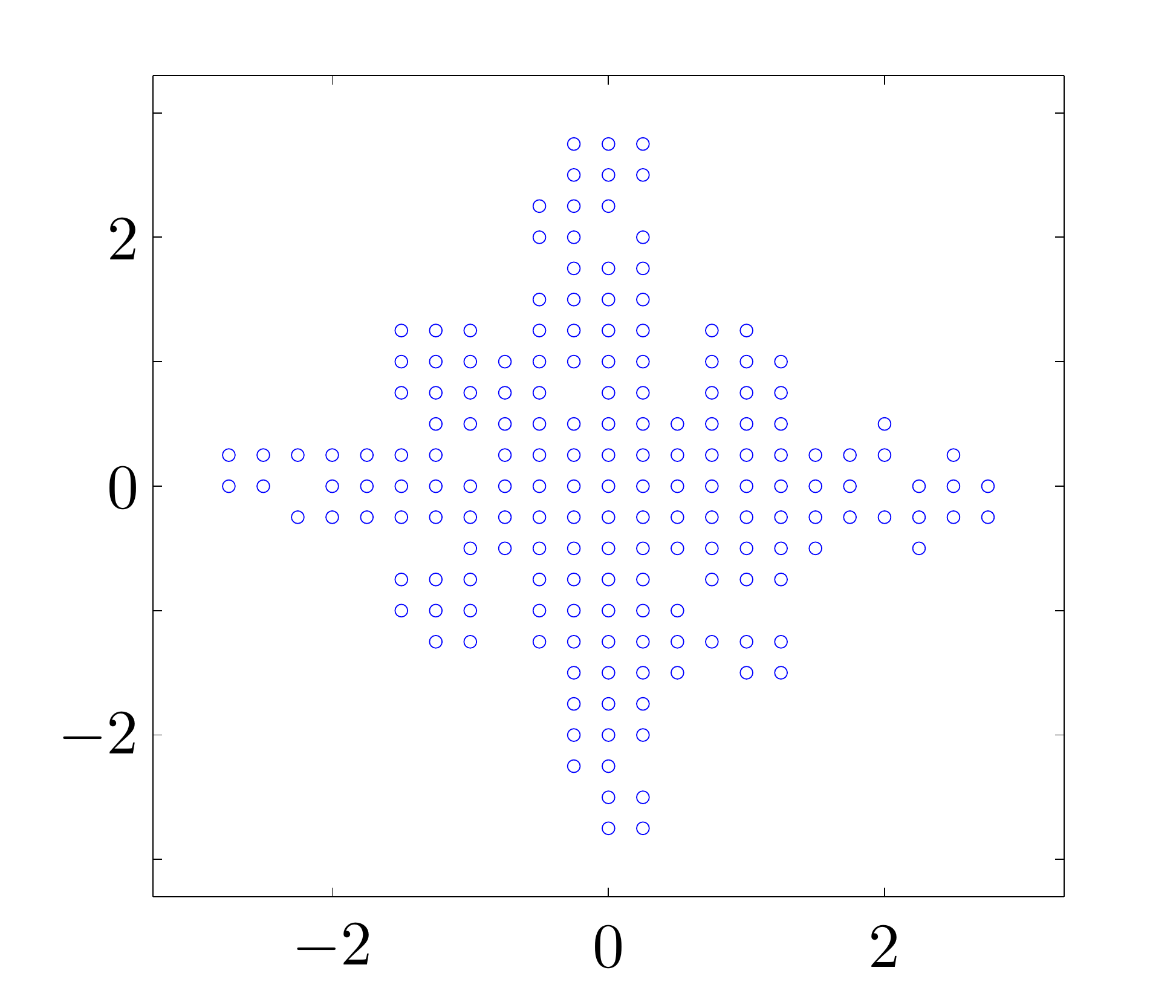}\reduceabovecaptionskip}%
\hfill%
\subcaptionbox{$j=1$: 2350/10000\label{fig:sing_loc:1}}%
	{\includegraphics[width=0.33\textwidth]{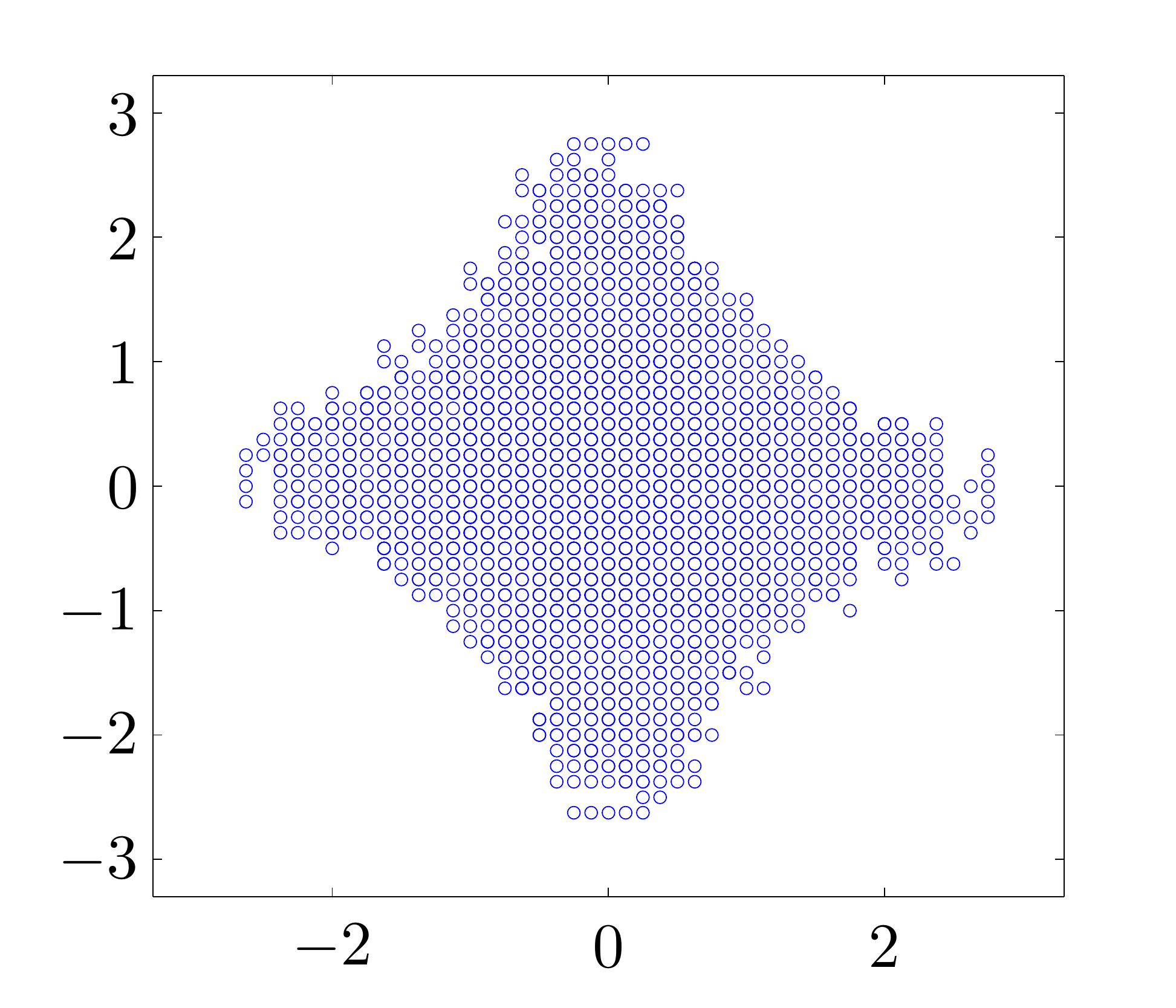}\reduceabovecaptionskip}%
\\[\reducebelow]% next line of pictures
\subcaptionbox{$j=2$: 2087/10000\label{fig:sing_loc:2}}%
	{\includegraphics[width=0.33\textwidth]{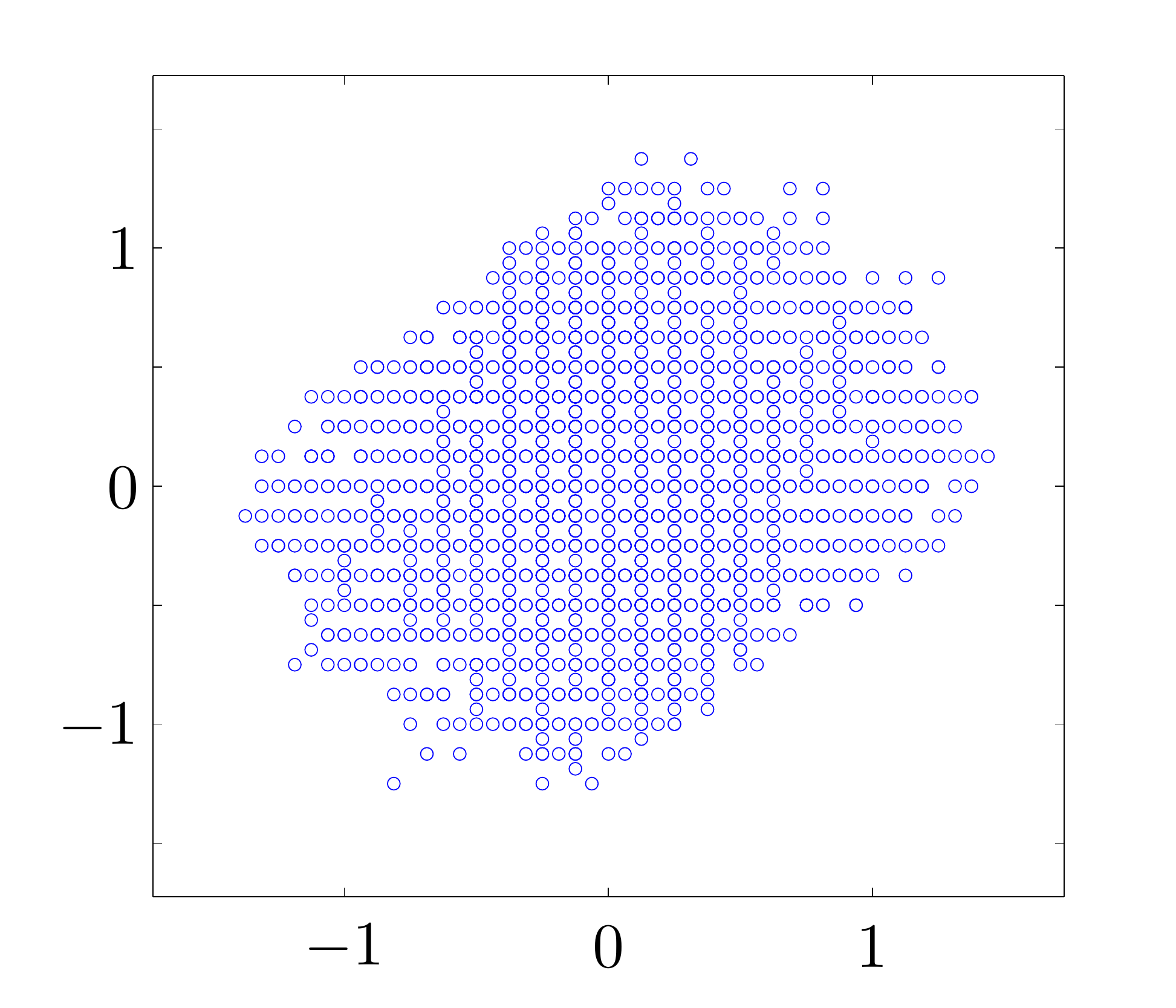}\reduceabovecaptionskip}%
\hfill%
\subcaptionbox{$j=3$: 1590/10000\label{fig:sing_loc:3}}%
	{\includegraphics[width=0.33\textwidth]{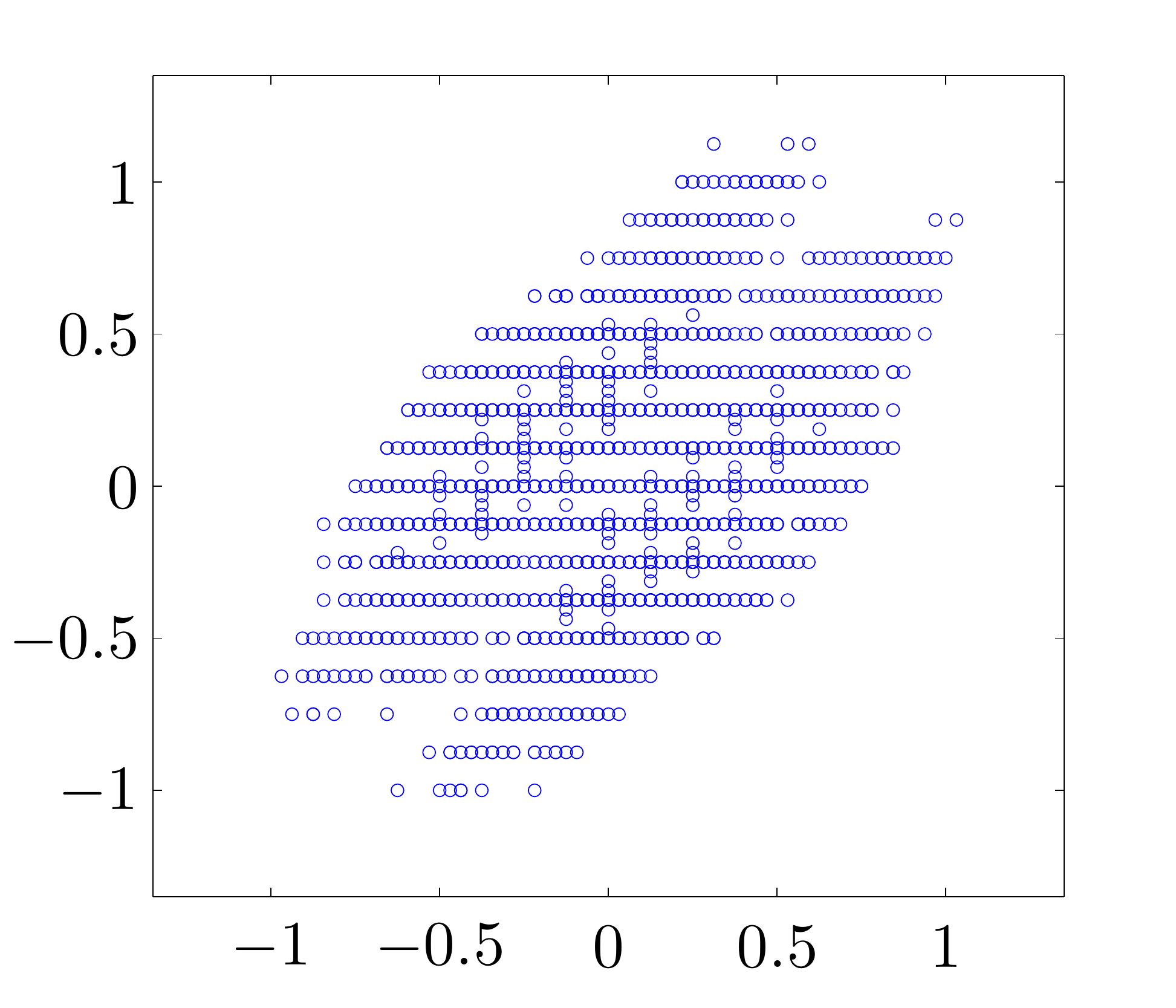}\reduceabovecaptionskip}%
\hfill%
\subcaptionbox{$j=4$: 1243/10000\label{fig:sing_loc:4}}%
	{\includegraphics[width=0.33\textwidth]{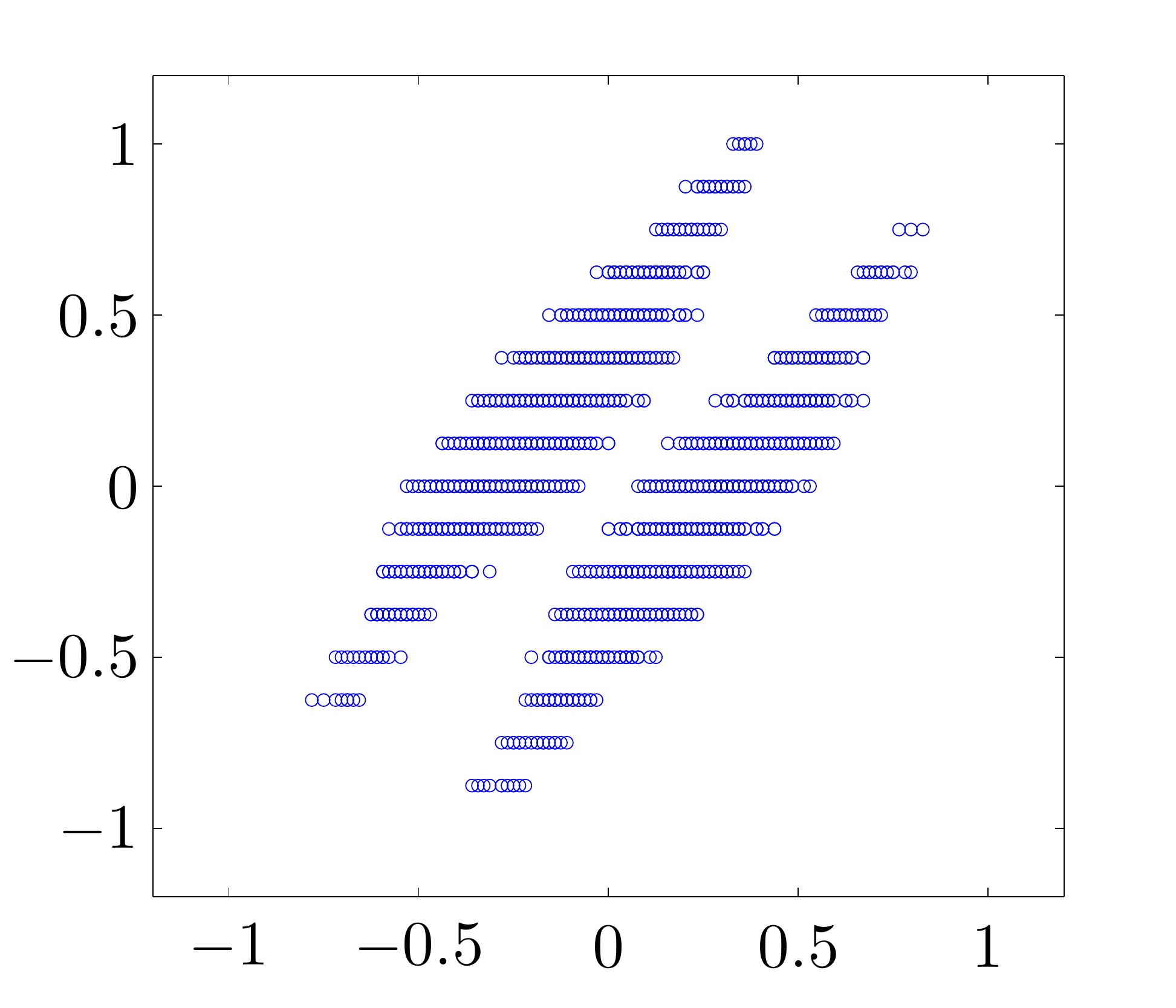}\reduceabovecaptionskip}%
\\[\reducebelow]% next line of pictures
\subcaptionbox{$j=5$: 912/10000\label{fig:sing_loc:5}}%
	{\includegraphics[width=0.33\textwidth]{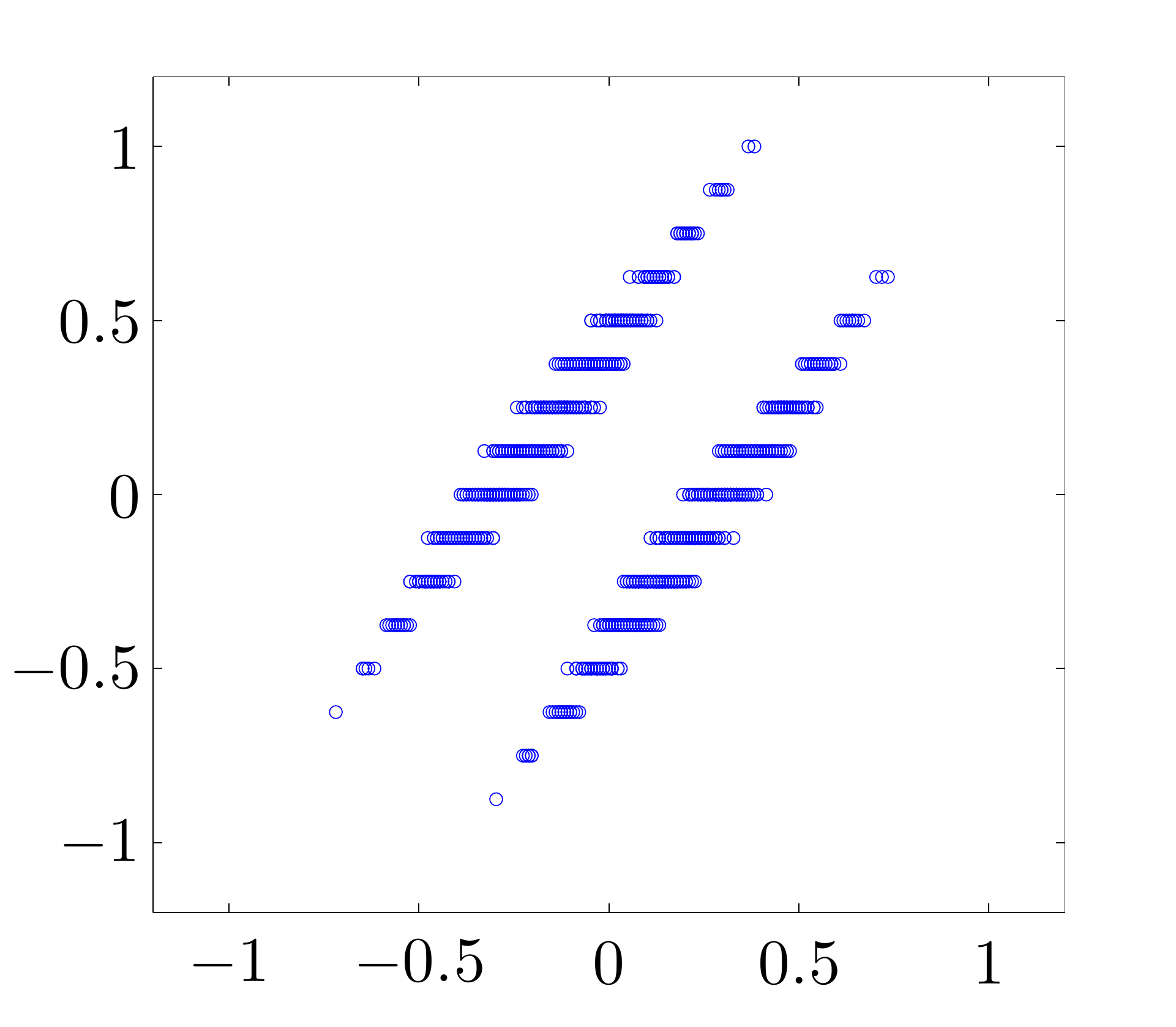}\reduceabovecaptionskip}%
\hfill%
\subcaptionbox{$j=6$: 620/10000\label{fig:sing_loc:6}}%
	{\includegraphics[width=0.33\textwidth]{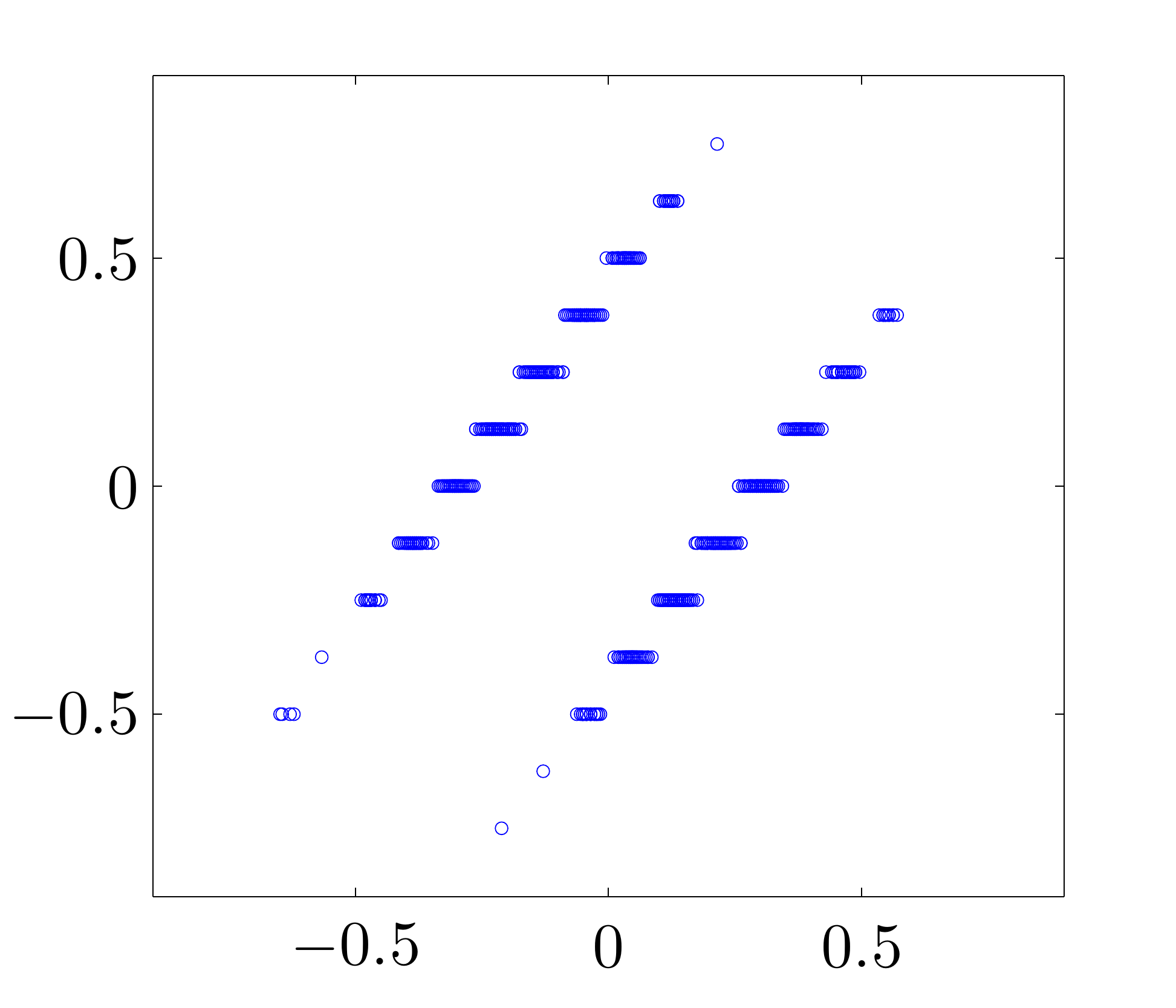}\reduceabovecaptionskip}%
\hfill%
\subcaptionbox{$j=7$: 417/10000\label{fig:sing_loc:7}}%
	{\includegraphics[width=0.33\textwidth]{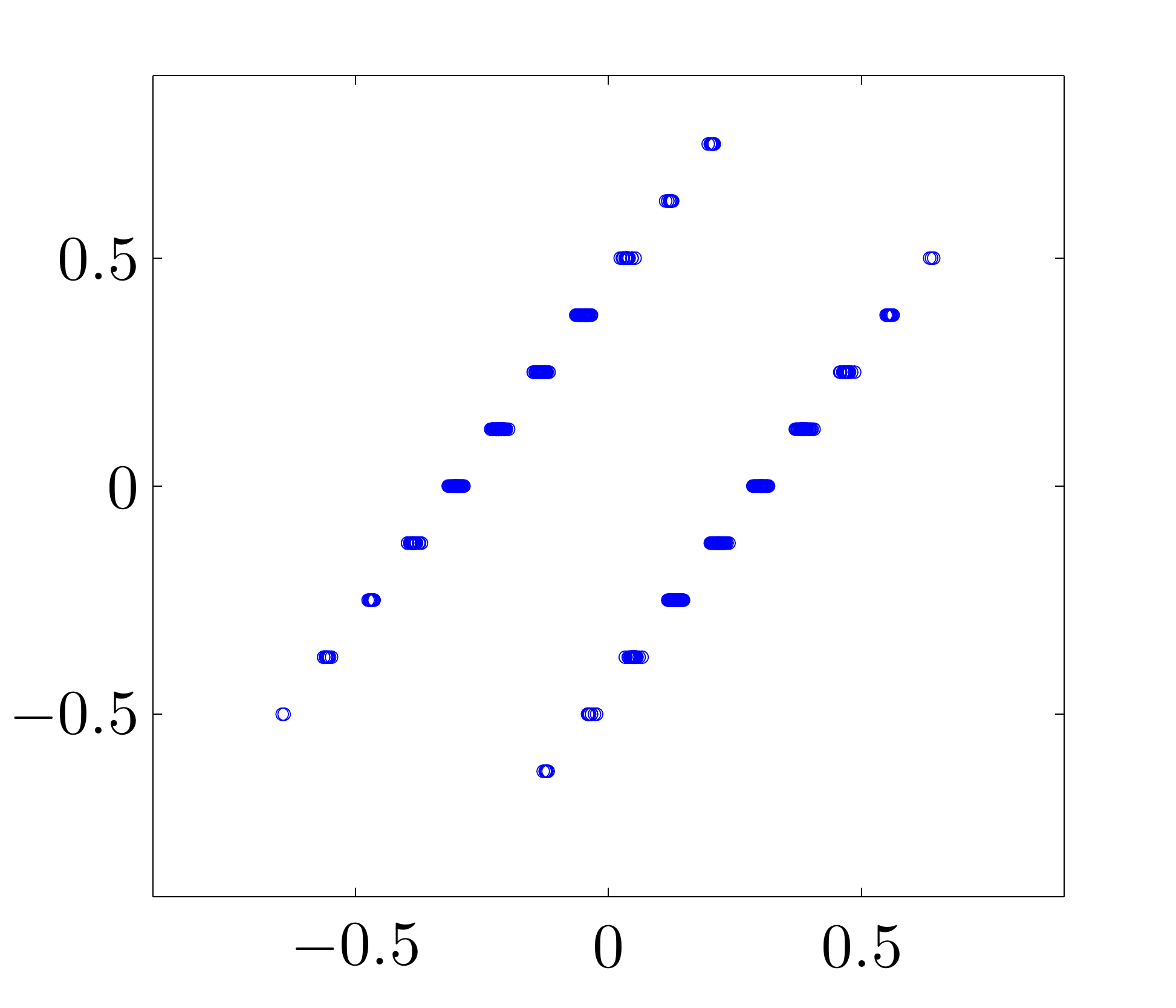}\reduceabovecaptionskip}%
\\[\reducebelow]% next line of pictures
\subcaptionbox{$j=8$: 252/10000\label{fig:sing_loc:8}}%
	{\includegraphics[width=0.33\textwidth]{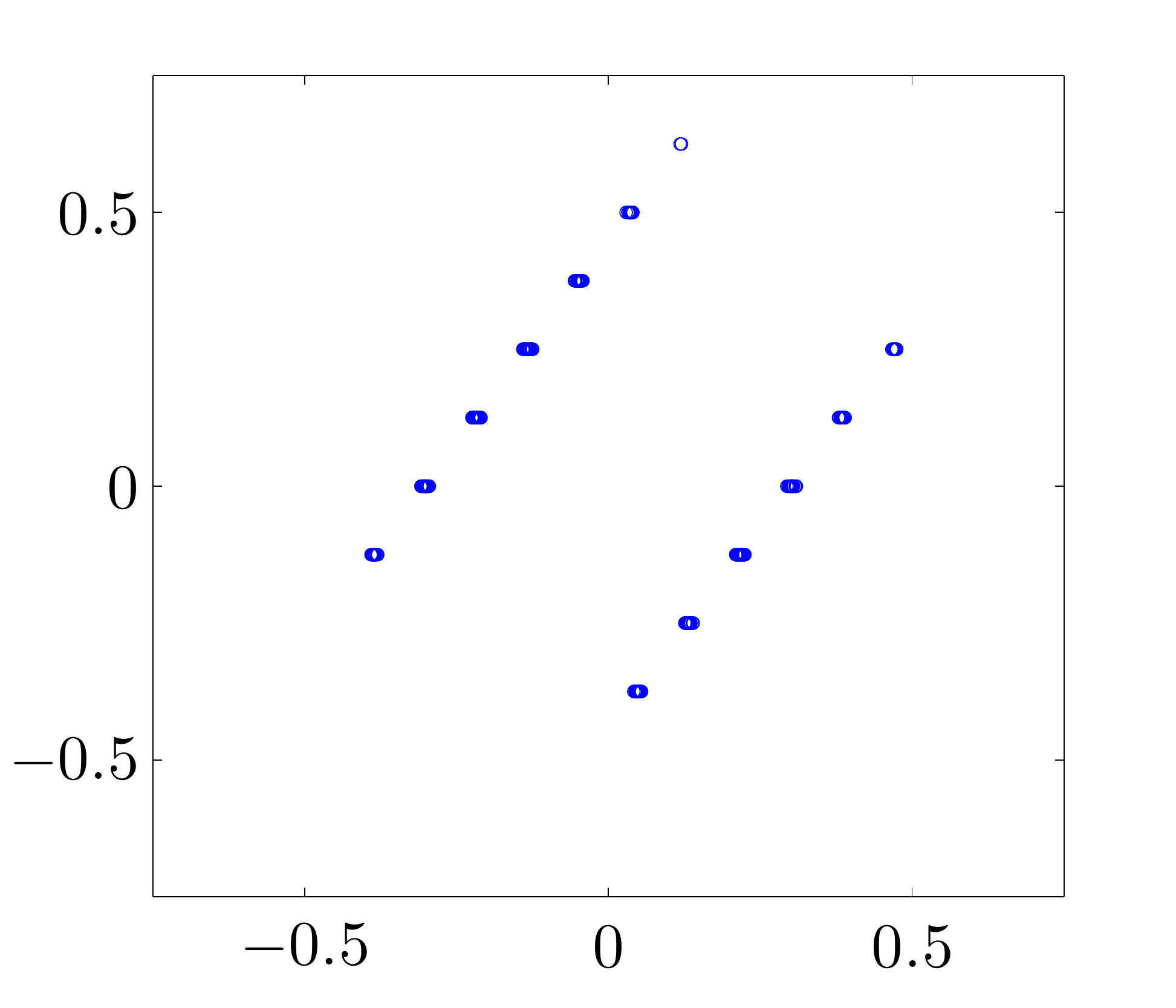}\reduceabovecaptionskip}%
\hfill%
\subcaptionbox{$j=9$: 231/10000\label{fig:sing_loc:9}}%
	{\includegraphics[width=0.33\textwidth]{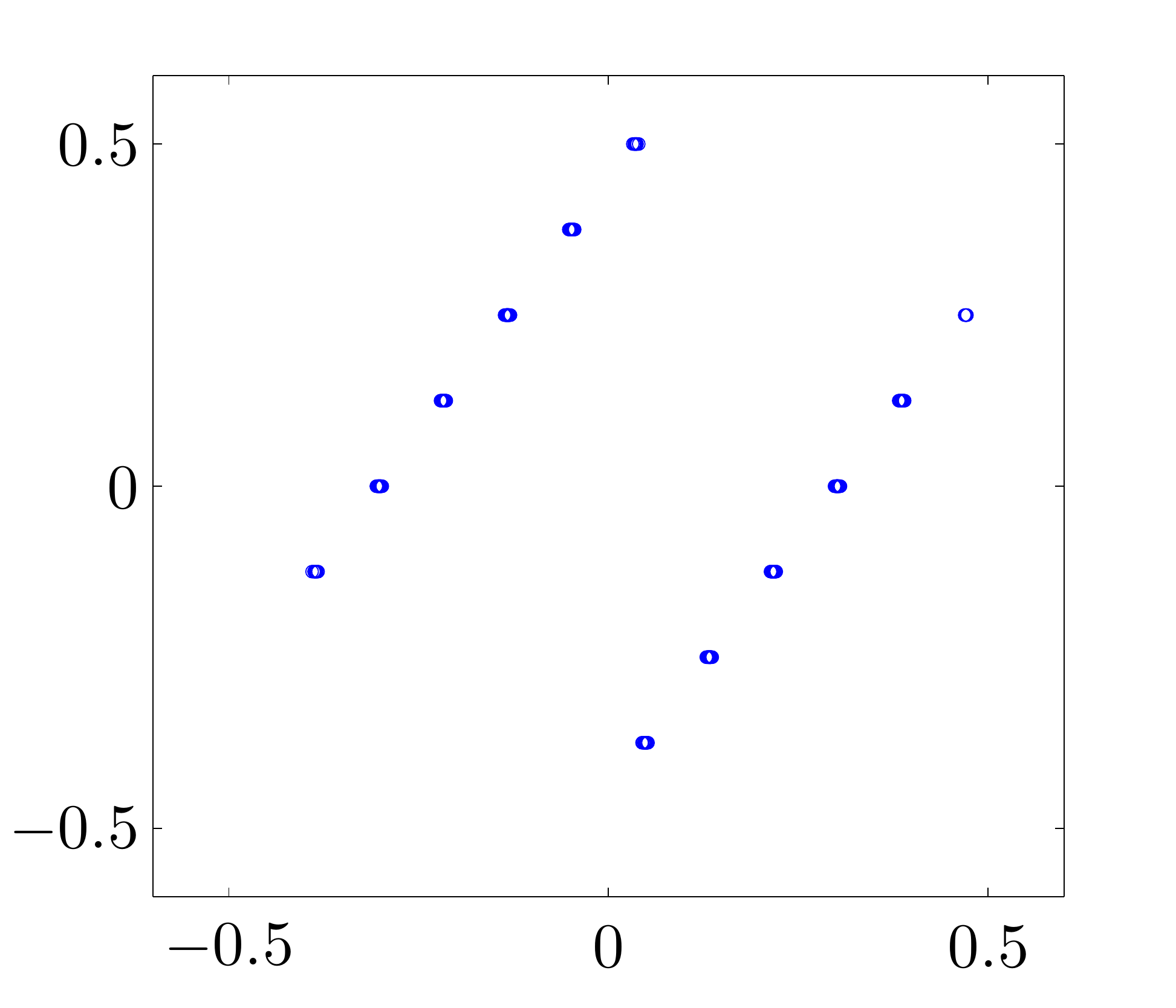}\reduceabovecaptionskip}%
\hfill%
\subcaptionbox{$j=10$: 117/10000\label{fig:sing_loc:10}}%
	{\includegraphics[width=0.33\textwidth]{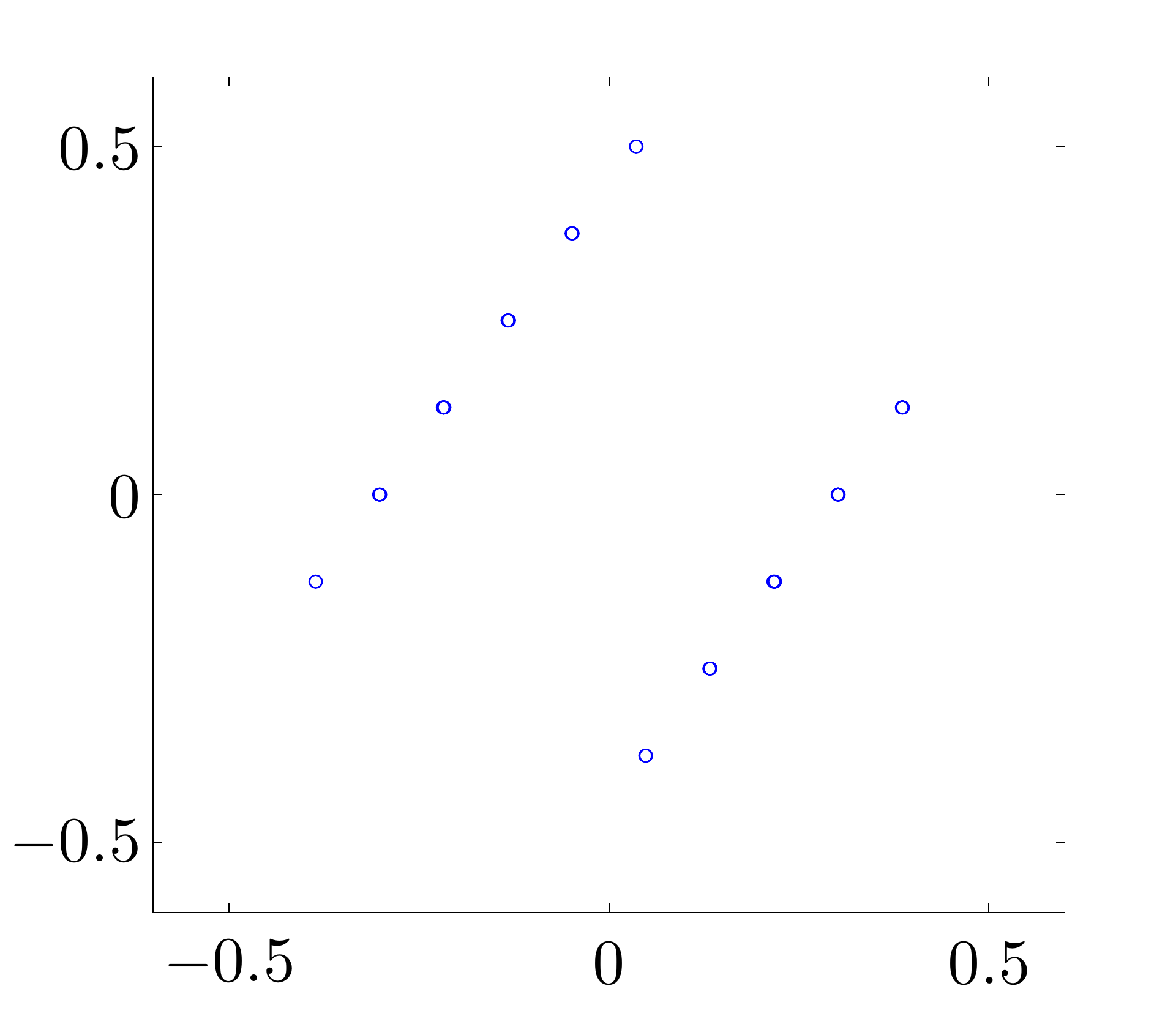}\reduceabovecaptionskip}%
\caption{Localisation of solution in \subref{fig:sing_loc:ref} in the ridgelet frame: Subplots \subref{fig:sing_loc:0}--\subref{fig:sing_loc:10} show the translations corresponding to the 10000 largest coefficients (up to scale $j=10$) within a given scale. At high scales, only coefficients close to the singularities are active -- as expected.% Note also how the extent diminishes with scale.
}\label{fig:sing_loc}
\end{figure}

\newpage

\begin{appendices}
%\addtocontents{toc}{\protect\setcounter{tocdepth}{0}}
\renewcommand{\theuniversalcounter}{\Alph{section}.\arabic{universalcounter}}

%\numberwithin{equation}{section}
\renewcommand{\theequation}{\Alph{section}.\arabic{equation}}

\renewcommand{\thedefinition} {\Alph{section}.\arabic{definition}}
\renewcommand{\theexample}    {\Alph{section}.\arabic{example}}
\renewcommand{\theremark}     {\Alph{section}.\arabic{remark}}
\renewcommand{\theremark}     {\Alph{section}.\arabic{remark}}
\renewcommand{\thelemma}      {\Alph{section}.\arabic{lemma}}
\renewcommand{\thecorollary}  {\Alph{section}.\arabic{corollary}}
\renewcommand{\theproposition}{\Alph{section}.\arabic{proposition}}
\renewcommand{\thefact}       {\Alph{section}.\arabic{fact}}
\renewcommand{\theassumption} {\Alph{section}.\arabic{assumption}}
\section{Geometric Considerations}\label{sec:geometric}

\subsection{Basic Properties of the Hypersphere}\label{app:hypsphere}

For various estimates, we need properties of the $(d-1)$-dimensional hypersphere $\bbSd=\set{\vec x \in \bbR^d}{ |x|=1}$, which we equip it with the geodesic metric
\begin{align*}
	\dist_\bbSd(\vec s, \vec s')=\arccos(\vec s \cdot \vec s').
\end{align*}
Naturally, an equivalent metric would make no difference other than changing some constants. For convenience we extend $\dist_\bbSd$ to a pseudo-metric on $\bbR^d$ by taking 
\begin{align*}
	\dist_\bbSd(\vec x, \vec x')=\arccos\parens*{ \CP_\bbSd(\vec x) \cdot \CP_\bbSd(\vec x')} = \arccos\parens{ \frac{\vec x}{|\vec x|} \cdot \frac{\vec x'}{|\vec x'|}}.
\end{align*}
\begin{remark}
	For $\vec s \in \bbSd$, straight-forward calculus shows that the geodesic metric is equivalent to the Euclidian metric,
	\begin{align}\label{eq:geod_eucl_equiv}
		|\vec s - \vec s'| \le \arccos(\vec s \cdot \vec s') \le \frac{\pi}{2}|\vec s - \vec s'|.
	\end{align}
\end{remark}
%
%\begin{proof}
%	Elementary calculus yields $\sqrt{2-2\cos\alpha}\le |\alpha|$ when measured in radians -- geometrically, this is due to the fact that a straight line connecting two points on the sphere is always shorter than the arc between the points along the surface (the left hand side is the length of the straight expressed in terms of the opening angle, while the angle can be identified with the length of the shortest possible arc, i.e. part of a great circle). Consequently,
%	%
%	\begin{align*}
%		\arccos(\vec s \cdot \vec s') \ge \sqrt{2-2 \vec s\cdot \vec s'} = \sqrt{\vec s\cdot\vec s + \vec s'\cdot\vec s' - \vec s \cdot \vec s' - \vec s'\cdot \vec s} = |\vec s-\vec s'|.
%	\end{align*}
%	%
%	As argued above, the arclength is always longer than the length of the straight -- however the ratio between the two cannot exceed $\frac{\pi}{2}$ (which is attained for $\vec s' = -\vec s$), thus the other inequality also follows.
%\end{proof}

The construction of the ridgelet frame uses window functions supported on ``balls'' in this metric space,
\begin{align*}%\label{eq:}
	B_\bbSd(\vec s, \alpha) := \set{\vec s'\in \bbSd}{\dist_\bbSd(\vec s, \vec s')<\alpha},
\end{align*}
appropriately called \emph{hyperspherical caps}. For $\alpha>\pi$, we define $B_\bbSd(\vec s, \alpha)$ as the whole sphere $\bbSd$. These are closely related with the solid angle corresponding to $\alpha$, which we estimate in the following lemma.%is a quantity we need to get a grip on for the estimation of the intersection numbers.
\begin{lemma}
	The $d$-dimensional solid angle corresponding to opening angle $\alpha$ can be estimated by
	\begin{align*}%\label{eq:}
		\Omega_d(\alpha) :=  \frac{\mu\parens*{ B_\bbSd(\vec s,\alpha_2)}}{\mu(\bbSd)} \lesssim \alpha^{d-1},
	\end{align*}
	where $\mu$ is the canonical surface measure of $\bbSd$. For $\alpha_1\le\frac\pi 2$ and arbitrary $\alpha_2>0$,
	\begin{align}\label{eq:hyp_cap_est}
		\frac{\mu\parens*{ B_\bbSd(\vec s,\alpha_2)}}{\mu\parens*{ B_\bbSd(\vec s,\alpha_1)}} \le \frac{C_d}{c_d}\parens{\frac{\alpha_2}{\alpha_1}}^{d-1},
	\end{align}
	where $c_d$, $C_d$ are constants that only depend on $d$.
\end{lemma}
\begin{proof}
	The area of the hyperspherical cap for arbitrary but fixed $\vec s\in\bbSd$ and opening angle $\alpha\in[0,\frac{\pi}{2}]$ can be calculated (see \cite{li}) as follows. The idea is that the intersection of $\bbSd$ with an affine hyperplane perpendicular to $\vec s$ is a $(d-2)$-dimensional sphere (if the intersection is not empty) -- all its points have the same angle to $\vec s$, say $\vartheta$, in which case the radius of this lower-dimensional sphere is $\sin\vartheta$. Integrating over this angle $\vartheta$ between $0$ and $\alpha$ will then yield the desired area. From this we obtain the solid angle by dividing through the area of the whole sphere. Denoting by 
	\begin{align*}%\label{eq:}
		S_i(r)= \frac{2\pi^\frac{i+1}{2}}{\Gamma(\frac{i+1}{2})}\, r^i
	\end{align*}
	the surface area of the $i$-dimensional hypersphere with radius $r$, we calculate
	\begin{align*}
	%\begin{split}\label{eq:solid_angle}
		\Omega_d(\alpha)
		&= \frac{1}{S_{d-1}(1)}\mu\parens*{B_\bbSd(\vec s, \alpha)} = \frac{1}{S_{d-1}(1)} \int_0^\alpha S_{d-2}(\sin\vartheta) \d\vartheta \\
		&= \frac{\Gamma(\frac{d}{2})}{2\pi^{\frac{d}{2}}}\frac{2\pi^{\frac{d-1}{2}}}{\Gamma(\frac{d-1}{2})} \int_0^\alpha \smash{\underbrace{(\sin \vartheta)}_{\le \vartheta}}^{d-2} \d\vartheta \le \frac{1}{B(\frac{d-1}{2},\frac{1}{2})} \frac{1}{d-1} \alpha^{d-1} \lesssim \alpha^{d-1},
	%\end{split}
	\end{align*}
	where $B(x,y)$ is the beta function.
%	Looking at the substitution we used, it is clear that this approach only works for $\alpha\le\frac{\pi}{2}$ (which we required above) -- however, it is easy to see that for $\alpha \in [\frac{\pi}{2},\pi]$, the solid angle is given by $\Omega^d_\alpha:=S_{d-1}(1)-\Omega^d_{\pi-\alpha}$.
%	
%	While the exact representation is not necessary for our purposes, we presented it here for the sake of completeness. Now we are able to estimate,
%	%
%	\begin{align*}%\label{eq:solid_angle_est}
%		\Omega_\alpha^d = \frac{1}{2 \, B(\frac{d-1}{2}, \frac{1}{2})} \int_0^\alpha  \smash{\underbrace{(\sin \vartheta)}_{\le \vartheta}}^{d-2} \d\vartheta
%		&\lesssim \frac{1}{d-1} \alpha^{d-1} \lesssim \alpha^{d-1}.
%	\end{align*}
	%
	The same argument can be used to yield (use e.g. $\sin \vartheta \ge \frac{\vartheta}{2}$ for the lower estimate), for $\alpha \in [0,\frac{\pi}{2}]$,
	\begin{align*}
		c_d \, \alpha^{d-1} \le \mu\parens*{ B_\bbSd(\vec s,\alpha)} \le C_d\, \alpha^{d-1}.
	\end{align*}
	Subsequently, for two angles $\alpha_1, \, \alpha_2$, the inequality
	\begin{align*}
		\mu\parens*{ B_\bbSd(\vec s,\alpha_2)} \le C_d \,\alpha_2^{d-1} \le \frac{C_d}{c_d}\parens{\frac{\alpha_2}{\alpha_1}}^{d-1}\mu\parens*{ B_\bbSd(\vec s,\alpha_1)}
	\end{align*}
	holds as long as $\alpha_1\le\frac \pi2$, which finishes the proof.
	%[wenn man im Beweis der ``Intersection numbers" den Fall $\alpha_{j=0}=2$ auch noch mit abdecken will, kann man nat\"urlich eine Absch\"atzung des Sinus nach unten nehmen, die bis $\vartheta=2$ g\"ultig ist. Z.B. w\"are obige Ungleichung f\"ur alle $\alpha_1\le\frac{3\pi}{4}$ g\"ultig, wenn man $\sin \vartheta \ge \frac{\sqrt{8}\vartheta}{3\pi}$ nimmt]
\end{proof}

\subsection{Construction of the $\oldvec s_{j,\ell}$}\label{app:construction_sjl}

The construction of the $\psi_{j,\ell}$ (see \cite{grohs1}) requires a sequence of points on the sphere with particular properties. The following proposition collects these properties and some consequences.
\begin{proposition}\label{prop:construction_sl}
	For fixed $\alpha > 0$, there exists a sequence $\{\vec s_\ell\}_{\ell \in \{0,\ldots,L\}}$ such that
	\begin{align*} %\label{eq:construction_sl}
		\bigcup_{\ell=0}^L B_\bbSd(\vec s_\ell,\alpha) = \bbSd, && B_\bbSd \parens[\Big]{\vec s_\ell,\frac{\alpha}{3} } \text{ are pairwise disjoint},
	\end{align*}
	and $L\lesssim \parens{\frac{1}{\alpha}}^{d-1}$. Additionally, for an arbitrary cap of  opening angle $\alpha'$ (and possibly using dilation $q,q'>0$), the number of non-empty intersections of the sequence with this cap can be estimated by
	\begin{align}\label{eq:est_intersect_sl}
		\# \set[\Big]{ \ell \in\{0,\ldots L\}}{ B_\bbSd(\vec s_\ell,q\alpha) \cap B_\bbSd(\vec s^{\?\?\prime},q'\alpha')} \le \frac{\mu\parens*{B_\bbSd(\vec s^{\?\?\prime},3(q\alpha)_>) }}{\mu\parens*{B_\bbSd(\vec s_{\ell},\frac{\alpha}{3})}},
	\end{align}
	where $(q\alpha)_>:=\max(q\alpha,q'\alpha')$. %$\alpha_<:=\min(\alpha,\alpha')$
\end{proposition}
\begin{proof}
	The construction of the sequence (and the idea for the estimate below) can be found in \cite{borup}. We note that for $\alpha>\pi$, we simply choose $\vec s_0:=\vec e_1$.
	
	To estimate the number of intersections for two such sequences, we let $\nu(\vec s^{\?\?\prime} \!,\alpha,\alpha'):=\set*{\vec s \in \bbSd}{ B_\bbSd(\vec s,\alpha) \cap B_\bbSd(\vec s^{\?\?\prime} ,\alpha') \neq \emptyset}$ and see that
	\begin{align*}
		\bigcup_{\vec s \in \nu(\vec s^{\?\prime}\!,\alpha,\alpha')} B_\bbSd(\vec s, \alpha) \subseteq B_\bbSd(\vec s^{\?\?\prime},\alpha'+2\alpha) \subseteq B_\bbSd(\vec s^{\?\?\prime},3\alpha_>).
	\end{align*}
	In particular, all member sets of our covering having non-empty intersection with  $B_\bbSd(\vec s^{\?\?\prime}, \alpha')$ are contained in $B_\bbSd(\vec s^{\?\?\prime}, 3\alpha_>)$. Consequently, the number of non-empty intersections $B_\bbSd(\vec s_{\ell},\alpha) \cap B_\bbSd(\vec s^{\?\?\prime},\alpha')$ can be estimated by assuming that $B_\bbSd(\vec s^{\?\?\prime},3\alpha_>)$ is perfectly filled out by the disjoint sets $B_\bbSd(\vec s_{\ell},\frac{\alpha}{3})$. In other words,
	\begin{align*}
		\# \set[\Big]{ \ell \in\{0,\ldots L\}}{ B_\bbSd(\vec s_\ell,\alpha) \cap B_\bbSd(\vec s^{\?\?\prime},\alpha')}
		\le \frac{\mu\parens*{B_\bbSd(\vec s^{\?\?\prime},3\alpha_>) }}{\mu\parens*{B_\bbSd(\vec s_{\ell},\frac{\alpha}{3})}}.
	\end{align*}
	In particular, by setting $\alpha'>\frac\pi3$, we obtain
	\begin{align*}
		\# \{0,\ldots, L\} \le \frac{\mu\parens*{\bbSd }}{\mu\parens*{B_\bbSd(\vec s_{\ell},\frac{\alpha}{3})}} \stackrel{\eqref{eq:hyp_cap_est}} \lesssim \frac1{\alpha^{d-1}},
	\end{align*}
	which is the desired estimate. In the case that dilations $q,q'>0$ are applied after the construction, we argue in a similar fashion, now using the disjoint sets $B_\bbSd(\vec s_{\ell},\frac{\alpha}{3})$ to fill out the dilated sets $B_\bbSd(\vec s^{\?\?\prime},3(q\alpha)_>)\supseteq B_\bbSd(\vec s^{\?\?\prime}, q\alpha + 2q'\alpha')$, thus
	\begin{align*}
		\# \set[\Big]{ \ell \in\{0,\ldots L\} }{ B_\bbSd(\vec s_\ell,q\alpha) \cap B_\bbSd(\vec s^{\?\?\prime},q'\alpha') }
		\le \frac{\mu\parens*{B_\bbSd(\vec s^{\?\?\prime},3q_> \alpha_>) }}{\mu\parens*{B_\bbSd(\vec s_{\ell},\frac{\alpha}{3})}}.
	\end{align*}
	This finishes the proof.
\end{proof}

\subsection{Properties of $U_\jl$ and $P_\jl$}\label{sec:U_jl_and_P_jl}
\begin{lemma}\label{lem:rot_lipschitz}
	For $\vec{s},\ \vec{s}'\in \bbSd$ 
	there exist rotations $R_{\vec{s}},\ R_{\vec{s}'}$ which 
	map $\vec{e}_1$ to $\vec{s},\ \vec{s}'$ respectively, such that
	\begin{align}\label{eq:app:rot_lipschitz}
		\norm{R_{\vec{s}} - R_{\vec{s}'}}\lesssim \dist_\bbSd(\vec{s},\vec{s}').
	\end{align}
\end{lemma}
\begin{proof}
	We only consider the case $d>2$ since the other cases are trivial.
	\\
	\noindent {\bf Step 1:}
	For any fixed $\delta >0$ it suffices to consider points with
	\begin{equation}\label{eq:distclose}
		\dist_\bbSd(\vec{s},\vec{s}')< \delta.
	\end{equation}
	To see this, assume that 
	\[
		\dist_\bbSd(\vec{s},\vec{s}') \geq \delta.
	\]
	In this case we have the trivial estimate
	\[
		\norm{R_{\vec{s}} - R_{\vec{s}'}}\le 
		\norm{R_{\vec{s}}} + \norm{R_{\vec{s}'}}=2
		\le \frac{2}{\delta}\delta \le 
		\frac{2}{\delta}\dist_\bbSd(\vec{s},\vec{s}').
	\]
	Therefore it is no loss in generality to assume that
	\eqref{eq:distclose} holds for a suitably small and henceforth fixed $\delta >0$.
	This implies, by applying a suitable rotation, that
	both $\vec{s},\ \vec{s}'$ can be assumed to lie in a fixed small neighborhood around
	$\vec{e}_1$.
	\\
	\noindent {\bf Step 2:}
	We first consider the case $d = 3$.
	Then we can write
	each $\vec{t}\in \bbSd$ as
	\[
		\vec{t} = \parens*{\cos(\theta)\cos(\psi),-\cos(\theta)\sin(\psi),\sin(\theta) }^\top.
	\]
	For $\vec{t}$ in a sufficiently small neighborhood of $\vec{e}_1$, the assignment
	\[
		\Phi\colon \vec{t}\mapsto (\theta,\psi)
	\]
	is smooth.
	Define
	\[
		R_{\vec{t}}:=\parens{\begin{array}{ccc}
		\phantom{-}\cos(\theta)\cos(\psi) & \sin(\psi) & -\sin(\theta)\cos(\psi) \\
		-\cos(\theta)\sin(\psi) & \cos(\psi) & \phantom{-}\sin(\theta)\sin(\psi) \\
		\sin(\theta) & 0 & \cos(\theta)
		\end{array}},
	\]
	where $(\theta,\psi)$ are defined by $\Phi(\vec{t})$. 
	Since $\Phi$ is smooth, the assignment $\vec{t}\mapsto R_{\vec{t}}$
	is smooth in a neighborhood of $\vec{e}_1$ and therefore Lipschitz --
	this implies \eqref{eq:app:rot_lipschitz} for $d = 3$.
	\\
	\noindent {\bf Step 3:}
	Now assume $d$ general. Pick an ONB $\curly{\vec{e}_1,\vec{e}_2',\dots , \vec{e}_d'}$ such that $\vec{s}, \vec{s}'$ lie in the span of $\vec{e}_1,\ \vec{e}_2',\ \vec{e}_3'$ and set
	\[
		\CE_{\vec{s},\vec{s}'}=\parens*{\vec{e}_1 \big|
		\vec{e}_2'\big| \dots \big| \vec{e}_d'}.
	\]
	Using this coordinate system, we may define the
	matrix $R_{\vec{t}}$ for any $\vec{t}
	= v_1(\vec{t}) \vec{e}_1 + v_2(\vec{t}) \vec{e}_2' + v_3(\vec{t}) \vec{e}_3' \in 
	\spann\curly{\vec{e}_1,\ \vec{e}_2',\ \vec{e}_3'}\cap \bbSd$ -- sufficiently close to $\vec{e}_1$ -- as follows:
	\[
		R_{\vec{t}} \colon \bbR^d \ni v
		\mapsto \CE_{\vec{s},\vec{s}'}
		\parens{R_{\parens{v_1(\vec{t}),v_2(\vec{t}),v_3(\vec{t})}^{\top}}\times \bbI_{d-3}} \CE_{\vec{s},\vec{s}'}^{-1} v
		\in \bbR^d,
	\]
	where the matrix $R_{\parens{v_1(\vec{t}),v_2(\vec{t}),v_3(\vec{t})}^{\top}}
	\times \bbI_{d-3}\in \bbR^{d\times d}$
	applies $R_{\parens{v_1(\vec{t}),v_2(\vec{t}),v_3(\vec{t})}^{\top}}$
	-- as defined above for three dimensions -- to the first three coordinates and leaves the other
	coordinates invariant.
	This matrix maps $\vec{e}_1$ to $\vec{t}$ as desired and it is smooth 
	in $\vec{t}$ with the same Lipschitz constant as the matrix $R_{\parens{v_1(\vec{t}),v_2(\vec{t}),v_3(\vec{t})}^{\top}}$.
	Therefore we may use this Lipschitz property to establish that
	\[
		\norm{R_{\vec{s}} - R_{\vec{s}'}} \lesssim \dist_\bbSd(\vec{s},\vec{s}')
	\]
	as required.
\end{proof}

\begin{lemma}\label{lem:U_jl}
	For the matrix $U_\jl= \Rjl^{-1} D_{2^{-j}}$ we have the inverse estimate
	\begin{align}\label{eq:app:inv_Ujls}
		\abs*{U_\jl^{-1}\vec s} \le w(\lambda),
	\end{align}
	where the $w(\lambda) = 1+2^j |\vec s \cdot \sjl|$ is again the weight of the preconditioning matrix.
	
	Additionally, for $U_\jld= \Rjld^{-1} D_{2^{-j}}$ with $\Rjld$ such that \eqref{eq:app:rot_lipschitz} holds for $\vec s = \sjl$ and $\vec s' = \sjld$, we have
	\begin{align}
		\norm*{U_\jl^{-1} U_\jld} &\lesssim \max(2^{j-j'},1) + 2^j \dist_\bbSd (\sjl, \sjld), \label{eq:app:U_lld}
		\intertext{%
	and
		}
		\abs*{U_\jl^{-1}\vec s } &\lesssim \max(2^{j-j'},1)\parens*{w(\lambda') + 2^{j'} \dist_\bbSd (\sjl, \sjld) }. \label{eq:app:inv_Ujls_wjld}
	\end{align}
\end{lemma}
\begin{proof}
	For \eqref{eq:app:inv_Ujls}, the components have to be computed individually (using the orthogonality of the rotation),
	\begin{align*}%\label{eq:}
		U_\jl^{-1}\vec s = D_{2^j}\Rjl \vec s = D_{2^j} \!\begin{pmatrix}
			\vec e_1 \cdot \Rjl\vec s \\ \vec e_2 \cdot \Rjl\vec s \\ \vdots
		\end{pmatrix}\! = D_{2^j} \!\begin{pmatrix}
			\Rjl^{-1} \vec e_1 \cdot \vec s\, \\ \Rjl^{-1} \vec e_2 \cdot \vec s\, \\ \vdots
		\end{pmatrix}\! = D_{2^j} \!\begin{pmatrix}
			\phantom{\!\Rjl^{-1} \vec e_2 \cdot \vec s\,} \mathllap{\sjl \cdot \vec s\,} \\ \!\Rjl^{-1} \vec e_2 \cdot \vec s\, \\ \vdots
		\end{pmatrix}\! = \!\begin{pmatrix}
			\phantom{\!\Rjl^{-1} \vec e_2 \cdot \vec s\,} \mathllap{2^j \sjl \cdot \vec s\,} \\ \!\Rjl^{-1} \vec e_2 \cdot \vec s\, \\ \vdots
		\end{pmatrix},
	\end{align*}
	and consequently, since all but the first component have modulus less than $1$,
	\begin{align*}
		\abs*{U_\jl^{-1}\vec s } \le \max\parens*{2^j|\vec s\cdot \sjl|, \, 1} \le w(\lambda).
	\end{align*}
	
	Denoting the identity by $\bbI$, we begin the proof of \eqref{eq:app:U_lld} by considering the matrix $\Rjl \Rjld^{-1}$ -- exploiting the orthogonality of the $\Rjl$ and \autoref{lem:rot_lipschitz} to yield
	\begin{align*}
		\norm*{\Rjl \Rjld^{-1} - \bbI} = \norm*{\Rjld - \Rjl} \lesssim \dist_\bbSd (\sjl, \sjld).
	\end{align*}
	Thus, we can estimate
	\begin{align*}
		\norm*{U_\jl^{-1} U_\jld}
		&= \norm*{D_{2^{j}}\Rjl \Rjld^{-1} D_{2^{-j'}}} = \norm*{D_{2^{j-j'}} + D_{2^{j}} (\Rjl \Rjld^{-1} -\bbI) D_{2^{-j'}}} \\
		&\lesssim \max(2^{j-j'},1) + 2^j \dist_\bbSd (\sjl, \sjld).
	\end{align*}

	Finally, for the last inequality, we compute
	\begin{align*}
		U_\jl^{-1} = D_{2^j}\Rjl = D_{2^{j-j'}} D_{2^{j'}} \Rjl \Rjld^{-1}\Rjld = D_{2^{j-j'}}\parens*{U_\jld^{-1}+ D_{2^{j'}} (\Rjl\Rjld^{-1}-\bbI) \Rjld },
	\end{align*}
	and after multiplying with $\vec s$ and taking the modulus, we use the above results to arrive at
	\begin{align*}
		\abs*{U_\jl^{-1}\vec s } \lesssim \max(2^{j-j'},1) \parens*{w(\lambda') + 2^{j'} \dist_\bbSd (\sjl, \sjld) }, 
	\end{align*}
	which is what we wanted to prove.
\end{proof}
\begin{proposition}\label{prop:bounding_cylinder}
	For $j\ge 1$, the transformation $U_\jl^\top$ takes the ``tiles" $P_\jl$ back into a bounded set around the origin (illustrated in \autoref{fig:UjlPjl}),
	\begin{align}\label{eq:app:trafo_P_jl}
		U_\jl^\top P_\jl %&\subseteq \set[\bigg]{\vec \eta \in \bbR^d}{ \eta_1 \in \bracket[\Big]{\frac{1}{2}\cos (\alpha_j), 2}, \, \sqrt{\eta_2^2+\ldots+\eta_d^2} \le \max\parens*{2^j \tan(\alpha_j) \eta_1,4}} \\
		&\subseteq \bracket[\bigg]{\frac{1}{2}\cos (\alpha_j), 2} \times \CP_{(\spann\{\vec e_1\})^\bot}\parens*{B_{\bbR^d}(0,4)} \subseteq B_{\bbR^d}(0,5).
	\end{align}
	The Minkowski sums $P_\jl^m:= P_\jl + B_{\bbR^d}(0,2^m)$ behave similarly,
	\begin{align}\label{eq:app:trafo_P_jlm}
		U_\jl^\top(P_\jl^m)\subseteq B_{\bbR^d}(0,5+2^m).
	\end{align}
	More importantly, we can calculate the opening angle of the cone containing $P_\jl$ as follows,
	\begin{align}\label{eq:app:est_alpha_j_m}
		\alpha_j^m  = \alpha_j+\arcsin \parens[\bigg]{\frac{2^m}{2^{j-1}}} \le c_\omega 2^{m-j},
	\end{align}
	as long as $j\ge m+1$, where $c_\omega\le \pi+2$.
\end{proposition}
\begin{figure}
\begin{tikzpicture}[xscale=2]
%	\clip (-1,-5) rectangle (3,8);
	\draw[<->] (-0.5,0) -- (2.5,0);
	\draw[<->] (0,-4) -- (0,4);
%	\draw[red, very thin] (0,0) circle (5);
	
	\draw (1/4,4) node[below right]{$\big(\frac14,4\big)$} -- (2,4) node[below right]{$\big(2,4\big)$}
		-- (2,-4) -- (1/4,-4) -- cycle;
	
	\foreach \j in {1,2,4}
	{
	\pgfmathsetmacro{\r}{2^(\j)}
	\pgfmathsetmacro{\opang}{2*2^(-\j)/pi*180}
	\begin{scope}[xscale=1/\r]
	\fill[opacity=0.2,black] (\r/2,0) arc [start angle=0, end angle=\opang,radius=\r/2]
		-- (\opang:2*\r) arc [start angle=\opang, end angle=-\opang,radius =2*\r]
		-- (-\opang:\r/2) arc [start angle=-\opang, end angle=0,radius=\r/2];
	\end{scope}
	}
\end{tikzpicture}
\caption{$U_\jl^\top P_\jl$ for arbitrary $\ell$ and $j=1,\ldots,3$. For better legibility, the $y$-axis is scaled down by a factor of 2.\label{fig:UjlPjl}}
\end{figure}
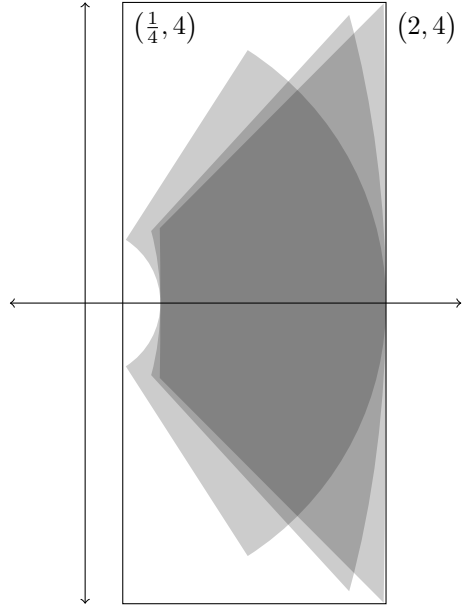
\begin{proof}
	From the definition of $\psi_\jl$ (see \eqref{eq:def_psi_jl}), we see that its support $P_\jl$ is contained in the intersection between a spherical shell (between radii $2^{j-1}$ and $2^{j+1}$) and a cone around $\sjl$ with opening angle $\alpha_j=2^{-j+1}$. The rotation in $U_\jl^\top=D_{2^{-j}}\Rjl$ brings the axis of this cone into $\vec e_1$. We see that the smallest value of $\eta_1$ for $\vec \eta \in U_\jl^\top P_\jl$ is $2^{-j} 2^{j-1} \cos(\alpha_j)=\frac 12 \cos (\alpha_j)>\frac 14$ since $\alpha_j=2^{-j+1}\le 1 < \frac\pi3$ for $j\ge1$.
	
	The largest extent perpendicular to $\vec e_1$ can be calculated as
	\begin{align*}
		2^{j+1}\cos \alpha_j \sin \alpha_j = 2^j \sin 2\alpha_j \le 2^j \cdot 2\alpha_j= 4,
	\end{align*}
	which proves \eqref{eq:app:trafo_P_jl}. \eqref{eq:app:trafo_P_jlm} follows immediately because the contraction $D_{2^{-j}}$ can not enlarge the distance $2^m$ to $P_\jl$. We note that choosing a different
	\begin{align*}%\label{eq:}
		\tilRjl = \begin{pmatrix} 1 & 0 \\ 0 & \wt R \end{pmatrix} \Rjl 
	\end{align*}
	with $\wt R\in\mathrm{SO}(d-1)$ yields exactly the same set, since the rotation $\wt R$ leaves disks (in $d-1$ dimensions) invariant, i.e. $\wt R\, B_{\bbR^{d-1}}(0,4) = B_{\bbR^{d-1}}(0,4)$.
	
\begin{figure}
\begin{tikzpicture}[scale=1.5,%
	underext/.style={decoration={pre=moveto,pre length=#1,post=moveto,post length=#1}}]
\pgfmathsetmacro\ang{asin(1.5/6)}
\pgfmathsetmacro\angg{asin(3.5/6)}
\clip (-\ang:6) arc (-\ang:90-\angg:6) -- ($(90-\angg:6)-(4.5,0)$) -- (-1,-1.5) -- cycle;

\foreach \r / \i [evaluate=\r as \ang using 90/\r] in {8/3} % radius and number of slice
{
\begin{scope}[rotate={(\i-1)*\ang}]
	\foreach \m in {2} 
	{
	% Minkowski sum
	\filldraw [name path=shell] [fill=black!5,draw=black] ($(-\ang:\r/2)+(-\ang-90:\m)$) arc (-\ang-90:-\ang-180:\m)
	arc (-\ang:\ang:\r/2-\m) arc (\ang+180:\ang+90:\m)
	-- ($(\ang:2*\r)+(\ang+90:\m)$) arc (\ang+90:\ang:\m)
	arc (\ang:-\ang:2*\r+\m) arc (-\ang:-\ang-90:\m) -- cycle;
	\node at (2*\ang:\r/2+1) {$P_{j,l}^m$};
	
	% support
	\filldraw[draw=black,fill=black!20] (-\ang:\r/2) arc (-\ang:\ang:\r/2) -- (\ang:2*\r) arc (\ang:-\ang:2*\r) -- cycle;
	\node at (0:\r/2+1) {$P_{j,l}$};
	
	% alpha_m
	\pgfmathsetmacro\angg{asin(2*\m/\r)+\ang}
	\pgfmathsetmacro\alphr{22}
	\path [name path=alpham] (0,0)--(\angg:5);
	\draw [name intersections={of=shell and alpham, by=B2}] [very thin] (0,0) -- (B2);
	\filldraw[fill=green!20,draw=black,very thin] (0,0)--(0:\alphr pt) arc (0:\angg:\alphr pt) -- cycle;
	\node at (\angg/2-2:\alphr-7 pt) {$\alpha_j^m$};
	
	% other side
	\draw[very thin] (0,0) -- ($(0,0)!1!-2*\angg:(B2)$);
	\filldraw[fill=green!20,draw=black,very thin] (0,0)--(0:\alphr/2 pt) arc (0:-\angg:\alphr/2 pt) -- cycle;
	
	% alpha_j
	\draw[very thin] (\ang:\alphr pt)--(\ang:\r/2-\m);
	\filldraw[fill=yellow!20,draw=black,very thin] (0:2*\alphr pt) arc (0:\ang:2*\alphr pt) -- (\ang:\alphr pt) arc (\ang:0:\alphr pt) -- cycle;
	\node at (\ang/2-1/2:2*\alphr-5 pt) {$\alpha_j$};
	
	% arcsin
	\filldraw[fill=blue!20,draw=black,very thin] (\ang:2*\alphr pt) arc (\ang:\angg:2*\alphr pt) -- (\angg:\alphr pt) arc (\angg:\ang:\alphr pt) -- cycle;
	\coordinate (C) at ({(\angg+\ang)/2}:3/2*\alphr pt);
	
	% axis of cone
	\draw[very thin] (0,0) coordinate (A1) -- (0:\r/2) coordinate (B1);
	\draw[decorate,decoration={brace,mirror,amplitude=10pt,raise=3pt,aspect=0.6},underext=1.5pt,very thin] (A1) -- (B1);
	\coordinate (mid1) at ($(A1)!0.6!(B1)$);
	\coordinate (ortho1) at ($(0,0)!1!90:($(B1)-(A1)$)$);
	\node[inner sep=0pt] at ($(mid1)!-14pt!($(mid1)+(ortho1)$)$) {$2^{j-1}$};
	
	% 2^m extension away from outer-most point
	\draw[dotted] (\ang:\r/2) -- ($(\ang:\r/2)+(\ang+180:\m)$);
	\draw[dotted] (\ang:\r/2) -- ($(\ang:\r/2)+(\ang+ 90:\m)$);
	\draw[very thin] (\ang:\r/2) coordinate (A2) -- (B2); % B2 already defined above
	\draw [decorate,decoration={brace,mirror,amplitude=10pt,raise=3pt},underext=1.5pt,very thin] (A2)--(B2);
	\coordinate (mid2) at ($(A2)!0.5!(B2)$);
	\coordinate (ortho2) at ($(0,0)!1!90:($(B2)-(A2)$)$);
	\node[fill=black!5,inner xsep=1pt,inner ysep=2pt] at ($(mid2)!-13pt!($(mid2)+(ortho2)$)$) {$2^m$};
	
	\draw[very thin] ($(B2)!5pt!(A2)$) arc (\angg+270:\angg+180:5pt);
	\draw[fill=black] ($(B2)+(\angg+225:2.8pt)$) circle (0.2pt);
	
%	% 2^m+1 extension away from axis
%	\draw[name path=arc] [very thin, red] (\r/2,0) node[shape=coordinate](A3){}--(\r/2,\m+1) node[shape=coordinate](B3){} arc (90:180:\m+1);
%	\draw [decorate,decoration={brace,mirror,amplitude=10pt,raise=1pt},red,very thin] ($(A3)!1 pt!(B3)$)-- ($(B3)!1 pt!(A3)$);
%	\coordinate (mid3) at ($(A3)!0.5!(B3)$);
%	\coordinate (ortho3) at ($(0,0)!1!90:($(B3)-(A3)$)$);
%	\node[red!90] at ($(mid3)!-22pt!($(mid3)+(ortho3)$)$) {\footnotesize $2^m\!+\!1$};
	
%	% angles - estimate first to be drawn over
%	\pgfmathsetmacro\angg{asin(2/\r*(\m+1))}
%	\path [name path=alpha est] (0,0)--(\angg:5);
%	\draw [name intersections={of=arc and alpha est, by=y}] [very thin,red] (0,0) -- (y);
%	\filldraw[fill=red!20,draw=red,very thin] (0,0)--(0:25pt) arc (0:\angg:25pt) -- cycle;
	};
\end{scope}
\begin{scope}[very thin,<->]
	\draw (-1,0) -- (6,0);
	\draw (0,-1.5) -- (0,4.85);
\end{scope}
\draw[blue!40] (C) -- (0.25,1.75) node[text=black,draw=blue!40,rounded corners,fill=white,above] {$\displaystyle\arcsin\left(\frac{2^m}{2^{j-1}}\right)\!$};
};
\end{tikzpicture}
\caption{The angle $\alpha_j^m$ can be computed explicitly\label{fig:alpha_m}}
\end{figure}
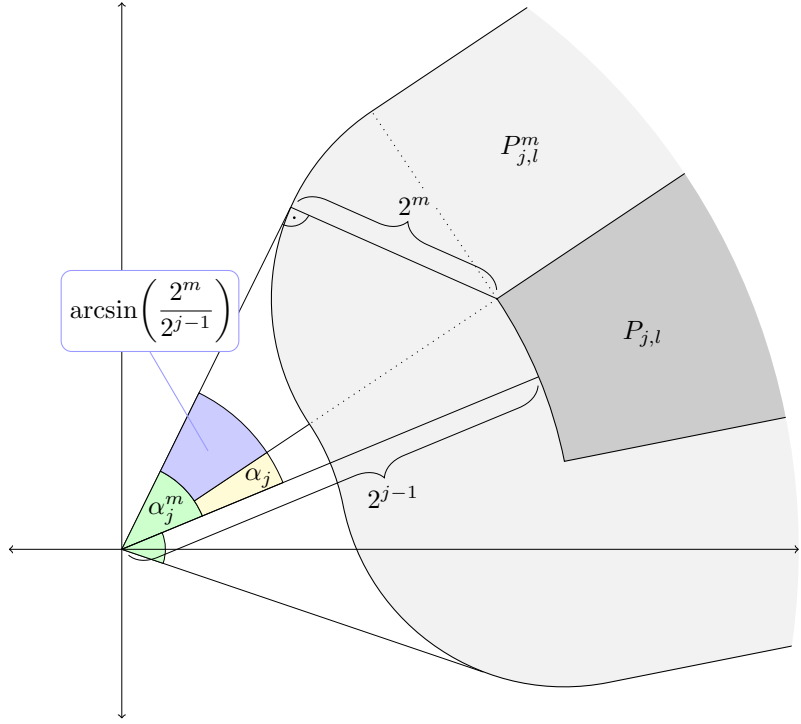
	
	By elementary geometric considerations (compare \autoref{fig:alpha_m}), we see that
	\begin{align}
		\alpha_j^m=\alpha_j+\arcsin \parens[\bigg]{\frac{2^m}{2^{j-1}}} \le \alpha_j + \pi 2^{m-j} = 2^{m-j}(2^{-m+1}+\pi) \le c_\omega 2^{m-j},
	\end{align}
	since $\arcsin{x}\le\frac\pi2 x$. The estimate can be made as long as $j\ge m+1$ and for all $m\ge0$, $c_\omega\le \pi+2$. This finishes the proof.
\end{proof}
\begin{lemma}\label{lem:intersection}
	Let $\jld$ as well as $\mmd$ be fixed and denote $m_>:= \max(m,m')$. If $j\ge m_> +3$, the intersections $P_\jl^m\cap P_\jld^{m'}$ can only be non-empty if $\jl$ satisfies
	\begin{align}\label{eq:app:ind_jl_impl}
		|j-j'|\le 2 \quad \text{and} \quad \dist_\bbSd (\vec s_\jl, \vec s_\jld) \le 5c_\omega 2^{m_>-j}.
	\end{align}
	For the complementary case $j\le m_>+2$, we are not able to restrict the contributing indicies and have to assume the worst-case scenario. Put together, we have the inclusion
	\begin{align*}%\label{eq:ind_jl_incl}
		\set*{(\jl)}{ P_\jl^m\cap P_\jld^{m'} \neq \emptyset} \subseteq \set*{(\jl)}{ j\ge m_>+3, \, \eqref{eq:app:ind_jl_impl} \text{ satisfied}} \cup \set*{(\jl)}{ j\le m_>+2, \, \ell \in \{0,\ldots,L_j\} }.
	\end{align*}
\end{lemma}
\begin{proof}
	For fixed $\jld$, the a necessary condition for the intersection $P_\jl^m\cap P_\jld^{m'}$ to be non-empty is
	\begin{align*}
		2^{j'+1}+2^{m'} > 2^{j-1} - 2^m \quad\text{and}\quad 2^{j'-1}-2^{m'} < 2^{j+1} + 2^m.
	\end{align*}
	For $\abs{j-j'}\le2$ this can always be satisfied for any $m\ge0$. For $\abs{j-j'}\ge3$, one can check that it's only possible for $m_> > j-3$. Said otherwise, if $j\ge m_>+3$, then all intersections must satisfy $\abs{j-j'}\le2$. This is illustrated in \autoref{fig:intersection}.
	
	In terms of the angle, we observe that the Minkowksi sums $P_\jl^m$ cannot anymore be easily described as the intersection of a spherical shell with a cone having its apex in the origin. However, it is still possible to find such a cone which contains $P_\jl^m$, having an enlarged opening angle $\alpha_j^m > \alpha_j$. By construction, we have that $\dist_\bbSd(\vec s_\jl, \vec s_\jld) \le \alpha_j^m + \alpha_{j'}^{m'}$ must be satisfied for the intersection to be non-empty. Naturally, these quantities can be estimated (see \autoref{prop:bounding_cylinder}),
	\vspace{-0.1cm}
	\begin{align*}
		\alpha_j^m \stackrel{\eqref{eq:app:est_alpha_j_m}}{\le} c_\omega 2^{m-j}, \qquad \alpha_{j'}^{m'} \stackrel{\eqref{eq:app:est_alpha_j_m}}{\le} c_\omega 2^{m'-j'},
	\end{align*}
	as long as $j\ge m+1$ and $j'\ge m'+1$, respectively. We can relate both quantities to $j'$, since by the above condition for $j$, we see that $\alpha_j^m + \alpha_{j'}^{m'} \le (4+1)c_\omega 2^{m_>-j'}$.
	
	Since $j'$ is arbitrary, we cannot make any restrictions on it -- however, we can use the fact that for $j\ge m_> +3$, the above consideration in terms of scale still hold, and that in this case $|j-j'|\le 2$ has to be satisfied. This gives us the desired condition $j'\ge m_> +1$, which allows us to use the estimates for the opening angles of the bounding cones. Collectively, these observations yield that $P_\jl^m \cap P_\jld^{m'} \neq \emptyset$ implies
	\begin{align*}
		|j-j'|\le 2 \quad \text{and} \quad \dist_\bbSd (\vec s_\jl, \vec s_\jld) \le 5c_\omega 2^{m_>} \alpha_{j'},
	\end{align*}
	as long as $j\ge m_>+3$. In other words,
	\begin{align*}
		\set*{(\jl)}{ P_\jl^m\cap P_\jld^{m'} \neq \emptyset } \subseteq \set*{(\jl)}{ j\ge m_>+3, \, \eqref{eq:app:ind_jl_impl} \text{ satisfied}} \cup \set*{(\jl)}{ j\le m_>+2, \, \ell \in \{0,\ldots,L_j\} },
	\end{align*}
	where we have assumed the worst-case scenario for $j< m_>+3$. %This estimate will be crucial for various estimates (see below), as well as the actual summation in Step 2.
\end{proof}

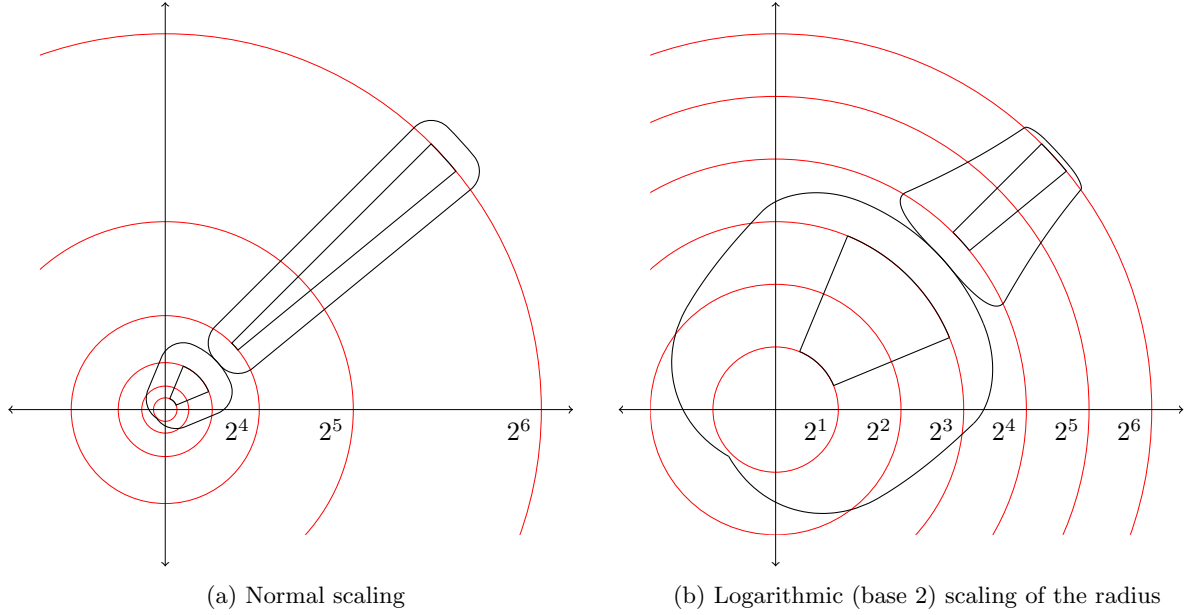
\begin{figure}
\subcaptionbox{Normal scaling\label{fig:intersection:normal}}{
\begin{tikzpicture}[scale=5/64] % actually this is 9/96*5/6, to match the cut-offs of the circles, as well as the extensions of the arrows beyond the largest circle
\begin{scope} % to contain clipping
\pgfmathsetmacro\ang{asin(64/3/(2^6+5))}
\clip (-\ang:2^6+5) arc (-\ang:90+\ang:2^6+5) -- (-64/3,-64/3) -- cycle;

\foreach \j in {1,...,3} % dyadic scales below and above the given radius
{
	\draw[red,very thin] (0,0) circle (2^\j);
}
\foreach \j in {4,...,6} % dyadic scales below and above the given radius
{
	\node[below left] at (2^\j,0) {$2^\j$};
	\draw[red,very thin] (0,0) circle (2^\j);
}
\end{scope}

\begin{scope}[very thin,<->]
	\draw (-80/3,0) -- (2^6+16/3,0);
	\draw (0,-80/3) -- (0,2^6+16/3);
\end{scope}

\foreach \r / \i [evaluate=\r as \ang using 90/\r] in {4/3,32/16} % enter radius and number of slice in the curly braces
{
\begin{scope}[rotate={(\i-1)*\ang}]
	\draw (-\ang:2*\r) arc (-\ang:\ang:2*\r) -- (\ang:\r/2) arc (\ang:-\ang:\r/2) -- cycle;
	\foreach \m in {4}
		{
		\pgfmathparse{round(1000*(\m-\r/2))}
		\ifnum 0<\pgfmathresult
			{
			\draw ($(-\ang:\r/2)+(-\ang-90:\m)$) arc (-\ang+270:-asin(\r/2*sin(\ang)/\m)+180:\m)
			arc (asin(\r/2*sin(\ang)/\m)+180:\ang+90:\m) -- ($(\ang:2*\r)+(\ang+90:\m)$)
			arc (\ang+90:\ang:\m) arc (\ang:-\ang:2*\r+\m) arc (-\ang:-\ang-90:\m) -- cycle;
			}
		\else
			{
			\draw ($(-\ang:\r/2)+(-\ang-90:\m)$) arc (-\ang-90:-\ang-180:\m)
			arc (-\ang:\ang:\r/2-\m) arc (\ang+180:\ang+90:\m)
			-- ($(\ang:2*\r)+(\ang+90:\m)$) arc (\ang+90:\ang:\m)
			arc (\ang:-\ang:2*\r+\m) arc (-\ang:-\ang-90:\m) -- cycle;
			}
		\fi
		};
\end{scope}
};
\end{tikzpicture}
}\hspace{-8pt}\hfill%
\subcaptionbox{Logarithmic (base 2) scaling of the radius\label{fig:intersection:log}}{
\begin{tikzpicture}
\begin{scope}[scale=5/6]
	\begin{scope} % to contain clipping
	\pgfmathsetmacro\ang{asin(2/6.5)}
	\clip (-\ang:6.5) arc (-\ang:90+\ang:6.5) -- (-2,-2) -- cycle;
	\foreach \j in {1,...,6} % dyadic scales below and above the given radius
	{
		\node[below left] at (\j,0) {$2^\j$};
		\draw[red,very thin] (0,0) circle (\j);
	}
	\end{scope}

\draw[very thin,<->] (-2.5,0) -- (6.5,0);
\draw[very thin,<->] (0,-2.5) -- (0,6.5);
\end{scope}

\foreach \r / \i [evaluate=\r as \ang using 90/\r] in {4/3,32/16} % radius and number of slice
{
\begin{scope}[rotate={(\i-1)*\ang}]
	\pgftransformnonlinear{\logtransformation}
	\draw (-\ang:\r/2) arc (-\ang:\ang:\r/2);
	\draw (\ang:\r/2) -- (\ang:2*\r);
	\draw (\ang:2*\r) arc (\ang:-\ang:2*\r);
	\draw (-\ang:2*\r) -- (-\ang:\r/2);
	\foreach \m in {4} % also works for \m=\r/2!
	{
	\pgfmathparse{round(1000*(\m-\r/2))}
	\ifnum 0<\pgfmathresult
		{
		\draw ($(-\ang:\r/2)+(-\ang+270:\m)$) arc (-\ang+270:-asin(\r/2*sin(\ang)/\m)+180:\m);
		\draw ($( \ang:\r/2)+( \ang+ 90:\m)$) arc ( \ang+ 90: asin(\r/2*sin(\ang)/\m)+180:\m);
		\draw ($( \ang:\r/2)+( \ang+ 90:\m)$) -- ($( \ang:2*\r)+( \ang+90:\m)$);
		\draw ($( \ang:2*\r)+( \ang+ 90:\m)$) arc ( \ang+ 90: \ang    :     \m);
		\draw ($( \ang:2*\r)+( \ang    :\m)$) arc ( \ang    :-\ang    :2*\r+\m);
		\draw ($(-\ang:2*\r)+(-\ang    :\m)$) arc (-\ang    :-\ang- 90:     \m);
		\draw ($(-\ang:2*\r)+(-\ang- 90:\m)$) -- ($(-\ang:\r/2)+(-\ang-90:\m)$);
		}
	\else
		{
		\draw ($(-\ang:\r/2)+(-\ang+180:\m)$) arc (-\ang+180:-\ang+270:     \m);
		\draw ($(-\ang:\r/2)+(-\ang+180:\m)$) arc (-\ang    : \ang    :\r/2-\m);
		\draw ($( \ang:\r/2)+( \ang+180:\m)$) arc ( \ang+180: \ang+ 90:     \m);
		\draw ($( \ang:\r/2)+( \ang+ 90:\m)$) -- ($( \ang:2*\r)+( \ang+90:\m)$);
		\draw ($( \ang:2*\r)+( \ang+ 90:\m)$) arc ( \ang+ 90: \ang    :     \m);
		\draw ($( \ang:2*\r)+( \ang    :\m)$) arc ( \ang    :-\ang    :2*\r+\m);
		\draw ($(-\ang:2*\r)+(-\ang    :\m)$) arc (-\ang    :-\ang- 90:     \m);
		\draw ($(-\ang:2*\r)+(-\ang- 90:\m)$) -- ($(-\ang:\r/2)+(-\ang-90:\m)$);
		}
	\fi
	};
\end{scope}
}
\end{tikzpicture}
}
\caption{Both subplots illustrate the argument of \autoref{lem:intersection}, that for $\abs{j-j'}=3$, $m=m'=m_>=j_<=j_>-3$ does not lead to an intersection}\label{fig:intersection}
\end{figure}

\section{A Suitable Choice of Window Functions}\label{sec:window}

To prove that \autoref{assump:psi_smooth} is satisfiable, we show a possible way of constructing the window functions such that the assumption holds. Independent of the specific form of the function, the key property we need to show the desired estimates, is that the function $G(x)$ below is constant for $x<1$.

\begin{lemma}\label{lem:suitable_window}
	Choose
	\begin{align*}%\label{eq:}
		t(x):= \frac{\exp\parens*{\frac{-1}{x^2} }}{\exp\parens*{\frac{-1}{x^2}} + \exp\parens*{\frac{-1}{(1-x)^2}}}.
	\end{align*}
	as a $\CC^\infty$ transition $t:[0,1] \to [0,1]$ with $t(0)=0$ and $t(1)=1$. Using this function, we construct
	\begin{align*}%\label{eq:}
		T(x) = \begin{cases}
			1, & 0\le x<1,\\
			t(1-x), & 1\le x \le 2, \\
			0, & 2<x.
		\end{cases}
	\end{align*}
	Setting
	\begin{align*}%\label{eq:}
		V^{(\jl)}(\vec \xi):= T\parens[\bigg]{\!2^j \arccos\parens[\bigg]{\frac{\vec \xi}{|\vec \xi|}\cdot \sjl }\!} = T\parens*{2^j \dist_\bbSd(\vec \xi,\sjl)}.
	\end{align*}
	and
	\begin{align*}%\label{eq:}
		W(x):= \begin{cases}
			\sin (\frac \pi2 t(x-1)) & 1\le x \le 2, \\
			\cos (\frac \pi2 t(\frac x2 -1)) & 2<x\le 4.
		\end{cases}
	\end{align*}
	Then \autoref{assump:psi_smooth} holds.
\end{lemma}

\begin{proof}
Since the argument of the $\arctan$ is independent of the length, $V^{(\jl)}$ is homogeneous of degree zero -- multiplicative constants in the argument do not change the result. Consequently, we may omit the normalising factor for this particular choice of $V^{(\jl)}$.

To estimate the derivatives of
\begin{align}%\label{eq:app:psi_trafo}
	\hat \psi_{(\jl)} (\vec \eta):= \hat\psi_\jl(U_\jl^{-\top}\vec \eta) = \frac{W\parens*{2^{-j} |D_{2^j}\vec \eta|} \, V^{(\jl)} \parens*{D_{2^j}\vec \eta}}{\sqrt{\Phi(\smash{U_\jl^{-\top}\vec \eta})\vphantom{\big()}}},
\end{align}
with arbitrary $\Rjl$ in $U_\jl=\Rjl^{-1}D_{2^{-j}}$, we have to estimate the derivatives of $W$, $V^{(\jl)}$ and $\Phi$ -- all three share the restriction that $\vec \eta$ must lie in the support of the numerator of $\hat \psi_\jl$, $U_\jl^\top P_\jl$, see \autoref{def:psi_jl}.

For $\Phi$, we see by \eqref{eq:Phi_neighbour}, that for $\vec \eta \in U_\jl^\top P_\jl$, the sum only consists of only a few terms with indices ``close to" $j$ and $\ell$,
\begin{align*}%\label{eq:}
	\Phi(U_\jl^{-\top} \vec \eta) = \sum_{\substack{ \forceheight{j'\in\bbN_0:} \\
		\forceheight[\vec b]{|j-j'|\le 1}}} \,\,
	\sum_{\substack{\forceheight{\ell' \in \{0,\ldots, L_{j'}\}:} \\
		\forceheight[\vec b]{\dist_\bbSd(\sjl, \sjld)\le 3\alpha_j}}}
	W\parens*{2^{-j'} |D_{2^j}\vec \eta|}^2 \, V^{(\jld)} \!\parens[\Big]{ \Rjl^{-1} D_{2^j} \vec \eta }^2.
\end{align*}
Again, we can rewrite the function $V^{(\jld)}$,
\begin{align*}%\label{eq:}
	V^{(\jld)} \!\parens*{ \Rjl^{-1} D_{2^j} \vec \eta }
	&= T\parens[\bigg]{\!2^j \arccos\parens[\bigg]{\Rjl^{-1}\frac{D_{2^j}\vec \eta}{|D_{2^j}\vec \eta|}\cdot \sjld }\!}
	= T\parens[\bigg]{\!2^j \arccos\parens[\bigg]{\Rjl^{-1}\frac{D_{2^j}\vec \eta}{|D_{2^j}\vec \eta|}\cdot \Rjld^{-1} \vec e_1 }\!}\\
	&= T\parens[\bigg]{\!2^j \arccos\parens[\bigg]{\Rjld\Rjl^{-1}\frac{D_{2^j}\vec \eta}{|D_{2^j}\vec \eta|}\cdot \vec e_1 }\!}
	= V^{(j'',0)} \!\parens*{ \Rjld\Rjl^{-1} D_{2^j} \vec \eta },
\end{align*}
where $\Rjld$ can be any rotation taking $\sjld$ to $\vec e_1$ -- we choose it such that the transformation $\Rjld \Rjl^{-1}$ is ``close to" the identity, see \autoref{lem:rot_lipschitz},
\begin{align*}%\label{eq:}
	\norm*{\Rjld \Rjl^{-1} - \bbI} \lesssim \dist_\bbSd(\sjl,\sjld)
\end{align*}
Therefore, instead of proving the estimates for $V^{(j,0)} \parens*{D_{2^j}\vec \eta}$ and $V^{(j',0)} \parens*{ \Rjld \Rjl^{-1} D_{2^j} \vec \eta }$ separately, we consider $V^{(j',0)}(\vec \zeta)$, where we let $M=(m_{i,k})_{i,k=1}^d$ be a general matrix and set $\vec \zeta := M\, D_{2^j} \vec \eta$ -- in other words,
\begin{align*}%\label{eq:}
	\zeta_i = m_{i,1} 2^j \eta_1 + m_{i,2}\eta_2 + \ldots + m_{i,d}\eta_d, \qquad i=1,\ldots,d
\end{align*}
for the $i$-th entry. Alternatively, instead of applying $\Rjld \Rjl^{-1}$ to $\frac{D_{2^j}\vec \eta}{|D_{2^j}\vec \eta|}$, we could have shifted both transformations to $\vec e_1$, which would eliminate some difficulties (no chain rule necessary, see below), but complicate other estimates, especially \eqref{eq:window_lower_bound_angle}.

As indicated above, the matrix $M$ will be the identity or very close to it. In particular, for $M:=\Rjld \Rjl^{-1}$, the equivalence of all norms on $\bbR^{d\times d}$ implies that we can estimate the individual matrix entries,
\begin{align}\label{eq:T_almost_id}
	|m_{i,k} -\delta_{i,k}| \le \norm*{\Rjld \Rjl^{-1} - \bbI}_{\mathrm{Fro}} \le c_\mathrm{Fro} \norm*{\Rjld \Rjl^{-1} - \bbI} \le c_\mathrm{Fro} 6 c_{\Rs} 2^{-j} =: c_\bbI 2^{-j},
\end{align}
which allows us to choose $j_0$ such that for all $j\ge j_0$,
\begin{align}\label{eq:choice_j_0}
	m_{i,i} \ge \frac{1}{2\cos(\alpha_1)}=0.925\ldots \quad \text{as well as} \quad 2^j \ge \sqrt{32 (d-1) c_\bbI}.
\end{align}
We can now investigate the support of (the derivatives of) $V^{(j',0)}(\vec \zeta)$. The main point here is that since $T(x)=1$ for $x<1$, the derivatives vanish there as well and thus we obtain a lower bound for the angle,
\begin{align} %split doesn't like so many alignment points
	\arccos \smash[b]{\parens[\bigg]{\frac{\zeta_1}{\sqrt{\zeta_1^2 +\ldots+ \zeta_d^2}} }} \ge 2^{-j'} &&\Longleftrightarrow && \zeta_1^2 &\le (\zeta_1^2 +\ldots+ \zeta_d^2)\cos(2^{-j'})^2 \notag \\
	&&\Longleftrightarrow && \zeta_1^2\sin(2^{-j'})^2 &\le  (\zeta_2^2 +\ldots+ \zeta_d^2)\cos(2^{-j'})^2 \label{eq:window_lower_bound_angle}\\
	&&\Longleftrightarrow && \zeta_1^2\tan(2^{-j'})^2 &\le \zeta_2^2 + \zeta_2^2 +\ldots+ \zeta_d^2. \notag
\end{align}
We want to derive a lower bound for the right hand side independently of $j$, and thus we return to
\begin{align*}%\label{eq:}
	|\zeta_1| &= |m_{1,1} 2^j \eta_1 + m_{1,2}\eta_2 + \ldots + m_{1,d}\eta_d| \ge m_{1,1} 2^j |\eta_1| - |m_{1,2}\eta_2 + \ldots + m_{1,d}\eta_d|,
	% \\
	%&\ge \frac{2^jm_{1,1}}{2\sqrt{2}} - 2^{-j+2}(d-1)c_\bbI \stackrel{\eqref{eq:choice_j_0}}{\ge} 2^j\parens[\bigg]{\frac{1}{4}-\frac{1}{8} } = 2^{j-3},
\end{align*}
where the second term can be estimated as follows
\begin{align}
	|m_{1,2}\eta_2 + \ldots + m_{1,d}\eta_d| \stackrel{\eqref{eq:T_almost_id}}{\le} (d-1)c_\bbI 2^{-j} \sqrt{\eta_2^2+\ldots+\eta_d^2} \stackrel{\eqref{eq:app:trafo_P_jl}}{\le} 2^{-j+2}(d-1)c_\bbI.
\end{align}
Together with \eqref{eq:choice_j_0} and \eqref{eq:app:trafo_P_jl}, we see that 
\begin{align}
	|\zeta_1|\ge 2^j\frac{\cos(\alpha_j)}{4\cos(\alpha_1)}-2^{-j+2}(d-1)c_\bbI \ge 2^j\parens[\bigg]{\frac{1}{4}-\frac{1}{8} } = 2^{j-3},
\end{align}
because $\cos(\alpha_j)\ge\cos(\alpha_1)$, yielding the desired lower estimate for $\vec \zeta$,
\begin{align*}%\label{eq:}
	\sqrt{\zeta_2^2 + \zeta_2^2 +\ldots+ \zeta_d^2} \ge \tan(2^{-j'}) \abs{\zeta_1} \ge 2^{-j'+j-3} \ge 2^{-4},
\end{align*}
since $\tan(x)\ge x$.

Now we are fully equipped to tackle the derivatives of
\begin{align*}%\label{eq:}
	V^{(\jld)} \!\parens*{ \Rjl^{-1} D_{2^j} \vec \eta } = V^{(j',0)}(T\,D_{2^j} \vec \eta) = V^{(j',0)}(\vec \zeta) = T\parens[\bigg]{\!2^{j'} \arccos\parens[\bigg]{\frac{\zeta_1}{|\vec \zeta|} }\!},
\end{align*}
where it suffices to control the derivatives of $2^{j'}\arccos\parens*{\frac{\zeta_1}{|\vec \zeta|} }$, since $G\in\CC^\infty$ and does not depend on $j$. We calculate for $2\le k \le d$
\begin{align*}%\label{eq:}
	\Dp{\eta_k} \arccos\parens[\bigg]{\frac{\zeta_1}{|\vec \zeta|} }
	&= \frac{-1}{\sqrt{1-\frac{\zeta_1^2}{\zeta_1^2+\ldots +\zeta_d^2}}} \parens[\bigg]{\frac{m_{1,k} \parens*{\zeta_1^2+\ldots +\zeta_d^2} - \zeta_1^2m_{1,k}}{\sqrt{\zeta_1^2+\ldots +\zeta_d^2}^3} - \frac{2\,\zeta_2 \, m_{2,k} + \ldots + 2\,\zeta_d \, m_{d,k}}{2\sqrt{\zeta_1^2+\ldots +\zeta_d^2}^3} } \\
	&= \frac{-\sqrt{\zeta_1^2+\ldots +\zeta_d^2}}{\sqrt{\zeta_2^2+\ldots +\zeta_d^2}} \parens[\bigg]{\frac{m_{1,k} \parens*{\zeta_2^2+\ldots +\zeta_d^2}}{\sqrt{\zeta_1^2+\ldots +\zeta_d^2}^3} - \frac{\zeta_2 \, m_{2,k} + \ldots + \zeta_d \, m_{d,k}}{\sqrt{\zeta_1^2+\ldots +\zeta_d^2}^3} } \\
	&= \frac{-m_{1,k} \sqrt{\zeta_2^2+\ldots +\zeta_d^2}}{\zeta_1^2+\ldots +\zeta_d^2} - \frac{1}{\parens*{\zeta_1^2+\ldots +\zeta_d^2} \sqrt{\zeta_2^2+\ldots +\zeta_d^2}} \parens*{\zeta_2 \, m_{2,k}+\ldots+ \zeta_d \, m_{d,k} }.
\end{align*}
The case $k=1$ is the same, except for an additional factor $2^j$ everywhere, due to the definition of $\vec \zeta$.
%	%
%	\begin{align*}%\label{eq:}
%		\Dp{\eta_1} \arccos\parens[\bigg]{\frac{\zeta_1}{|\vec \zeta|} }
%		= \frac{-m_{1,1} \, 2^j \! \sqrt{\zeta_2^2+\ldots +\zeta_d^2}}{\zeta_1^2+\ldots
%		+\zeta_d^2} - \frac{2^j}{\parens*{\zeta_1^2+\ldots +\zeta_d^2} \sqrt{\zeta_2^2+\ldots
%		+\zeta_d^2}} \parens*{\zeta_2 \, m_{2,k}+\ldots+ \zeta_d \, m_{d,k} }.
%	\end{align*}
%	%

By induction, we extend this to higher derivatives, using standard multi-index notation. We only care about the kinds of terms that will appear, but not their respective weights -- exact calculation is certainly possible, but only by not investigating the constants are we able to present a proof of acceptable length.

As might be expected from looking at the definition of $\vec \zeta$ (and the first derivative above), we will get a $m_{n,k}$-factor for each derivative after $\eta_k$, depending on which component $\zeta_n$ with $n\in\{1,\ldots,d\}$ is being derived. To compress this notation, we let $\vec a$ be a vector in $\{1,\ldots,d \}^{|\alpha|}$ (as it is necessary to choose one component for each derivative), and denote
\begin{align*}%\label{eq:}
	m_{\vec a, \alpha}:= \prod_{k=1}^d \prod_{n= \sum_{r=1}^{k-1} \alpha_r +1}^{\sum_{r=1}^k \alpha_r} m_{a_n,k}.
\end{align*}
Apart from this, operations of the multi-indices are to be interpreted componentwise. Lastly, since $|j-j'|\le 1$, we can replace $2^{j'}$ with $2^j$ up to a constant. This leads to the promised result of the induction,
\begin{align*}%\label{eq:}
	\Dpi{\alpha}{\vec\eta} \, 2^j \arccos \!\parens[\bigg]{\!\?\frac{\zeta_1}{|\vec \zeta|} \!\?}
	&= \!\!\!\! \sum_{ \substack{ \forceheight{\beta+\gamma+\delta=\alpha} \\
		\forceheight[\vec b]{\delta\le \beta + \gamma} \\
		\forceheight[\vec b]{|\beta|\ge 1}	}	} \,\,
	\sum_{ \substack{ \forceheight{\vec a' \in \{1,\ldots,d \}^{|\beta|+|\delta|}} \\
		\forceheight[\vec b]{\vec b' \in \{2,\ldots,d \}^{|\gamma|}}	}	}
	c_{\alpha,\beta,\vec a',\gamma,\vec b',\delta} \frac{2^{j(\alpha_1 +1)} \, m_{\vec a',\beta} \, m_{\smash{\vec b},\gamma} \, \vec\zeta^{\beta+\gamma-\delta}}{\parens*{\zeta_1^2+\ldots +\zeta_d^2}^{|\beta|} \sqrt{\zeta_2^2+\ldots +\zeta_d^2}^{\,2|\gamma|+1}} + \ldots \\
	&\phantom{=}\,\: \mathllap{\ldots +} \!\!\!\!
	\sum_{ \substack{ \forceheight{\beta+\gamma+\delta=\alpha} \\
		\forceheight[\vec b]{\delta\le \beta + \gamma} \\
		\forceheight[\vec b]{\mathclap{|\beta+\gamma+\delta|= |\alpha|-1}}	}	} \,\,
	\sum_{ \substack{ \forceheight{\vec a' \in \{1,\ldots,d\}^{\mathrlap{|\beta| +|\delta|+1} \phantom{|\beta|+|\delta|}}} \\
		\forceheight[\vec b]{\vec b' \in \{2,\ldots,d \}^{|\gamma|}}	}	}
	d_{\alpha,\beta,\vec a',\gamma,\vec b',\delta} \frac{2^{j(\alpha_1 +1)} \, m_{\vec a',\beta} \, m_{\smash{\vec b},\gamma} \, \vec\zeta^{\beta+\gamma-\delta}}{\parens*{\zeta_1^2+\ldots +\zeta_d^2}^{|\beta|+1} \sqrt{\zeta_2^2+\ldots +\zeta_d^2}^{\,2|\gamma|-1}}.
\end{align*}
Note, that in the second sum, the constant $d_{\alpha,\ldots}$ is zero unless the vector $\vec a'$ contains an entry which is $1$ -- in fact, all changes in the second sum boil down to requiring that at least once, the component of $\vec \zeta$ being derived was $\zeta_1$. In the first sum there is a somewhat complementary condition, namely that $c_{\alpha,\ldots}$ is zero unless at least one entry of $\vec a'$ does \emph{not} contain $1$.

The reward for this rather unwieldy formula is that we are now able to prove that it can be bounded independently of $j$. The goal is to balance the powers of $2^j$ in numerator and denominator -- the other factor in the denominator is unproblematic because we derived $\sqrt{\zeta_2^2+\ldots +\zeta_d^2} \ge 2^{-4}$ above.

Since we know $|\zeta_1|\gtrsim 2^j$, the exponent of $2^j$ in the denominator is $2|\beta|$ and $2|\beta|+2$, respectively. In the numerator of the first sum, powers of $2^j$ may appear in $\vec \zeta^{\beta+\gamma-\delta}$, and thus the exponent is at worst
\begin{align*}%\label{eq:}
	\alpha_1 +1 + |\beta+\gamma-\delta| \le 	|\alpha|+ |\beta|+|\gamma|-|\delta| +1 =  2|\beta| + 2|\gamma| + 1,
\end{align*}
using the decomposition of $\alpha$ and the fact that the ``absolute value" of a multi-index is linear. At this point we have to exploit the proximity of $M$ to the identity matrix, which implies that off-diagonal elements satisfy $m_{i,k} \le c_\bbI \, 2^{-j}$ as derived above. Since the powers of $2^j$ can only appear when deriving by $\eta_1$, and $m_{\smash{\vec b'},\gamma}\sim 2^{-j|\beta'|}$ can never yield an index $1$ in the first component (by the restriction on $\vec b'$), it follows that the term $2|\gamma|$ can be eliminated. Lastly, as we mentioned above, one component of $\vec a'$ must also not be equal to $1$, and thus we can eliminate the term ``$+1$" as well, and have balanced the powers of $2^j$ in the first term.

For the second sum we proceed similarly, the worst exponent in the numerator is
\begin{align*}%\label{eq:}
	\alpha_1 +1 + |\beta+\gamma-\delta| \le 	|\alpha|+ |\beta|+|\gamma|-|\delta| +1 = 2|\beta| + 2|\gamma| + 2,
\end{align*}
by the same argument as above, taking the different decomposition of $\alpha$ into account. The $2|\gamma|$-term is eliminated like before and this concludes the hardest part.

Wrapping everything up, we now see that $V^{(j',0)} \parens*{ M\, D_{2^j} \vec \eta }$ has bounded derivatives for $\eta \in U_\jl^\top$. The derivatives of the function $W\parens*{2^{-j'} |D_{2^j}\vec \eta|}$ are much easier to handle, because $W\in\CC^\infty$ is benign and the inner derivatives
\begin{align*}%\label{eq:}
	\Dpi{\alpha}{\vec\eta} 2^{-j'} |D_{2^j}\vec \eta| = \sum_{\substack{\beta+\gamma=\alpha \\ \gamma\le \beta}} \frac{2^{-j'+j \alpha_1} \vec \eta^{\beta-\gamma}}{\sqrt{2^{2j}\eta_1^2 + \eta_2^2 + \ldots + \eta_d^2}^{\,2|\beta|-1}}
\end{align*}
are (more than) balanced in terms of powers of $2^j$ since $\eta_1\ge \frac{1}{4}$ and $|j-j'|\le 1$. Together, this implies that the derivatives of $\Phi(U_\jl^{-\top}\vec \eta)$ are bounded independently of $j$ for $\vec \eta \in U_\jl^\top P_\jl$. For the numerator of $\psi_{(\jl)}$, we insert $M=\bbI$ and $j'=j$ into the above equations, which finally proves \autoref{assump:psi_smooth} for the presented choice of window functions.
\end{proof}

\section{Derivatives and Convolutions}\label{sec:prodrule}

In the proof of \autoref{th:sparse_stiff}, we need to explicitly calculate terms of the form $\Delta^n(fg)$. Although we are not aware of any reference, the formula below is almost certainly known already. However, it seems to be sufficiently non-standard to justify exploring it in a little bit more detail.
\begin{proposition}\label{prop:prodrule}
	For two sufficiently smooth functions $f,\,g\colon \bbR^d\to\bbC$, the product rule for the Laplacian reads as follows,
	\begin{align}\label{eq:prodrule}
		\Delta^n (fg)=\sum_{j+k_1+k_2=n} 2^j \binom{n}{j,k_1,k_2} \sum_{|\alpha|=j} \Dpi{\alpha}{\vec\eta} \parens*{\Delta^{k_1} f } \Dpi{\alpha}{\vec\eta} \parens*{\Delta^{k_2} g },
	\end{align}
	where $\displaystyle \binom{n}{j,k_1,k_2} = \frac{n!}{j!\,k_1!\,k_2!}$ is the \emph{trinomial} coefficient and the differential operator $\Dpi{\alpha}{\vec\eta}$ in standard multi-index notation operates on the different coordinates of $\vec \eta$,
	\begin{align*}%\label{eq:}
		\Dpi{\alpha}{\vec\eta} = \frac{\partial^{\alpha_1+\ldots+\alpha_d}}{\partial \eta_1^{\alpha_1} \ldots \partial \eta_d^{\alpha_d}}.
	\end{align*}
\end{proposition}
\begin{remark}
	If we interpret $\nabla$ as an operator taking a tensor of order $i$ to a tensor of order $i+1$, then \eqref{eq:prodrule} can be written even more compactly,
	\begin{align*}
%		\bracket*{\Delta^n (fg)}(\vec \eta) = \sum_{j+k_1+k_2=n} 2^j \binom{n}{j,k_1,k_2} \inpr[\Big]{\bracket*{\nabla^j \parens*{\Delta^{k_1}f }}(\vec \eta), \bracket*{\nabla^j \parens*{\Delta^{k_2} g }}(\vec \eta)}_\mathrm{Fro},
		\Delta^n (fg) = \sum_{j+k_1+k_2=n} 2^j \binom{n}{j,k_1,k_2} \inpr[\Big]{\nabla^j \parens*{\Delta^{k_1}f }, \nabla^j \parens*{\Delta^{k_2} g }}_\mathrm{Fro},
	\end{align*}
	where $\inpr{A,B}_\mathrm{Fro}$ is the sum over all componentwise products of the two tensors -- i.e. the tensor analogue of the Frobenius inner product for matrices (which induces the Frobenius norm). We note that it's possible to generalise this formula to a product $\prod_{i=1}^q f_i$ with $q\in\bbN$ as well.
	
	As an unrelated observation, setting $d=1$ and comparing coefficients with the standard product rule yields a curious relation between bi- and trinomial coefficients,
	\begin{align*}
		\binom{2n}{\ell} = \sum_{k=0}^{\floor{\frac{\ell}2 }} 2^{\ell-2k} \binom{n}{\ell-2k,k,n+k-\ell}  = \sum_{k=\max(\ell-n,0)}^{\floor{\frac{\ell}2 }} \frac{2^{\ell-2k} n!}{(\ell-2k)!\,k!\,(n+k-\ell)!},
	\end{align*}
	and in particular,
	\begin{align*}
		(2n)!=\sum_{k=0}^{\floor{ \frac{n}2 }} \frac{2^{n-2k} (n!)^3}{(n-2k)!\,(k!)^2}.
	\end{align*}
\end{remark}
\begin{proof}
%	[Eine M\"oglichkeit, aber der Beweis w\"ar deutlich l\"anger...: This is a special case of
%	\autoref{thm:prodrule_general} for $q=2$.]
	The proof is a simple induction -- the case $n=0$ is trivial.
	For $n\to n+1$ we consider
	\begin{align*}
		\Delta^{n+1} (fg) &= \Delta \parens[\bigg]{\sum_{j+k_1+k_2=n}  2^j \binom{n}{j,k_1,k_2} \inpr[\Big]{\nabla^j \parens*{\Delta^{k_1}f },
		\nabla^j \parens*{\Delta^{k_2} g }}_\mathrm{Fro} } \\
		&=\sum_{\csetwidth{j+k_1+k_2=n}{3.9em}}  2^j \binom{n}{j,k_1,k_2} \bigg( \inpr[\Big]{\nabla^j \parens*{\Delta^{k_1+1}f }, \nabla^j \parens*{\Delta^{k_2} g }}_\mathrm{Fro} \ldots +\\
		&\phantomrel \ldots + 2\,\inpr[\Big]{\nabla^{j+1} \parens*{\Delta^{k_1}f }, \nabla^{j+1} \parens*{\Delta^{k_2} g }}_\mathrm{Fro} + \inpr[\Big]{\nabla^j \parens*{\Delta^{k_1}f }, \nabla^j \parens*{\Delta^{k_2+1} g }}_\mathrm{Fro} \bigg) \\
%%% Explicit index shifts
%		&= \sum_{j=0}^{\ccopywidth{n}{\scriptstyle n+1}} \sum_{k_1=1}^{n+1-j} \sum_{k_2=0}^{n+1-j-k_1} 2^j \binom{n}{j,k_1-1,k_2} \inpr[\Big]{\nabla^j \parens*{\Delta^{k_1}f }, \nabla^j \parens*{\Delta^{k_2} g }}_\mathrm{Fro} + \ldots\\
%		&\phantomrel\ldots+ \sum_{j=1}^{n+1} \sum_{k_1=0}^{n+1-j} \sum_{k_2=0}^{n+1-j-k_1} 2^j \binom{n}{j-1,k_1,k_2} \inpr[\Big]{\nabla^j \parens*{\Delta^{k_1}f }, \nabla^j \parens*{\Delta^{k_2} g }}_\mathrm{Fro} + \ldots\\
%		&\phantomrel\ldots+\sum_{j=0}^{\ccopywidth{n}{\scriptstyle n+1}} \sum_{k_1=0}^{\ccopywidth{n-j}{\scriptstyle n+1-j}} \sum_{k_2=1}^{n+1-j-k_1} 2^j \binom{n}{j,k_1,k_2-1} \inpr[\Big]{\nabla^j \parens*{\Delta^{k_1}f }, \nabla^j \parens*{\Delta^{k_2} g }}_\mathrm{Fro}\\
%		&= \sum_{j=0}^{n{\color{red} + 1}} \sum_{k_1={\color{red}0}}^{n+1-j} \sum_{k_2=0}^{n+1-j-k_1} 2^j \binom{n}{j,k_1-1,k_2} \inpr[\Big]{\nabla^j \parens*{\Delta^{k_1}f }, \nabla^j \parens*{\Delta^{k_2} g }}_\mathrm{Fro} + \ldots\\
%		&\phantomrel\ldots+ \sum_{j={\color{red}0}}^{n+1} \sum_{k_1=0}^{n+1-j} \sum_{k_2=0}^{n+1-j-k_1} 2^j \binom{n}{j-1,k_1,k_2} \inpr[\Big]{\nabla^j \parens*{\Delta^{k_1}f }, \nabla^j \parens*{\Delta^{k_2} g }}_\mathrm{Fro} + \ldots\\
%		&\phantomrel\ldots+\sum_{j=0}^{n{\color{red} + 1}} \sum_{k_1=0}^{n{\color{red} + 1}-j} \sum_{k_2={\color{red}0}}^{n+1-j-k_1} 2^j \binom{n}{j,k_1,k_2-1} \inpr[\Big]{\nabla^j \parens*{\Delta^{k_1}f }, \nabla^j \parens*{\Delta^{k_2} g }}_\mathrm{Fro}\\
		&= \sum_{\csetwidth{j+k_1+k_2=n+1}{3.9em}}  2^j \parens[\bigg]{\binom{n}{j,k_1-1,k_2}+\binom{n}{j-1,k_1,k_2}+\binom{n}{j,k_1,k_2-1} }  \inpr[\Big]{\nabla^j \parens*{\Delta^{k_1}f },
		\nabla^j \parens*{\Delta^{k_2} g }}_\mathrm{Fro},
	\end{align*}
	where for the last equation, we performed an index shift for each of the three summands independently (in  $k_1,\,j,\,k_2$, respectively) and were able to extend the range of the indices because all additional terms have a trinomial coefficient of zero (either one entry is negative or the sum $j+k_1+k_2$ is greater than $n$). At this point we need an analogous result to a well-known property of Pascal's triangle, namely
	\begin{align*}
		\binom{n}{j-1,k_1,k_2}+\binom{n}{j,k_1-1,k_2}+\binom{n}{j,k_1,k_2-1} = \frac{n!(j+k_1+k_2)}{j!\,k_1!\,k_2!}=\frac{(n+1)!}{j!\,k_1!\,k_2!} = \binom{n+1}{j,k_1,k_2},
	\end{align*}
	since $j+k_1+k_2=n+1$. This finishes the proof.
\end{proof}
An immediate corollary to \autoref{prop:prodrule} is the following.
\begin{corollary}\label{cor:est_prodrule}
	Under the same assumptions as in \autoref{prop:prodrule}, we have
	\begin{align*}
		\abs*{\bracket*{\Delta^n \parens*{fg}}(\vec \eta)} \le (4d)^n |f(\vec \eta)|_{\CC^{2n}} \, |g(\vec \eta)|_{\CC^{2n}} \le (4d)^n \norm{f}_{\CC^{2n}} \norm{g}_{\CC^{2n}},
	\end{align*}
	where $|f(\vec \eta)|_{\CC^{2n}}=\max_{0\le r\le 2n} |f^{(r)}(\vec \eta)|$ is the maximum of all derivatives up to order $2n$ of $f$ at $\vec \eta$.
\end{corollary}
\begin{proof}
	The sum $\sum_{|\alpha|=j}$ consists of $d^j$ terms. This can be seen since the sum can also be interpreted as selecting $j$ (possibly redundant) coordinates from $\{1,\ldots,d\}$ -- a vector in $\{1,\ldots,d\}^j$. Alternatively, one can use multinomials for selecting multiplicities $\alpha_1,\ldots,\alpha_d$ which sum to $j$.
	
	Similarly, the operator $\Delta^k$ consists of $\binom{d+k-1}{k}\le d^k$ terms (which corresponds to choosing $k$ out of $d$ elements with repetitions). With this in mind, taking the absolute value of \eqref{eq:prodrule} leads to
	\begin{align*}
		\abs*{\Delta^n \parens*{f(\vec \eta)g(\vec \eta)} }
		&= \abs[\bigg]{\sum_{j+k_1+k_2=n} 2^j \binom{n}{j,k_1,k_2} \sum_{|\alpha|=j} \Dpi{\alpha}{\vec\eta} \parens*{\Delta^{k_1} f (\vec \eta) } \Dpi{\alpha}{\vec\eta} \parens*{\Delta^{k_2} g(\vec \eta) }} \\
		&\le \sum_{j+k_1+k_2=n} 2^{j} \binom{n}{j,k_1,k_2} d^{j+k_1+k_2} |f(\vec \eta)|_{\CC^{j+2 k_1}} \, |g(\vec \eta)|_{\CC^{j+2 k_2}} \\
		&\le (4d)^n |f(\vec \eta)|_{\CC^{2n}} \, |g(\vec \eta)|_{\CC^{2n}} \le (4d)^n \norm{f}_{\CC^{2n}} \norm{g}_{\CC^{2n}},
	\end{align*}
	where we used the identity $\sum_{j+k_1+k_2=n} (2d)^{j} d^{k_1} d^{k_2} \binom{n}{j,k_1,k_2} = (4d)^n$, which is immediate by setting $(a,b,c)$ to $(2d,d,d)$ in the trinomial expansion 
	\begin{align*}%\label{eq:}
		(a+b+c)^n = \sum_{j+k_1+k_2=n}  \binom{n}{j,k_1,k_2} a^{j}b^{k_1}c^{k_2}.
	\end{align*}
\end{proof}
Finally, we need the following auxiliary result for differentiating the pullbacks of a convolution.
\begin{lemma}\label{lem:deriv_conv}
	For any invertible linear transformatiton $U:\bbR^d \to \bbR^d$, the derivatives of the pullback of the convolution can be estimated as follows,
	\begin{align*}%\label{eq:}
		\abs[\Big]{ \Dpi{\alpha}{\vec\eta} \parens*{ (f*g)(U\vec \eta) } } \le \norm[\Big]{\Dpi{\alpha}{\vec\eta} \parens*{f(U \cdot)}}_\infty \parens*{\ind_{\supp f}(\cdot) * |g(\cdot)| }(U \eta).
	\end{align*}
\end{lemma}
\begin{proof}
	We begin by computing
	\begin{align*}%\label{eq:}
		\parens*{f*g}(U\vec \eta) = \int f(\vec \zeta) g(U\vec \eta - \vec \zeta) \d\vec \zeta = |\det U| \int f(U\vec \xi) g\parens*{U(\vec \eta - \vec \xi)} \d\vec \xi = |\det U| \parens*{f(U\cdot)*g(U\cdot)}(\vec \eta).
	\end{align*}
	We apply all derivatives of the convolution to the function $f$, thus
	\begin{align*}%\label{eq:}
		\Dpi{\alpha}{\vec\eta} \parens*{ (f*g)(U\vec \eta) } = |\det U| \parens[\Big]{\Dpi{\alpha}{\vec\eta} \parens*{f(U\cdot)}*g(U\cdot)}(\vec \eta) = |\det U| \int \parens[\bigg]{\!\Dpi{\alpha}[f(U\cdot)]{\vec\eta}\!}(\vec \xi) \, g\parens*{U(\vec \eta - \vec \xi)} \d\vec \xi.
	\end{align*}
	Estimating the derivatives of $f$ by its maximal value times its support, we arrive at
	\begin{align*}%\label{eq:}
		\abs[\Big]{\Dpi{\alpha}{\vec\eta} \parens*{ (f*g)(U\vec \eta) } }
		&\le |\det U| \int \norm[\Big]{\Dpi{\alpha}{\vec\eta} \parens*{f(U \cdot)}}_\infty \ind_{\supp f}(U\vec \xi) \abs*{g\parens*{U(\vec \eta - \vec \xi)}} \d\vec \xi\\
		&= \norm[\Big]{\Dpi{\alpha}{\vec\eta} \parens*{f(U \cdot)}}_\infty \int \ind_{\supp f}(\vec \zeta) \abs*{g\parens*{U\vec \eta - \vec \zeta }} \d\vec \zeta \\
		&= \norm[\Big]{\Dpi{\alpha}{\vec\eta} \parens*{f(U \cdot)}}_\infty \parens*{\ind_{\supp f}(\cdot) * |g(\cdot)| }(U \eta),
	\end{align*}
	which finishes the proof.
\end{proof}

\end{appendices}

\pagestyle{plain}

\bibliographystyle{alpha} %style ``alphanum'' fuer nummerierung statt jahreszahlen
\bibliography{ridgelets_paper}

\begin{thebibliography}{KLLW05}

\bibitem[BN07]{borup}
L.~Borup and M.~Nielsen.
\newblock Frame decomposition of decomposition spaces.
\newblock {\em J. Fourier Anal. Appl.}, 13(1):39--70, 2007.

\bibitem[Can98]{Can98}
E.~Cand\`{e}s.
\newblock Ridgelets: Theory and applications.
\newblock Ph{D} thesis, Stanford University, 1998.

\bibitem[Can01]{mutilated}
E.~Cand\`{e}s.
\newblock Ridgelets and the representation of mutilated {S}obolev functions.
\newblock {\em SIAM J. Math. Anal.}, 33(2):347--368, 2001.

\bibitem[CD05a]{Candes2005b}
E.~Cand\`{e}s and D.L. Donoho.
\newblock Continuous curvelet transform: {I}. {R}esolution of the {W}avefront
  {S}et.
\newblock {\em Appl. Comput. Harmon. Anal.}, 19(2):198--222, 2005.

\bibitem[CD05b]{Candes2005a}
E.~Cand\`{e}s and D.L. Donoho.
\newblock Continuous curvelet transform: {II}. {D}iscretization and frames.
\newblock {\em Appl. Comput. Harmon. Anal.}, 19(2):198--222, 2005.

\bibitem[CDD01]{cdd}
A.~Cohen, W.~Dahmen, and R.~DeVore.
\newblock Adaptive wavelet methods for elliptic operator equations: convergence
  rates.
\newblock {\em Math. Comp.}, 70(233):27--75, 2001.

\bibitem[CDDY06]{CDDY06}
E.~Cand\`{e}s, L.~Demanet, D.L. Donoho, and L.~Ying.
\newblock Fast discrete curvelet transforms.
\newblock {\em Mult. Model. Simul.}, 5:861--899, 2006.

\bibitem[Dau92]{wavelets}
I.~Daubechies.
\newblock {\em Ten lectures on wavelets}, volume~61 of {\em CBMS-NSF Regional
  Conference Series in Applied Mathematics}.
\newblock Society for Industrial and Applied Mathematics (SIAM), Philadelphia,
  PA, 1992.

\bibitem[DeV98]{DeVore1998}
R.~DeVore.
\newblock Nonlinear approximation.
\newblock {\em Acta Numerica}, pages 51--150, 1998.

\bibitem[DFR07]{Dahlke2007}
S.~Dahlke, M.~Fornasier, and T.~Raasch.
\newblock Adaptive frame methods for elliptic operator equations.
\newblock {\em Advances in Computational Mathematics}, 27(1):27--63, 2007.

\bibitem[DV05]{DV05}
M.N. Do and M.~Vetterli.
\newblock The contourlet transform: an efficient directional multiresolution
  image representation.
\newblock {\em IEEE Trans. Image Proc.}, 14:2091--2106, 2005.

\bibitem[EG04]{ern}
A.~Ern and J.-L. Guermond.
\newblock {\em Theory and practice of finite elements}, volume 159 of {\em
  Applied Mathematical Sciences}.
\newblock Springer-Verlag, New York, 2004.

\bibitem[EGO14]{FFT-paper}
S.~Etter, P.~Grohs, and A.~Obermeier.
\newblock {FFRT} -- a fast finite fourier transform for radiative transport.
\newblock {\em Submitted}, 2014.
\newblock Preprint available as a SAM Report (2014), ETH Z\"urich,
  \url{http://www.sam.math.ethz.ch/sam_reports/index.php?id=2014-11}.

\bibitem[GK14]{par-mol}
P.~Grohs and G.~Kutyniok.
\newblock Parabolic molecules.
\newblock {\em Foundations of Computational Mathematics}, 14(2):299--337, 2014.

\bibitem[GKKS14]{alpha-mol}
P.~Grohs, S.~Keiper, G.~Kutyniok, and M.~Sch\"afer.
\newblock $\alpha$-molecules.
\newblock {\em Submitted}, 2014.
\newblock Preprint available as a SAM Report (2014), ETH Z\"urich,
  \url{http://www.sam.math.ethz.ch/sam_reports/index.php?id=2014-16}.

\bibitem[GO14]{go_approx}
P.~Grohs and A.~Obermeier.
\newblock On the approximation of functions with line singularities by
  ridgelets.
\newblock {\em In preparation}, 2014.

\bibitem[Gro11]{grohs1}
P.~Grohs.
\newblock Ridgelet-type frame decompositions for sobolev spaces related to
  linear transport.
\newblock {\em J. Fourier Anal. Appl.}, 2011.

\bibitem[GS11]{grohs_schwab}
P.~Grohs and Ch. Schwab.
\newblock Sparse twisted tensor frame discretization of parametric transport
  operators.
\newblock 2011.
\newblock Preprint available as a SAM Report (2011), ETH Z\"urich,
  \url{http://www.sam.math.ethz.ch/sam_reports/index.php?id=2011-41}.

\bibitem[HS78]{halmos}
P.R. Halmos and V.S. Sunder.
\newblock {\em Bounded integral operators on {$L^{2}$} spaces}, volume~96 of
  {\em Ergebnisse der Mathematik und ihrer Grenzgebiete [Results in Mathematics
  and Related Areas]}.
\newblock Springer-Verlag, Berlin-New York, 1978.

\bibitem[KL12]{KL12i}
G.~Kutyniok and D.~Labate.
\newblock {\em Shearlets: Multiscale Analysis for Multivariate Data}, chapter
  Introduction to Shearlets, pages 1--38.
\newblock Birkh\"auser, 2012.

\bibitem[KLLW05]{KuLaLiWe}
G.~Kutyniok, D.~Labate, W.-Q Lim, and G.~Weiss.
\newblock Sparse multidimensional representation using shearlets.
\newblock {\em Wavelets XI(San Diego, CA), SPIE Proc.}, 5914:254--262, 2005.

\bibitem[Li11]{li}
S.~Li.
\newblock Concise formulas for the area and volume of a hyperspherical cap.
\newblock {\em Asian J. Math. Stat.}, 4(1):66--70, 2011.

\bibitem[Mod13]{modest2013radiative}
M.F. Modest.
\newblock {\em Radiative heat transfer}.
\newblock Academic press, 2013.

\bibitem[Ste04]{Stevenson2004}
R.~Stevenson.
\newblock Adaptive solution of operator equations using wavelet frames.
\newblock {\em SIAM Journal on Numerical Analysis}, pages 1074--1100, 2004.

\end{thebibliography}
\addcontentsline{toc}{section}{Bibliography}

\end{document}